\documentclass{amsart}

\usepackage{amsthm}
\usepackage{amsmath} 
\usepackage{color}

\usepackage{thmtools}
\usepackage{thm-restate}

\usepackage{hyperref}

\usepackage{cleveref}

\usepackage{bbm} 

\usepackage{graphicx, epsfig}
\usepackage{epstopdf}
\usepackage{amsmath, amssymb, latexsym, euscript, amscd}
\usepackage{url}
\usepackage[all]{xy}
\usepackage{psfrag}
\usepackage{mathrsfs}
\usepackage{ caption, wrapfig, multirow, tabularx, mathrsfs,verbatim}
\usepackage{subcaption}
\usepackage{tikz,tikz-cd}

\newcounter{notes}%

\usepackage[marginpar=1cm]{geometry}
\definecolor{darkgreen}{rgb}{0.0, 0.5, 0.0}

\newtheorem{theorem}{Theorem}[section]
\newtheorem{lemma}[theorem]{Lemma}
\newtheorem{corollary}[theorem]{Corollary} 
\newtheorem{definition}[theorem]{Definition} 
\newtheorem{proposition}[theorem]{Proposition}
\newtheorem{remark}[theorem]{Remark}

\def\gap{\vspace{.3cm}\noindent}

\def\smallskip{\vspace{.15cm}}
\def\medskip{\vspace{.3cm}}
\def\text{\mbox}
\def\RR{{\mathbb R}}

\def\CC{{\mathbb C}}
\def\EE{{\mathbb E}}
\def\ZZ{{\mathbb Z}}
\def\PP{{\mathbb P}}
\def\HH{{\mathbb H}}

\def\AA{{\mathbb A}}

\def\R{{\mathbb R}}
\def\RPn{\operatorname{\mathbb{R}P}^n}
\def\RP2{\operatorname{\mathbb{R}P}^2}
\def\RP3{\operatorname{\mathbb{R}P}^3}
\def\RP{\operatorname{\mathbb{R}P}}

\def\Jcal{{\mathcal J}}

\def\Im{\operatorname{Im}}
\def\interior{\operatorname{int}}
\def\SL{\operatorname{SL}}

\def\SO{\operatorname{SO}}
\def\PO{\operatorname{PO}}
\def\PGL{\operatorname{PGL}}
\def\GL{\operatorname{GL}}

\def\Aff{\operatorname{Aff}}
\def\Hom{\operatorname{Hom}}
\def\Mono{\operatorname{Mono}}
\def\Rep{\operatorname{Rep}}

\def\cl{\operatorname{cl}}
\def\dev{\operatorname{dev}}

\def\Aff{\operatorname{Aff}}

\def\hol{\operatorname{hol}}
\def\Isom{\operatorname{Isom}}

\def\Diag{\operatorname{Diag}}

\def\C2{\operatorname{C^2}}
\def\halfgap{\vspace{.05in}}

\def\Modsp{\mathcal C}
\def\Bcal{\mathcal B}
\def\Rcal{\mathcal R}
\def\Pcal{\mathcal P}

\def\Hcal{\mathcal H}
\def\Ccal{\mathcal C}

\def\Wcal{\mathcal W}
\def\Tcal{\mathcal T}
\def\Sym{\operatorname{S}}
\def\Tr{\operatorname{Tr}}
\def\Image{\operatorname{Im}}
\def\Id{\operatorname{I}}

\def\harm{\mathscr H}
\def\rad{\mathscr R}

\def\editA{\color{black}}
\def\editB{\color{black}}
\def\editC{\color{black}}
\def\editD{\color{black}}
\def\editE{\color{black}}

\def\aff{\operatorname{\mathfrak{aff}}}
\def\cinvt{\eta}
\def\SLpm{\SL^{\pm}}
\def\weight{\xi}

  \newcommand{ \kap}{\kappa}
  
  \def\Kappa{\mathcal{K}}

\def\Pos{\operatorname{\mathcal P}}
\def\O{\operatorname{O}}

\def\bdy{\partial}

\def\Modsp{{\mathcal T}}

\def\dxi{\weight}
\def\trace{\operatorname{trace}}

\def\ppsi{\psi}
\def\ur{{\bf u}}
\def\rank{{\bf r}}
\def\type{{\bf t}}

  \def\liealgteich{\mathfrak{c}}

  \def\Aut{\operatorname{Aut}}
  \def\Sim{\operatorname{Sim}}

\newcommand{\Abs}[1]{\left|\left|#1\right|\right|}
\newcommand{\abs}[1]{\left|#1\right|}

\definecolor{back}{RGB}{255,255,255}
\definecolor{fore}{RGB}{0,0,0}
\definecolor{title}{RGB}{255,0,90}

\definecolor{green}{rgb}{0.0, 0.5, 0.0}
\definecolor{purple}{rgb}{0.5, 0.0, 0.5}
\definecolor{bluegreen}{rgb}{0.0,0.5, 0.5}
\definecolor{orange}{rgb}{1,0.5, 0.1}
\definecolor{redgreen}{rgb}{0.5, 0.5, 0.0}

\def\green{\color{green}}

\def\green{\color{green}}

\def\g2{{\green 2}}

  \def\SP{\operatorname{SP}}

\newcommand{\bv}{\left[\begin{array}{c}}
\newcommand{\ev}{\end{array}\right]}
\newcommand{\bbmat}{\begin{bmatrix}} 
\newcommand{\ebmat}{\end{bmatrix}}
\newcommand{\bmat}{\begin{matrix}} 
\newcommand{\emat}{\end{matrix}}
\newcommand{\bpmat}{\begin{pmatrix}} 
\newcommand{\epmat}{\end{pmatrix}}

\begin{document}
\title[Generalized Cusps]{The Moduli Space of Marked Generalized Cusps in Real Projective Manifolds}
\date{\today}
\author{Samuel A. Ballas, Daryl Cooper,  and Arielle Leitner}
\begin{abstract}
ln this paper, a generalized cusp is a properly convex manifold with strictly convex boundary that
is diffeomorphic to $M\times [0,\infty)$ where $M$ is a closed Euclidean manifold. 
These are classified in \cite{BCL}. The marked
 moduli space  is {\editA homeomorphic to} a subspace of the space of conjugacy classes of representations {\editA of $\pi_1M$}.
 {\editA It has one description as a generalization of a {\editD trace-variety}, and
 another description involving weight data that is similar to that used to describe semi-simple Lie groups.}
It is also a bundle over the space
of Euclidean similarity (conformally flat) structures on $M$, and the fiber is a closed cone in the space  of cubic differentials.
For $3$-dimensional orientable generalized  cusps, the fiber is homeomorphic to a cone on a solid torus.\end{abstract}

\maketitle

A \emph{generalized cusp} is a properly-convex real-projective manifold, $C$, 
such that $C$ is diffeomorphic to $[0, 1) \times \partial C$,
and $\pi_1C$ is virtually-nilpotent, and
 $\partial C$ contains no line segment. 
 
From now on, in this paper, we use
 the term {\em generalized cusp} in the narrow sense that $\bdy C$ is also compact.
It was shown in \cite[(0.7)]{BCL} this implies that $\pi_1C$ is virtually abelian,
and  that $C$ has a natural affine structure that is a stiffening of the projective structure.

Let $\AA^n$ denote affine space, and $\Aff(n)$ the affine group.
Then  $C=\Omega/\rho(\pi_1C)$ where
$\Omega\subset\AA^n $ is a non-compact, convex,  closed set, bounded by a
strictly-convex hypersurface that covers $\bdy C$, and $\rho:\pi_1 C\to\Aff(n)$ is the holonomy.

{\editD The moduli space of marked generalized cusps turns out to be a beautiful object with interesting structure,
that admits several different descriptions.
We concentrate on the case that $\bdy C\cong\RR^{n-1}/\ZZ^{n-1}$, then the holonomy $\rho$ extends over $V\cong\RR^{n-1}$. In this case the moduli space $\Tcal_n$,  consists  of  all conjugacy classes of monomorphisms of $\RR^{n-1}$
into $\Aff(n)$ such that the orbit of a generic point is
a properly-embedded,
strictly-convex hypersurface.

  Then $\Tcal_n\cong \Pcal\times F$ where $\Pcal$ is
the space of unimodular, positive-definite quadratic forms on $V$,
 and $F$ is the space  consisting of all  unordered $n$-tuples of pairwise-orthogonal vectors (allowing $0$) in $V\times\RR$,
that all have the same, non-negative, $\RR$-coordinate.

 It follows that one may view a generalized cusp as a Euclidean manifold  with extra structure
obtained by a deformation of a {\em standard} cusp i.e. equivalent to one in a hyperbolic manifold. The bundle structure on the moduli
space admits several descriptions.

A generalized cusp is determined up to equivalence by the {\em complete invariant} $(\chi,[\beta])$ comprising the character
 $\chi:V\to\RR$ of $\rho$, together with the projective class of a positive definite quadratic form $\beta$ on $V$.

A generalized cusp is also determined by $[\beta]$ together with
the Lie algebra weights $\xi_i:V\to\RR$ of $\rho$ that are arbitrary subject to a simple {\em geometric} constraint (\ref{weightseqtn}).
The weights may be regarded as harmonic $1$-forms representing elements of $H^1(C)$.
These $1$-forms determine transversally measured foliations on $\bdy C$ which, together with the similarity structure, determine $C$.
For non-diagonalizable holonomy, the cohomology classes are arbitrary subject to being pairwise orthogonal with respect to the dual of $\beta$.

The next description is differential-geometric: as the projective class of the sum of a quadratic and a cubic
differential both defined on $\bdy C$. This exhibits  $\Tcal_n$ as the product of the space of
flat conformal structures on $\bdy C$ times a cone in the space of cubic polynomials on $V$. The second
factor is a closed cone in $\Sym^3V$ that is not a manifold. Points in the interior of this cone correspond to diagonalizable holonomy.
The cone point corresponds to a {\em standard cusp}.
%This description illustrates that a generalized cusp is a deformation of a standard cusp. 
The cubic is a weighted sum of the cubes of the weights,
and it is {\em harmonic} if and only if  $\bdy \Omega$ 
is an affine sphere.

For three-manifolds this data is encoded by $(w,r,h)\in\CC^3$
subject to $\Im w>0$ and $|r|\le3|h|$. Here $w$ determines the conformal structure on $\bdy C$, and $r,h$
are respectively the {\em radial} and {\em harmonic} components of the cubic polynomial. The generalized
cusp is standard, with cusp shape $w$, if and only if $r=h=0$. }

\section{Summary of results}
Given $\ppsi \in \Hom (\RR ^n, \RR)$ with $\ppsi (e_1) \geq \ppsi (e_2) \geq \cdots \geq \ppsi(e_n) \geq 0$
a {\em generalized cusp Lie group} $G(\ppsi ) \subset \Aff(n)$ was defined in
 \cite{BCL}  and  generalized cusps correspond to 
 lattices in  $G(\ppsi)$. Two generalized cusps are {\em equivalent} if they deformation-retract to affinely isomorphic cusps. 

 Every generalized cusp is equivalent to a {\em homogeneous} one for which $G(\psi)$ acts transitively on $\bdy\Omega$.
For these, there is a natural  underlying Euclidean metric on $\bdy C$. 
 This metric is covered by one on $\bdy\widetilde C=\bdy\Omega\subset\AA^n$
 that is conformally equivalent to the 
 second fundamental form, and is scaled so that $\operatorname{volume}(\bdy C)=1$. It follows from the Bieberbach theorems that $C$ has a finite cover by a generalized cusp
 with boundary a torus  $T^{n-1}=\RR^{n-1}/\ZZ^{n-1}$.  These are called {\em torus cusps} and
 we concentrate on them.
The general case 
 reduces to this by (\ref{nottorus}).
 
 Set $V=\RR^{n-1}$. It is shown in \cite{CLT1}
  that $G(\psi)$ contains a unique subgroup $\editC \Tr(\ppsi)\cong V$ called the {\em translation subgroup} that
  acts simply transitively on $\bdy\Omega$.
  Moreover the image of the holonomy  $\rho:\ZZ^{n-1}\to\editC \Tr(\ppsi)$ is a lattice.
  Thus $\rho$  extends to an isomorphism $\rho:V\to \editC \Tr(\ppsi)$ called the {\em extended holonomy}.
 
   The moduli space of equivalence classes of 
  {\em marked} generalized cusps diffeomorphic to 
  $C$ is denoted $\Modsp(C)$ and $\Modsp_{n}:=\Modsp(T^{n-1}\times[0,\infty))$. {\editD It consists
  of equivalence classes of developing maps.}
The map, that sends a point in $\Modsp(C)$ to the conjugacy class of the extended holonomy,
  identifies $\Modsp(C)$ with the subspace {\editD $\Rep(C)$} of
    the quotient space 
  $\Hom(V,\Aff(n))/\Aff(n)$ consisting of conjugacy classes
  of isomorphisms onto translation subgroups, see (\ref{holishomeo}). 
  Fenchel-Nielsen coordinates provide a lift of Teichmuller space into the representation variety. However,
   we do not know if it is possible to lift  $\Modsp(C)$ into $\Hom(\pi_1C,\Aff(n))$.    
  
   {\editE   Let $A_n$ be the closed Weyl chamber of $\SL(n+1,\RR)$. There is a family
   of representations parameterized by $A_n\times \SL V$.
  Theorem (\ref{modspquotient}) says the holonomy map identifies $\Tcal_n$ with  the quotient of $A_n\times \SL V$ where 
  $(\lambda,A)$ is identified to $(\lambda,A')$ whenever $A^{-1}A'$ lies
  in a certain orthogonal group that depends on $\lambda$.}

  The Euclidean structure on $\bdy C$  pulls back to 
 give a unimodular positive definite quadratic form $\beta_{\rho}$ on $V$. The {\em character} $\chi_{\rho}:V\to\RR$
 is given by $\chi_{\rho}(v)=\trace(\rho v)$. The  {\em complete invariant} of $\rho$ is
 $\cinvt(\rho)=(\chi_{\rho},[\beta_{\rho}])$.  It plays the role in our theory that the character plays
 in the theory of semi-simple representations, namely two representations have the same complete invariant if and only if they are conjugate.
 {\editD The {\em {\editD trace-variety}} $\chi(V)$ is the set of all characters.}  Let $X_n$ be the set of all $\cinvt(\rho)$ {\editD topologized as a subspace of $\chi(V)\times\PP\Sym^2V$.} 
 \begin{theorem}\label{completeinvtvar} The complete invariant  $\eta:\Modsp_n\to X_n$ is a homeomorphism.
  \end{theorem}
  
In  \cite{Dold} Dold    studies the symmetric product $\SP^n X =\left(\prod_{1}^{n}X\right)/S_n$ of a topological space $X$,  where the symmetric group $S_{n}$ permutes factors. 
When $X=V$ and $n>1$, this is distinct from the vector space, $\Sym^n V$, of  symmetric tensors of degree $n$.
The {\editA linear part of the} holonomy $\rho$ has $n$ weights $\exp\dxi_i$ (counted with multiplicity) where $\dxi_i\in V^*$, 
and these give a point  $\xi_{\rho}=[\xi_1,\cdots,\xi_n]\in \SP^n V^* $. 
The following description of the moduli space is reminiscent of the classification of semi-simple Lie groups via {\em roots}.
Let    $\Pos\subset \Sym^2V$ be the space of unimodular positive definite quadratic forms on $V$. 
 Define $\Rcal_n$ to be the subspace of all $([\xi_1,\cdots,\xi_n],\beta)$ in $\SP^n(V^*)\times\Pos$ satisfying the {\em weights equation}
\begin{align}\label{weightseqtn}\exists\ \varpi\ge 0\qquad \forall\ i\ne j\quad \langle \xi_i,\xi_j\rangle_{_{\beta^*}}=-\varpi\end{align}
where $\langle \cdot,\cdot\rangle_{_{\beta^*}}$ is the inner product on $V^*$ dual to $\beta$.
{\editD A geometrical interpretation of this condition is given in (\ref{sigmaeqtn}).}

 \begin{theorem}\label{weightspace}
The {\em weight data} is $\nu:\Modsp_n\longrightarrow \Rcal_n$ given by $\nu(\rho)=\editC (\xi_{\rho},\beta_{\rho})$ and is a
homeomorphism, {\editB and $\Rcal_n$ is a semi-algebraic set.}
Moreover generalized cusps with non-diagonalizable holonomy form the subspace of $\Rcal_n$ where $ \varpi=0$. 

{\editD
Let $F_n=\{[v_1,\cdots,v_n]\in\SP^n V\ :\ \exists\ \varpi\ge 0\ \ \forall\ i\ne j\ \  \langle v_i,v_i\rangle=-\varpi\}$ and $U_n\subset\SL V$ be the group of upper triangular unipotent matrices.
There is a bundle isomorphism $$\theta:U_n\times F_n\to\Rcal_n\qquad\text{given by}\quad\theta(A,[v_1,\cdots,v_n])=(\editC [\xi_1,\cdots,\xi_n],A^tA)$$
where $\xi_i(v)=\langle v_i,Av\rangle$.}
   \end{theorem}

   The {\em type} of $\rho$ is the number of non-trivial distinct {\em weights} of $\rho$,  and can be
any integer $0\le \type\le n$. It equals the number of non-zero coordinates of $\ppsi$ and also of $\xi_\rho$. 
   There is an affine projection $\pi:\Omega\to(0,\infty)^{\type}$.
Each fiber has the geometry of 
horoball in $\HH^{n-\type}$. The  geometry transverse to the fibers 
is {\em Hex geometry}: the projective geometry
of an open simplex, see \cite{BCL} Section 1.5.
 
 The similarity structure is part of a certain kind of geometric structure on $\bdy C$,
 called a {\em cusp geometry},
 that uniquely determines the cusp up to equivalence. The extra structure
  consists of $\type$  transversally measured
 codimension-1 foliations with flat leaves.   The foliations are the preimages of foliations of $(0,\infty)^{\type}$ by coordinate hyperplanes.
  When $\type<n$ then these
 foliations are {\editA arbitrary, subject to being} pairwise orthogonal. %Each transverse measure encodes a weight of the holonomy.
 The transverse measures are harmonic 1-forms representing the cohomology classes $\dxi_i$ given by the weights.

 The cusp geometry is also encoded by a polynomial, $J$, {\editC called the {\em shape invariant}}, defined up to scaling, that 
 is the sum of the quadratic, $\beta_\rho$, and a cubic. This gives an embedding of the marked moduli space into
 the vector space of such polynomials. 
 Projection onto the quadratic term exhibits the moduli space 
 as a bundle over $\Pos$. The fiber
 is a cone in the space of cubic differentials. 
 {\editA The cubic is a  linear combination of the cubes of the weights (\ref{ccformula}).}

 This is reminiscent of the result of Hitchin \cite{Hit}, Labourie \cite {Lab1}, and Loftin \cite{Lof1}, that the
 moduli space of properly convex structures on a closed surface is a vector bundle over
 the space of conformal structures, with fiber the space of \emph{holomorphic} cubic differentials. 
 {\editA However, in general the cubic differentials for generalized cusps are not holomorphic.}

The polynomial $J$ is defined as follows. Choose a basepoint $b\in\bdy\Omega\subset\RR^n$ and
an affine map $\tau:\RR^n\to\RR$ so that $\tau(b)=0$ and $\tau(\interior\Omega)>0$.
The hyperplane $H=\tau^{-1}(0)$ is then tangent to $\Omega$ at $b$. The hypersurface 
$\bdy\Omega$ is parameterized by the function $\mu:V\to\bdy\Omega$
given by the orbit, $\mu(v)=\rho(v)(b)$ of $b$.
The function $h=\tau\circ\mu$ can be thought of as the height of points
 in $\bdy \Omega$ above $H$. {\em \editC However $\bdy\Omega$ is {\bf not} the graph of $h$, see (\ref{htlemma})}.  Then $J:V\to\RR$ is the 3-Jet of $h$,  normalized 
 so the quadratic term is unimodular. The cubic 
 is zero if and only if  $C$ is equivalent to a cusp in a hyperbolic manifold. This is similar to  \cite[Thm 4.5]{Nom}, that an affine hypersurface is quadratic if and only if a certain cubic differential form vanishes identically.
  There is a subspace $\Jcal_{n}\subset\PP(\Sym^2 V\oplus\Sym^3V)$ defined
  in (\ref{cuspspacestructures}) and

 \begin{theorem}\label{polythm} If $n\ge 3$ then {\editC the shape invariant}
 $J:\Modsp_n\to \Jcal_{n}$ is a homeomorphism. Moreover, the projection  $\pi:\Jcal_n\to\Pos$ is a
 trivial bundle
 with fiber homeomorphic to a closed cone in  $\Sym^3V$.
 \end{theorem}
 
%There is a natural action of the orthogonal group $\SO(n-1)$ on each fiber of $\pi$. This 
%is the restriction of the representation of $\SO(n-1)$ on the space of cubic polynomials $\Sym^3V$. This representation
% is the sum of two irreducible representations. One, which we call the space of {\em radial cubics}, is isomorphic to $V$, %and the other is the space of {\em harmonic
 %cubics}.  
 {\editC The cubic is harmonic if and only if $\bdy\Omega$ is an affine sphere.} 
 %{\editC A pleasant general fact is that the  affine normal to a convex surface is determined
 %by the radial cubic part of the 3-Jet of the graph of a convex function,   (\ref{affinenormalencode}), though %see(\ref{affsphharmonic}) .}
 The moduli space $\Modsp_n$ is stratified by  type. The stratum for each type is a manifold whose dimension increases
 with type, see Proposition (\ref{strata}). The frontier of the stratum of type $\type$ consists of the union of strata of smaller type. 
 The largest type corresponds to diagonalizable holonomy.
 In particular:
 %diagonalizable generalized cusps are dense.
% The next result is an immediate simple consequence, for which we do not know a more direct proof.
  \begin{corollary}\label{dense} Every generalized cusp is a geometric limit of diagonalizable cusps.
  \end{corollary}
  It seems hard to show this directly. Another consequence is:
  \begin{theorem}\label{notmanifold} {\editC $\Modsp_n$ is 
  contractible, of dimension $k=n^2-n$, and is manifold if and only if $n=2.$}
\end{theorem}

 Suppose $M=\EE^n/G$ is 
 a closed Euclidean manifold with holonomy $\rho:\pi_1M\to\Isom(\EE^n)$.
 Using the decomposition $\Isom(\EE^n)=O(n)\ltimes\RR^n$ gives a surjection
 $R:\Isom(\EE^n)\to O(n)$ called the {\em rotational part}. 
 By the Bieberbach theorems \cite{Bieb}, \cite{Cha} $M$
 has a finite cover by a torus $T^n=\EE^n/H$ where $H$ is a lattice in $\RR^n$. Thus
 $R\circ\rho(\pi_1M)$ is a finite subgroup $F\subset O(n)$ and we may choose $H=\ker R\circ\rho$.
 Applying this to the generalized cusp $C\cong M\times[0,\infty)$ shows that there is a finite cover
 $p:\widetilde C\to C$ corresponding to $H$, and $\widetilde C\cong T^n\times[0,\infty)$.   
 %{\purple change $H$ to $\Gamma$ or $\Delta$?}
  
 This cover induces a map $p^*:\Modsp(C)\to\Modsp(\widetilde C)$ {\editD that sends an affine structure on $C$
 to the structure on $\widetilde C$ that covers it. This structure on $\widetilde C$
 is preserved by the action of $F$ by covering transformations. Using the identification of a structure
 with its holonomy gives an algebraic formulation.}
 Since $H$ is an abelian  normal subgroup of $G$, the action of $G$ on $H$ by conjugation
 determines a homomorphism $\theta:F\to\Aut(H)$.
Define
$${\editD \Rep(\widetilde C;\theta)}=\{[\rho]\in{\editD \Rep}(\widetilde C) : \forall\ f\in F
\ \  \ \rho\sim \rho\circ\left(\theta f\right) \}$$
 where $\sim$ denotes conjugate representations.

 \begin{theorem}\label{nottorus} The map $\hol{\editE \circ\; p^*}:\Modsp(C)\to\editD \Rep(\widetilde C;\theta)$ is a homeomorphism. 
  \end{theorem}
  
 {\editC  A generalized cusp $C$ in a 3-manifold is determined by three complex numbers $(w,h,r)$ subject to $\Im w>0$
 and $|c|\le 3|h|$.
The conformal structure on $\bdy C$ is $\CC/(\ZZ\oplus\ZZ w)$. The parameter $w$
was used by Thurston to describe cusps in hyperbolic 3-manifolds.
There is a unique upper-triangular matrix $A=A_w\in\SL(2,\RR)$ with positive eigenvalues
 such that the Mobius transformation $\alpha$ corresponding to $A$ satisfies  $\alpha(w)=i$.
 Then the quadratic term in $J$ is $q_w=A^tA\in\Sym^2\R^2$.

After identifying $\RR^2\equiv \CC$ a cubic $p\in\Sym^3\RR^2$ is uniquely expressible as $p=\operatorname{Re}(h z^3)+\operatorname{Re}(r z  |z|^2)$
for some $h,r\in\CC$. The first term is harmonic and the second is called radial.    
\begin{theorem}\label{3manifold} {\editC There is a homeomorphism 
  $$\Theta:\Modsp_3\longrightarrow \{(w,h,r)\in\CC^3: \Im(w)>0,\ \ |r|\le 3|h|\}$$ 
 If $\theta(x)=(w,h,r)$ then
$J(x)=[q_w+c]$ with  $q_w,A_w$ as above,
and $c=\operatorname{Re}(h z^3+r z  |z|^2)\circ A_w$.}
   \end{theorem}
  
  This result determines exactly which cubic differentials appear.
 One may regard the generalized cusp for $(w,h,r)$ as a deformation of the hyperbolic cusp corresponding to $(w,0,0)$.
 The generalized cusps with a fixed conformal structure, $w$, on the boundary are parameterized by a point
 in $\{(h,r)\in\CC^2:\ |r|\le3|h|\}$. This is a cone on a solid torus. The cubic is harmonic if and only if $r=0$, in which case either the cusp 
  holonomy is  conjugate in $\GL(4,\RR)$ into a unipotent subgroup of $O(3,1)$, or else into the diagonal subgroup of $\Aff(\RR^3)$
 where the determinant is one. }
 
  We assume the reader is familiar with  the  main results and definitions up to the end of Section 1
   from \cite{BCL}. {\editC Each facet of the closed Weyl chamber $A_n\subset\RR^n$  parameterizes those translation groups 
   $\Tr(\psi)$ of a fixed type. The main new ingredient, (\ref{lambdakappa}), is a connected set $\widetilde A_n$ of representations that give conjugates of 
   generalized cusps of all types. 
   
   \editC The set $\widetilde A_n$ is obtained by a kind of iterated {\em blowup} of $A_n$ in the sense of algebraic geometry, and each fiber of each blowup consists of
   pairwise conjugate representations. There seems to be no obvious way to replace $\widetilde A_n$ by a continuous family containing only one representative of
   each conjugacy class. The subspace of $\widetilde A_n$ consisting  of {\em guys} of type $\type$ 
   is the interior of a compact manifold, $M$, with boundary. The {\em direction} that a sequence 
   $\rho_n\in\interior(M)\subset\widetilde A_n$ 
    converges to a point  $p\in\bdy M $ determines a point in $\widetilde A_n$ that is some conjugate 
   of some representation corresponding to $p$. %This is related to how transitions occur between generalized cusps of different types.
   
   The paper is organized as follows. In  Section 2 we review the translation
    groups $\Tr(\psi)$ and show that a marked translation group
   is uniquely determined by the complete invariant. In Section 3 we introduce a connected space $\widetilde A_n$ that
   continuously parameterizes
   translation groups   of all types. In Section 4 we 
      prove the complete invariant provides an embedding of
       the marked moduli space $\Modsp_n$. In Section 5 we obtain the characterization (\ref{weightseqtn}) 
       of the weights of marked translation groups.
       In Section 6 we show that a marked translation
   group is determined by the sum of a quadratic and a cubic differential. In Section 7 we compute $\Modsp_3$, the marked moduli space for 3-manifolds.  {\editD Various routine computational proofs were moved into an appendix to avoid disrupting the flow of ideas.}}

    The proof that the {\editC shape invariant} determines a marked 
    translation group that is unique up to conjugacy  is a rather long  and technical computation in Lemma \ref{localmaxlemmadiag}
    that is an ad-hoc algebraic argument. Perhaps there is a better way to establish this with some differential
geometry. The various descriptions of the moduli space only gradually emerged as we stumbled upon various clues.
In particular, the new
parameters in Section 2 were discovered by a very circuitous route.
We thank {\em Kent Vashaw} for assistance with some representation theory
{\editC and {\em Daniel Fox} for providing references concerning the affine normal, and a proof of (\ref{affinenormalencode}) based on them. }  The first author was partially supported by the NSF grant DMS-1709097. The second author
 thanks the University of Sydney Mathematical Research
Institute (SMRI) for partial support and hospitality while working on this paper. The third author was partially supported by ISF grant 704/08. 

\section{The Complete Invariant}

Throughout  $V\equiv\RR^{n-1}$ will denote the {\em extended domain} of the holonomy of a marked generalized cusp,
and $\{e_1,\cdots,e_{k}\}$ is the standard basis of $\RR^k$, {\editB and $\{e_1^*,\cdots,e_k^*\}$ is the dual basis of the dual vector space.}  If $X\subset\RR^n$ then
$\GL(X)\subset\GL(n,\RR)$ is the subgroup that preserves $X$.
Affine space is $\AA^n:=\RR^n\times 1\subset\RR^{n+1}$ and
the affine group is $\Aff(n):=\GL(\AA^n)\subset\GL(n+1,\RR)$.
% andconsists of all matrices $A$ with last row defined by $A_{n+1,i}=\delta_{n+1,i}$.
If $X\subset\AA^n$ then $\Aff(X)\subset\Aff(n)$ is the subgroup that preserves $X$. 
What follows, up to Theorem \ref{classification}, is from \cite{BCL}.

\begin{definition} Suppose $\Omega\subset\AA^n$ is a  closed, convex, subset bounded 
by a non-compact, properly
embedded, strictly convex hypersurface $\bdy\Omega$. Also
suppose $\Aff(\Omega)$ contains a subgroup $T=T(\Omega)\cong(V,+)$ 
 that
acts simply-transitively on $\bdy\Omega$. Then
$T$ is called a {\em translation group} and the group $G(\Omega)\subset\Aff(\Omega)$ 
that preserves each $T$-orbit is called a {\em cusp Lie group}.\end{definition}

 The subgroup $T$ is unique.
The $T$-orbit of a point in $\Omega$ is called a {\em horosphere}. Horospheres are smooth, strictly-convex
hypersurfaces that foliate $\Omega$. In particular $\bdy\Omega$ is a horosphere. 
Moreover $G(\Omega)=\Aff(\Omega)$ unless $\Omega\cong \HH^n$, in which case {\editC $G(\Omega)$ is conjugate into 
a subgroup $\PO(n,1)$.}
A {\em generalized cusp} is an affine manifold $\Omega/\Gamma$ where $\Gamma\subset G(\Omega)$
 is a torsion-free lattice.
Choose a {\em basepoint} $b\in\bdy\Omega$. The subgroup $\O(\Omega,b)\subset G(\Omega)$ that fixes
 $b$   is called a {\em cusp orthogonal group}, and is compact, and
 $G(\Omega)=\O(\Omega,b)\ltimes T$. {\editA Different notation was used for this in \cite[Definition 1.45]{BCL}.}
  We focus on torus cusps. Then  the holonomy is an isomorphism
  $\theta':\ZZ^{n-1}\to\Gamma\subset T$. The {\em extended holonomy}  is the extension of this homomorphism to
an isomorphism $\theta:V\to T$. 

\begin{definition}\label{markedtransgrpdef}
A {\em marked translation group} is an isomorphism $\theta:V\to T$ where $T\subset\Aff(n)$
is a translation group. 
\end{definition}

Given a marked translation group $\theta$, there is a direct sum decomposition  
\begin{equation}\label{DplusU}V=D\oplus U\end{equation}  where
  $\theta(U)$ is the subgroup 
of unipotent elements, and $\theta(D)$ is the subgroup  
of elements for which the largest Jordan block has size $2$. Thus $\theta(D)$
 contains the diagonalizable subgroup.  In the notation of {\editD \cite[(1.41)]{BCL} $U=P(\psi)$ and $D=T_2$.}

\begin{definition}\label{unipotentrankdef} The {\em type} $\type:\RR^n\to \ZZ$, the {\em unipotent rank} $\ur:\RR^n\to \ZZ$ and the 
{\em rank} 
$\rank:\RR^n\to \ZZ$ are defined for $x=(x_1,\cdots,x_n)$ by 
$$\type(x)=|\{i:x_i\ne 0\}| \qquad \rank(x)=\min(\type(x),n-1)\qquad \ur(x)+\rank(x)=n-1$$
\end{definition}

These functions are used in the context of two families of marked translation groups that involve a 
parameter $x\in\RR^n$ and for these, 
 $\rank(x)=\dim D$, and $\ur(x)=\dim U$, and $\type(x)$ is the number of non-constant weights of $\theta$. 
 If  $\psi:\RR^n\to \RR$ is a homomorphism  we will often identify
 $\psi$ with $(\psi_1,\cdots,\psi_n)\in \RR^n$  where $\psi_i=\psi(e_i)$.

\begin{definition}\label{psigroup}  {\editC The group $\Tr(\psi)=\zeta_{\psi}(V)$ is defined as follows.}
$$\begin{array}{rcl}
A_n(\Psi)&:=&\{(\psi_1,\cdots,\psi_n):\psi_1\ge\psi_2\ge\cdots\ge\psi_n\ge 0\}\\
A_n^u(\Psi)&:=&\{(\psi_1,\cdots,\psi_n):\ \forall\ i\ \psi_i\ge 0\ \&\ \exists\ t\ (\psi_i>0\Leftrightarrow i\le t)\ \}\} \qquad (\textrm{unordered})\\
\end{array}
$$
If  $\psi\in A_n^u(\Psi)$
set $\type=\type(\psi)$ and $\ur=\ur(\psi)$ and $\rank=\rank(\psi)$.
If $\type=0$ set $E=\emptyset$ and $\psi^-=0$, otherwise {\editC define $\psi^-\in V^*$ and $E$} by $$\psi^-(v_1,\cdots,v_{n-1})=-\psi(v_1,\cdots,v_{n-1},0),\qquad E={\editC\psi_{\type}\cdot}\Diag(v_1,\cdots,v_{\rank})$$
Define  $\zeta_{\psi}:V\to\Aff(n)$ by $\zeta_{\psi}(v)=\exp f_{\psi}(v)$
 where
 $f_{\psi}(v)=$
$$
\begin{array}{ccccc} \type<n-1 & &\type=n-1 &  & \type=n \\
\\
\begin{pmatrix}
{\editB E} & 0\\
0 &\begin{pmatrix}
0 & v_{\rank+1}&\cdots& v_{\rank+\ur} & \psi^-(v)\\
0 &\cdots &  & 0& v_{\rank+1}\\
\vdots & &  & & \vdots\\
0 &\cdots  &  & 0& v_{\rank+\ur}\\
0 & \cdots &  &   0& 0
\end{pmatrix} \\
\end{pmatrix} 
 & &  \begin{pmatrix}
{\editB E} &  0 & 0\\
0 & 0 & \psi^-(v) \\
0 & 0 & 0
\end{pmatrix} & &
 \begin{pmatrix}
E &  0 & 0\\
0 & \psi^-(v) & 0 \\
0 & 0 & 0
\end{pmatrix}
\end{array}
$$
Observe that $\rank+\ur=n-1$.
\end{definition}
Since all the eigenvalues are positive, $\zeta_{\psi}:V\to \Tr(\psi)$  is an isomorphism, so $\Tr(\psi)\cong\RR^{n-1}$. {\editC
It follows from \cite[Theorem 0.2]{BCL}, and we show below, that  $\zeta_{\psi}$ is conjugate to $\zeta_{\psi'}$ if and only if $\psi=\psi'$. However
$\Tr(\psi)$ and $\Tr(\psi')$ are conjugate subgroups if and only if $\psi=s\psi'$ for some $s>0$.} 
 
 \begin{theorem}\label{classification}
  {\editC (a) $\Tr(\psi)$ is a translation group.\\
 (b) If $s>0$ then $\zeta_{s\psi}=\zeta_{\psi}\circ ((s \Id_{\rank})\oplus \Id_{\ur})$.\\
Suppose $\theta:V\to\Aff(n)$ is a marked translation group then  \\
(c) $\exists !\ \psi\in A_n(\Psi)$ and $\exists\ B\in \SLpm(V)$ such that $\theta$ is conjugate to $\zeta_{\psi}\circ B$.\\
(d) $\exists\ \psi'\in A_n^u(\Psi)$ and $\exists\ B'\in \SL(V)$ such that $\theta$ is conjugate to $\zeta_{\psi'}\circ B'$.}
 \end{theorem}
 \begin{proof} {\editC (b) 
 %each entry in $f_{\psi}(v)$ is either some $v_k$, or else is a bilinear function of $v$ and $\psi$.
 %If $\type\ge n-1$ then $\rank=n-1$ and $f_{s\psi}(v)=s f_{\psi}(v)=f_{\psi}(sv)$. When $\type <n-1$ then $\ur>0$, and
{\editD  The definition shows
 $f_{s\psi}(v_1,\cdots,v_{n-1})=f(sv_1,\cdots,sv_{\rank},v_{\rank+1},\cdots, v_{\rank+\ur})$. }
 
(a) Given a marked translation group $\rho:V\to\Aff(n)$ then,
by \cite[Theorem 0.1]{BCL}, there is $\psi\in A_n(\Psi)$  such that $\rho(V)$ is conjugate into
 the group $T(\psi)$ defined in \cite[Definition 1.32]{BCL}. Moreover if $\psi\ne0$ we may choose $\psi_{\type}=1$ and then $T(\ppsi)=\Tr(\ppsi)$ as in (\ref{psigroup}). {\editD This proves (a)}. 
 
It follows that
 $\rho=\zeta_{\psi}\circ A$ for some $A\in\GL(V)$.  If $\rank>0$ then there is $s>0$ so that $A= ((s \Id_{\rank})\oplus \Id_{\ur})B$
 with $B\in\SLpm(V)$. Then
 $\rho=\zeta_{s\psi}\circ B$ by (b).
 If $\rank=0$ then $\psi=0$ and $\zeta_0\circ(s \Id)$ is conjugate to $\zeta_0$. Thus in this
 case we may also choose  $B\in\SLpm(V)$.
 
To show $\psi$ is unique, by \cite[Theorem 0.2]{BCL} $\psi$ is unique up to multiplication by some $s>0$. Suppose $\zeta_{\psi}\circ B$ is conjugate
to $\zeta_{s\psi}\circ B'$. Then $\zeta_{\psi}$ is conjugate to $\zeta_{s\psi}\circ (B' B^{-1})$, and thus to 
$\zeta_{\psi}\circ ((s \Id_{\rank})\oplus \Id_{\ur})B' B^{-1})$. By \cite[Theorem 0.2]{BCL} 
$((s \Id_{\rank})\oplus \Id_{\ur})B' B^{-1}\in \O(\EE^{n-1},\psi)$. By \cite[(1.44)]{BCL} this is a subgroup of the orthogonal group, thus $s=1$. This proves (c).
}
 
 For (d), when $n=2$ the result is easy,
 so assume $n\ge 3$ and $\det B=-1$.
 There are two coordinates $\psi_i,\psi_{i+1}$
 of $\psi$ that are either both zero or both non-zero. Swapping columns $i$ and $i+1$ of $B$
 gives $B'\in\SL (V)$ and swapping $\psi_i$ and $\psi_{i+1}$ gives $\psi'\in A_n^u(\Psi)$. Then 
$\zeta_{\psi}\circ B$ is conjugate to $\zeta_{\psi'}\circ B'$ by swapping the $i$ and $i+1$ coordinates
in $\RR^{n+1}$.
  \end{proof}

{\editA   We regard the second symmetric power, $\Sym^2 V$, as
the
vector space of homogeneous polynomials $\beta:V\to\RR$ of degree two. 
The subspace $\widetilde{\Pos}(V)\subset\Sym^2 V$
consists of positive definite forms and $\Pos(V)\subset\widetilde{\Pos}(V)$ is the subspace of unimodular forms.
Let $\pi_{_{\Pcal}}:\widetilde\Pcal(V)\to\Pcal(V)$ be the projection
$$\pi_{_{\Pcal}}(\beta)=(\det\beta)^{-1/(n-1)}\beta$$
 {\editB The notation $\beta\sim\beta'$ means there is $\lambda>0$ with $\beta'=\lambda\beta$.}
}
Given a marked translation group $\theta:V\to\Aff(n)$ 
  the {\em orbit map} $\mu_{\theta,b}:V\to\bdy\Omega$ is the homeomorphism defined by 
 \begin{equation}\label{orbitmap}\mu_{\theta,b}(v)=\left(\theta v\right)b\end{equation}
 where $b\in\bdy\Omega$ is some choice of basepoint. 
Since $\bdy\Omega$ is smooth and strictly convex, there is a unique
affine hyperplane $H_b\subset\AA^n$ with $H_b\cap\Omega=b$.  There is an affine map $\tau:\AA^n\to\RR$ 
with $\tau(H_b)=0$ and
 $\tau(\interior \Omega)>0$. 
The {\em height function}   \begin{equation}\label{heightfn}h_{\theta}=\tau\circ\mu_{\theta,b}:V\to\RR\end{equation} 
is only defined up to multiplication by a positive real.  We remind the reader that $\bdy\Omega$ is {\bf not} the graph of $h_{\theta}$, see (\ref{htlemma}). Note that if $b'$ is a different choice of basepoint, then there is unique element $A\in T(\Omega)$ such that $Ab=b'$. In this case $\tau'=\tau\circ A^{-1}$ is an affine map such that $\tau'(H_{b'})=0$ and $\tau'(\interior \Omega)>0$. Furthermore, $\mu_{\theta,b'}=A\circ \mu_{\theta,b}$, and so $\tau'\circ \mu_{\theta,b'}=\tau\circ \mu_{\theta,b}$. It follows that the height function is independent of the choice of basepoint.

{\editA Since $\bdy\Omega$ is strictly convex one obtains  positive definite quadratic forms }
\begin{equation}\label{betaeq}\widetilde\beta(\theta)=\operatorname{D}^2h_{\theta},\qquad \beta(\theta)=\pi_{\Pcal}(\widetilde\beta(\theta))\in\Pcal(V)\end{equation} 
{\editD After rescaling, the orbit map is an isometry from $(V,\beta)$ to $\bdy\Omega$ with the horosphere metric \cite[(2.14)]{BCL}.}
{\editE The form $\widetilde\beta$ is only defined up to scaling. To emphasize this we usually work with $[\beta]\in\PP\Pcal$.
However it is sometimes convenient to use the natural identification $\Pcal\equiv\PP\Pcal$. Then one
must remember that preserving $\beta$ only means $\beta$ is preserved up to rescaling. }

Writing $v=\sum_{i=1}^{n-1} v_ie_i$  and $u_i=(\partial\mu_{\theta,b}/\partial v_i)_{v=0}\in\RR^{n}$ then $(u_1,\cdots,u_{n-1})$ is a basis of 
the tangent space $\operatorname{T}_b\bdy\Omega\cong H_b$. We may use $\tau(x)=\pm\det(u_1,\cdots,u_{n-1},x)$ and
 a height function is then given by
\begin{equation}\label{geqtn1}
h_{\theta}(v)=\pm\det(u_1,\cdots,u_{n-1},\mu_{\theta,b}(v)-b)\end{equation}
where the sign is chosen so that $\tau(\Omega)\ge 0$.

{\editC  The space  $\Hom(V,\Aff(n))$ is  given the weak topology. This coincides with the Euclidean
topology when it is realized as an algebraic subset of Euclidean space.
 The space  $\Hom(V,\Aff(n))/\Aff(n)$ is the quotient space under the action of conjugacy. It is not Hausdorff.
 In Section 4 we define $\Modsp(V)$ as  equivalence classes of developing maps and show it is homeomorphic to $\Rep(V)$.
 Various functions defined on $\Rep(V)$ in this section can then be re-interpreted as functions on $\Modsp(V)$
 }

\begin{definition} ${\editD \widetilde\Rep(V)}\subset\Hom(V,\Aff(n))$
 is the subspace
of  marked translation groups, and
${\editD \Rep(V)}=\widetilde{\Rep}(V)/\Aff(n)$ is the space of conjugacy classes with the quotient topology.
\end{definition}

\begin{lemma}\label{betacts} $\widetilde\beta:\widetilde{\editD \Rep(}V)\to\Pos(V)$ is smooth and covers
a continuous map $\beta:{\editD \Rep}(V)\to\Pos(V)$.
\end{lemma}
\begin{proof} By {\editD the discussion above}    $\widetilde\beta$  does not depend on the choice of basepoint $b$ or height function used above.
{\editD Given a marked translation group  $\theta$} every choice of basepoint $b$ has orbit a convex hypersurface unless $b$ lies
is a projective subspace preserved by $\theta$. 
Thus
in a neighborhood of $\theta$ in $\widetilde{\editD \Rep}(V)$ a fixed choice of basepoint $b$ can be used for the orbit map, \cite[(1.52)]{BCL}. Then 
the function $\mu:\widetilde{\editD \Rep}(V)\times V\to\RR$ given by $\mu(\theta,v)=\mu_{\theta,b}(v)$
is smooth near $(\theta,\editD v)$. Equations (\ref{orbitmap}) and (\ref{geqtn1}) then imply $h_{\theta}$ is smooth near $\theta$, so $\widetilde\beta$ is smooth.
{\editD It is clear that} $\widetilde\beta(\theta)$ is invariant under conjugation of $\theta$. Therefore $\widetilde\beta$
covers a map $\beta:{\editD \Rep}(V)\to\Pos(V)$ which is continuous by properties of the quotient topology.
\end{proof}

The {\em character} of a homomorphism $\rho:V\to\GL(n+1,\RR)$ is $\chi(\rho):V\to\RR$ given by 
$\chi(\rho)=\trace\circ\rho$. The {\em {\editD trace-variety}}, $\chi(V)$, is the set of characters of all such homomorphisms.
{\editD $\Hom(V,\Aff_n)$ is a real algebraic variety, and $\chi(V)$  is its image under a polynomial map.
  Thus $\chi(V)$ is a semi-algebraic set, and  in particular is homeomorphic to a subset of Euclidean space}. 

{\editD By (\ref{classification})} a marked translation group is conjugate to an upper triangular group. The character is not changed by conjugation.
The character of an upper-triangular representation is a function on $V$ that is the sum of $(n+1)$ functions,
each of which is the
 exponential of an element of $V^*$. Thus the subspace of $\chi(V)$ consisting of characters of marked translation
 groups is homeomorphic to a subspace of $\editD \SP^{n+1}V^*$.

\begin{definition}\label{metricdef} 
Given 
a marked translation group  $\theta:V\to \Aff(n)$ then
\begin{itemize}
\item The {\em horosphere metric}   is 
the unimodular quadratic form
$\beta(\theta)\in \Pos(V)$
\item
The {\em complete invariant} is $\cinvt(\theta)=(\chi(\theta),{\editD [\beta(\theta)]})$.
\end{itemize}
Also $\O(\eta(\theta))\subset\GL(V)$ is the subgroup that preserves both $\chi(\theta)$ and $[\beta(\theta)]$.
\end{definition}
{\editD Lemma (\ref{Opsi}) implies $\O(\eta(\theta))$ is a subgroup of the orthogonal group of $\beta$ unless $\type(\theta)=0$,
in which case it is the group of Euclidean similarites fixing $0$.}

\begin{proposition}\label{iotacts} The complete invariant $\cinvt:{\editC\Rep(V)}\to \chi(V)\times \editD\PP(\Sym^2 V)$
is continuous.
\end{proposition}
\begin{proof} It is well known that $\chi$ is continuous, and
 $\beta$ is continuous by (\ref{betacts}). \end{proof}

A dual vector $\weight\in V^*$
is a {\em Lie-algebra weight} of $\theta:V\to\GL(n+1,\RR)$, and  $\exp\circ\weight$ is a {\em weight}, if the {\em weight space}
$$V(\theta,\weight):=\bigcap_{v\in V}\ker(\theta(v)-\exp\circ\weight(v))\ne0.$$
Let $\langle\cdot,\cdot\rangle_{\beta}$ be the inner product on $V$ given by $\beta$.
 Let $\beta^*\in\Sym^2V^*$ denote the dual quadratic form
{\editD defined by $\beta^*(\phi)=\beta(v)$ if $\phi(x)=\langle v,x\rangle$.}
Let $\langle\cdot,\cdot\rangle_{\beta^*}$ be inner product on $V^*$ given by $\beta^*$. {\editC The proof
of the following is routine and in the appendix.}

\begin{proposition}\label{completepsi}  Given $\ppsi\in  A_n^{\editC u}(\Psi)$ the decomposition $V=D\oplus U$ for $\zeta_{\psi}$ is orthogonal with respect to 
$\beta(\zeta_{\ppsi})$. Set  $\ur=\ur(\psi)$ and $\type=\type(\psi)$, then $\beta(\zeta_{\ppsi}) \sim \beta'$ where for $v\in V$
$$\editA\begin{array}{cc}
\type <n &  \type =n\\
\beta'(v) =
    \sum_{i=1}^{\type} \psi_iv_i^2+{\editC \psi_{\type}^{-1}}\sum_{i=\type+1}^{n-1}v_{i}^2
    &  \beta'(v) = \sum_{i=1}^{n-1} \psi_iv_i^2+ \psi_n^{-1}\left(\sum_{i=1}^{n-1}\psi_iv_{i}\right)^2\\
     \chi(\zeta_{\ppsi}) (v) =
 2+\ur+  \sum_{i=1}^{\type} \exp({\editC \psi_{\type}}v_i)
&     \chi(\zeta_{\ppsi}) (v) =
 1 + \sum_{i=1}^{n-1} \exp({\editC \psi_n}v_i)+\exp\left(-\sum_{i=1}^{n-1}\psi_iv_i\right)
\end{array}
 $$
 Moreover, when $\type<n$ then $\det\beta'= \psi_1\cdots\editC \psi_{\type-1}\psi_{\type}^{\type+2-n}$ and
 the  non-zero Lie
 algebra weights of $\zeta_{\ppsi}$ are $\{\dxi_i={\editD\psi_{\type}}e_i^*:1\le i\le\type\}$,
 and their duals are an orthogonal basis of $D$, and
 $\beta^*(\xi_i)={\editD\psi_{\type}^2}\left(\det\beta'\right)^{-1/(n-1)}\psi_i^{-1}$.
 {\editC Also when $\type=n$ then $\det \beta'=\psi_1\cdots\psi_{n-1}\psi_n^{-1}\sum_{i=1}^n\psi_i$}.
  \end{proposition} 
  
  {\editC Theorem (\ref{completeinvt}) shows that the complete invariant determines a marked translation group up to conjugacy. 
  Theorem (\ref{Jcomplete}) shows the same for the shape invariant. {\em  The strategy is the same in both cases.
  One argument shows the invariant determines the translation group up to conjugacy. The second part
 is to  show that the invariant determines the stabilizer of a point 
  $\O(\bdy\Omega,b)\subset G(\Omega)$.}}
  
   {\editA
\begin{corollary}\label{psiformula} Suppose {\editD$n\ge 3$} and  $\theta:V\to\Aff(n)$ is a marked translation group.  
Then $\theta$ is conjugate to $\zeta_{\psi}\circ B$ for some $\psi=(\psi_1,\cdots,\psi_n)\in A_n$  and $B\in\SLpm V$, and
the complete invariant $\eta(\theta)$ {\editC uniquely} determines $\psi$.
\end{corollary}
\begin{proof} {\editC By  (\ref{classification}b) $\theta$ is conjugate to some $\zeta_{\psi}\circ B$ with $B\in\SLpm V$, and $\psi$ is
uniquely determined by the conjugacy class of $\theta$.
If $\type=n$ then $\zeta_{\psi}$ is diagonal, so $\chi(\theta)$ determines $\theta$ up to conjugacy, and hence determines $\psi$
by (\ref{classification}b).}
So suppose $\type<n$.
It follows immediately from  the definitions that $\beta(\zeta_{\psi}\circ B)=\beta(\zeta_{\psi})\circ B$, and
$\xi_i(\zeta_{\psi}\circ B) = \xi_i(\zeta_{\psi})\circ B$. Hence $\beta^*(\xi_i\circ B)=\beta^*(\xi_i)$.
By (\ref{completepsi}) it follows that $\eta(\theta)$ determines 
$$\editC (\beta^*\xi_1,\cdots,\beta^*\xi_{\type})=\psi_{\type}^2\left(\psi_1\cdots \psi_{\type-1}\psi_{\type}^{\type+2-n}\right)^{-1/(n-1)}(\psi_1^{-1},\cdots,\psi_{\type}^{-1})$$
up to permutations. {\editC  Let $x_i=\log \psi_i$ and $y_i=\log\beta^*\xi_i$ and $x=(x_1,\cdots,x_{\type})$
and $y=(y_1,\cdots,y_{\type})$. Define $v:\RR^{\type}\to\RR$ by
\begin{align*}v(x)=\log\left[\psi_{\type}^2\left(\psi_1\cdots \psi_{\type-1}\psi_{\type}^{\type+2-n}\right)^{-1/(n-1)}\right]=-(n-1)^{-1}(x_1+\cdots +x_{\type-1}+(\type+4-3n)x_{\type})
\end{align*}
\if0
 Regard $\{x_i:1\le i\le \type\}$ as independent variables. We wish to solve the equations $y_i=x_i+v$  for $x_i$. These
 are linear equations in the real vector space with basis $\{x_i:1\le i\le \type\}$.  
 A solution exists unless the $y_i$ are linearly dependent, i.e. unless there are $\alpha_i\in\RR$, not all $0$, with
$$\sum \alpha_i(x_i+v)=0$$ 
where in this proof summation is from $1$ to $\type$. Since the $x_i$ are linearly independent, $\sum\alpha_i\ne 0$. Then we may scale
so that $\sum\alpha_i=1$. This implies
$$v=-\sum\alpha_ix_i$$
Multiplying this by $-(n-1)$ and using  (\ref{beq}) gives
$$(n-1)(\alpha_1x_1+\cdots+\alpha_{\type}x_{\type})=-(n-1)v=(x_1+\cdots+x_{\type-1}+(\type+4-3n)x_{\type})$$
Equating coefficients gives $(n-1)\alpha_i=1$ for $i<\type$ and $(n-1)\alpha_{\type}=\type+4-3n$. Using $\sum \alpha_i=1$ gives
$$(n-1)=\sum(n-1)\alpha_i=(\type-1)+(\type+4-3n)=2\type-3n+3$$
This implies $\type=2n-2$. But $n\ge 3$ and $\type\le n$ so there are no such $\alpha_i$, hence the $x_i$ are uniquely 
determined by the $y_i=\log\beta^*\xi_i$, and thus by $\eta(\theta)$.}
\end{proof}
}
\fi
Let $e=(1,\cdots,1)$  then $y=x+(v(x))e=(I+G)x$ where $G=e\otimes v$.
Then $\eta(\theta)$ determines $y$, and 
recovering the $\psi_i$ amounts to finding $x$ that solves the linear equation $y=(I+G)x$.

 We claim that $I+G$ is invertible. For the sake of contradiction assume that $0\neq w\in \ker(I+G)$, then $w+v(w)e=0$. This implies that $w=\alpha e$ for some $\alpha\neq 0$. Since all non-zero multiples of $w$ are also in the kernel there is no loss of generality in assuming that $\alpha=1$. This implies that $e+v(e)e=0$ and so $v(e)=-1$. From the definition of $v$ the equation $v(e)=-1$ becomes $-(n-1)^{-1}(2 \type+3-3n)=-1$, or equivalently that $\type=2n-2$. However, since $n\geq 3$ this implies that $\type>n$, which is a contradiction.  It follows that the $x_i$ can be recovered from the $y_i$, and by exponentiating we recover the $\psi_i$. }
\end{proof}

{\editA The characteristic polynomial of a square matrix $A$  is $c(A)=\det(x\Id-A)$.
An affine automorphism of $\RR^n$ is given by $f(x)=Ax+b$ with {\em linear part $A\in\GL(n,\RR)$} and also given by 
$B\in\Aff(n)\subset\GL(n+1,\RR)$. Then $c(B)=(x-1)c(A)$. This means that a translation group has one more zero Lie-algebra weight
than the   linear part.
The character of a marked translation group determines the weights:}
\begin{lemma}\label{charpoly} {\editA Suppose $\theta:V\to\Aff(n)$ is a marked translation group.
 Let ${\editC \xi_{\theta}}=[\xi_1,\cdots,\xi_n]\in \SP^{n} V^*$ be the Lie-algebra weights of the linear part of $\theta$.
Then the characteristic polynomial $c_{\theta}:V\to\RR[x]$ given by
$$c_{\theta}=\det(x I-\theta)=(x-1)\prod_{i=1}^{n}(x-\exp\circ \weight_i)$$
 is uniquely determined by $\chi(\theta)$.  Moreover  there is $f:X_n\to\Rcal_n$ with $\nu=f\circ\eta$,} where $\nu:\Modsp_n\longrightarrow \Rcal_n$ is the weight data $\nu(\rho) = (\xi_\rho, [\beta_\rho])$.   \end{lemma}
\begin{proof} {\editD Suppose $A=\theta(v)$. Then $\theta(kv)=A^k$ so $\chi(\theta)(kv)=\trace A^k$.}
If $A$ has eigenvalues $\mu_0,\cdots,\mu_n$
counted with multiplicity  then $p_k:=\trace(A^k)=\sum\mu_i^k$ is a symmetric polynomial function of the eigenvalues.
Every symmetric polynomial is a polynomial in the  $p_k$, and in particular the coefficients of $c(A)$ have this property.
Hence  $\chi(\theta)$ determines the characteristic polynomial of $\theta(v)$ for every $v\in V$.
Thus  $\chi(\theta)$  determines the function 
$c_{\theta}=c\circ\theta:V\to\RR[x]$
which sends  $v\in V$ to the characteristic polynomial of $\theta(v)$.
Since all the eigenvalues of $\theta(v)$ are positive,
there are $\weight_i\in V^*$ with
$c_{\theta}=\prod_{i=0}^{n}(x-\exp\circ \weight_i)$. Hence $\chi(\theta)$ determines the Lie algebra weights $\xi_i$.
The factorization of a polynomial into linear factors is unique up to order and scaling. It follows that ${\editC \xi_{\theta}}$ is also uniquely determined,
and thus $f$ exists.
\end{proof}

Given a translation group $T(\Omega)$ together with  a basepoint $b\in\bdy\Omega$,  then $\O(\Omega,b)\subset G(\Omega)$ is the subgroup that fixes $b$, and
 acts on $\RR^n$, preserving $\bdy\Omega$. The orbit map $\mu_{\theta,b}$ identities $V\cong\R^{n-1}$ with 
$\bdy\Omega$,
therefore $O(\Omega,b)$ also acts on $V$. 
 Under this identification $\O(\Omega,b)\subset\Aff(\RR^n)$
is conjugate  to $\O(\cinvt(\theta))\subset\GL(V)$ when $\type>0$. {\editC The group $\Sim(\beta)\subset\GL V$
is the group of similarities that preserve $[\beta]$.
\begin{lemma}\label{Opsi} {\editD Suppose $\theta$ is a marked 
translation group. If $\type(\theta)>0$ then there is an isomorphism $f:\O(\Omega,b)\rightarrow \O(\cinvt(\theta))$ given
by  $f(A)=\mu^{-1}A\mu$ where $\mu=\mu_{\theta,b}:V\to\bdy\Omega$ is the orbit map. 
If $\type(\theta)=0$ then $\O(\cinvt)=\Sim(\beta)$.} 
\end{lemma}
\begin{proof} {\editD Let $\cinvt=\cinvt(\theta)=(\chi,[\beta])$ and $\type=\type(\theta)$.
By  definition $\O(\cinvt)$ is the subgroup of $\Sim(\beta)$ that preserves $\chi(\theta)$. If $\type=0$ then  $\theta$ is unipotent so $\chi$ is constant and the result follows. Now
assume $\type>0$, thus $\chi$ is not constant.} 

 {\editD  
 We claim that $\O(\cinvt)$ is a subgroup of $\O(\beta)$.
 The character $\chi:V\to\RR$ is preserved by the action of $O(\eta)$. Now $\O(\eta)\subset\Sim(\beta)$, so
 if the claim is false there is $A\in\O(\eta)$ that moves all points in $V$ closer to $0$. {\editC It follows
 that $\chi(v)=\lim_{n\to\infty} \chi(A^n v)=\chi(0)$ so $\chi$ is constant,  which is a contradiction.}

 We claim that $f$ maps into $\O(\eta)$.
  The orbit map $\mu=\mu_{\theta,b}$ defined in (\ref{orbitmap}) is given by $\mu(v)=\theta(v)b$.
  Recall that $\bdy\Omega$ is the orbit of $b$ under $\Image\theta$. 
  Given $A\in\O(\Omega,b)$ {\editE and $v\in V$} then $A(\mu(v))=\mu(u)$ for some $u\in V$, and $(f A)(v)=u$.
Since $A$ fixes $b$ it follows that
$$(\theta u)b=\mu(u)=A(\mu v)=(A\theta v)b=(A(\theta v)A^{-1})Ab=(A(\theta v)A^{-1})b$$
Now $A(\theta v)A^{-1}\in T(\Omega)$, and
the action of $V$ on $\bdy \Omega$ is free, thus $\theta u = A(\theta v)A^{-1}$, so $$(f A)(v)=u=\theta^{-1}\left(A(\theta v)A^{-1}\right)$$
Now $\theta$, and conjugation by $A$, are both group isomorphisms, thus $f A$ is a group automorphism of $(V,+)$, and it is continuous
thus $f A\in \GL V$. Now
$$\trace \theta((fA)(v))=\trace \theta u=\trace A(\theta v)A^{-1}=\trace  \theta v$$
Thus $\chi\circ (fA)=\chi$. It is clear that $fA$ preserves $\beta$ hence $fA\in\O(\eta)$, which proves the claim.

\if0
Otherwise we may assume $r>1$.
Because $\O(\beta)$ is compact, there is a large integer $n$ with $B^n$ close to $\Id$. 
Since $\chi$ is not constant, there is $v\in V$ such that the largest eigenvalue of $\theta(v)$ is $\lambda>1$ with an eigenvector
 $x\in \RR^n$ with $\theta(v)x=\lambda x$.
Now $(r B)^nv\approx r^n v$ so $\theta((r B)^nv)\approx  (\theta v)^{r^n}$ and $ (\theta v)^{r^n}x=\lambda^{r^n}x$.
It follows that $\trace\theta((r B)^nv)\ne\tr\theta(v)$  as $n\to\infty$. But then $rB$ does not preserve $\chi$ so $rB\notin\O(\cinvt)$.
\fi
}

The lemma is true  for $\theta$ if and only if it is true for a conjugate of $\theta\circ B$ for some $B\in \GL(V)$.
By Theorem (\ref{classification}) it suffices to prove the result when $\theta=\zeta_{\psi}$. Set $\type=\type(\psi)$.
First consider the case $0<\type<n$ and
define $$B=\Diag(\psi_1^{-1/2},\cdots,\psi_{\type}^{-1/2},\editC\psi_{\type}^{1/2},\cdots,\psi_{\type}^{1/2})\in\GL(V)$$
It suffices to assume $\theta=\zeta_{\psi}^{\perp}:=\zeta_{\psi}\circ B$.
 By (\ref{completepsi}) $$\beta(\zeta_{\psi}^{\perp})(v)=\langle v,v\rangle\qquad \chi(\zeta_{\ppsi}^{\perp}) (v) =
 2+\ur+  \sum_{i=1}^{\type} \exp(\psi_i^{-1/2}{\editC \psi_{\type}}v_i),$$
 where $\langle \cdot,\cdot\rangle$ is the standard inner product on $\R^{n-1}$. 
 
{\editD By (\ref{charpoly}) $\chi(\theta)$ determines and is determined by the Lie algebra weights of $\theta$, thus
$O(\eta)$ is the subgroup of $\O(\beta)$ that preserves the Lie-algebra weights. Hence it is the subgroup that preserves
the set consisting of the duals  $\{v_i:1\le i\le n\}\subset V$ with respect to $\beta$ of these weights. By (\ref{dualwteqtn}) the non-zero duals are
$\{v_i=(\gamma\psi_i)^{-1}\psi_{\type}e_i:\ 1\le i\le\type\}$.  The action of $\O(\eta)$
permutes this set, but preserves the lengths of vectors.
Thus  $\O(\cinvt)$ is the subgroup of $\O(\beta)$
 that permutes  $\{e_i:\ 1\le i\le\type\}$ and preserves the vector 
 $${\editC \gamma^{-1}\psi_{\type}}(\psi_1^{-1},\cdots,\psi_{\rank}^{-1},0,\cdots,0)\in V$$
 where the last $\ur$ coordinates are $0$.
 Clearly this is the same as preserving
  $$(\psi_1,\cdots,\psi_{\rank},0,\cdots,0)\in V$$
 Let $S(\psi)$ be the group of coordinate permutations of $\RR^{\rank}$ that preserve
 {\editC $(\psi_1,\cdots,\psi_{\rank})$, then
 $\O(\cinvt(\zeta_{\psi}))=S(\psi)\oplus O(\ur)$.}
 When $\type<n$ it follows from \cite[Proposition 1.44]{BCL}  that 
 $f(\O(\Omega,b))=S(\psi)\oplus O(\ur)$ which gives the result.}

The remaining case is that $\type=n$, and then $\zeta_{\psi}$
 has $n$ non-zero Lie-algebra weights $\weight_i\in V^*$ and $\sum\psi_i\weight_i=0$.
 Observe that $\psi$ is determined up to scaling by this equation.
If $B\in\O(\cinvt(\zeta_{\psi}))$ then it preserves $\chi(\zeta_{\psi})$, and therefore, by  (\ref{charpoly}), permutes these weights,
so that $\weight_i\circ B=\weight_{\sigma i}$ for some permutation $\sigma$ of $\{1,\cdots,n\}$. However 
$\sum\psi_i\weight_{\sigma i}=0$ so $\psi_i=\psi_{\sigma i}$.
Thus $\mu\cdot B\cdot\mu^{-1}=A\in \Aff(n)$ permutes  the coordinate axes of $\RR^n$
and preserves $\psi$.
Again by  \cite[Proposition 1.44]{BCL}  $A\in \O(\Omega,b)$.
It follows that $\O(\cinvt(\zeta_{\psi}))\subset \mu^{-1}\cdot \O(\Omega,b)\cdot\mu$. 
It is clear that $\O(\cinvt(\zeta_{\psi}))\supset\mu^{-1}\cdot \O(\Omega,b)\cdot \mu$.
\end{proof}

Suppose $\theta:V\to\Aff(n)$ is a marked translation group.
If we consider a generalized cusp as a projective manifold, instead of as an affine one, 
 then the holonomy  might be
given as  $\theta_*:V\to\SL(n+1,\RR)$ where
\begin{equation}\label{wteq}\theta_*(v)=\alpha(v)\cdot \theta(v)\qquad\rm{and}\quad \alpha(v)=\left(\det \theta(v)\right)^{-1/n+1}\end{equation}
It was shown in \cite[Prop.\ 1.29]{BCL} that if two marked translation groups are conjugate in $\GL(n+1,\RR)$
then they are conjugate in $\Aff(n)$, and therefore have the same complete invariant. 
In (\ref{characterlemma})
below we
show if $\theta_*:V\to\SL(n+1,\RR)$ is the corresponding projective
 translation group then $\chi(\theta_*)$ 
determines $\chi(\theta)$. However computations
are simpler using $\chi(\theta)$.

 We now explain how to recover $\theta$ from $\theta_*$. The idea is that to recover the affine action amounts
 to determining the weight of $\theta_*$ that corresponds to the hyperplane at infinity for affine space. 
%(\ref{characterlemma}) says that $\chi(\theta_*)$ determines $\chi(\theta)$ and vice versa.
Suppose $\theta:V\to\GL(n+1,\RR)$ and every weight is {\editD real and} positive.
Let $\Wcal(\theta)=(\weight_0,\weight_1,\cdots,\weight_n)$ be the Lie algebra weights of $\theta$ counted with multiplicity.
The Lie algebra weight $\weight_i$ is called {\em a middle weight}  if
$$\forall\ v\in V\quad \weight_i(v)\le \max\{\weight_j(v)\ :\ j\ne i\ \}$$
Applied to diagonalizable representations, this is the {\em middle eigenvalue condition} of Choi, \cite{choi1}.
It follows that a Lie algebra weight with multiplicity larger than $1$ is a middle weight.

 If  $\theta:V\to \Aff(n)$ then
$\weight_i$ is a middle weight of $\theta$ if and only if $\weight_i=0$. 
From (\ref{wteq}) it follows  that if $\Wcal(\theta)=(\weight_0,\weight_1,\cdots,\weight_n)$ then 
$\Wcal(\theta_*)=(\weight_0-\mu,\cdots,\weight_n-\mu)$
where $\mu=(n+1)^{-1}\sum \weight_i$. 
The characterization
above implies that $\weight$ is a middle weight for $\theta$ if and only if $\weight-\mu$ is a middle weight for $\theta_{\editD *}$.
Since the middle weight of $\theta_*$  only depends on $\theta_*$, this shows $\theta_*$ determines $\theta$.

 \begin{proposition}\label{characterlemma} Let $\theta:V\to\Aff(n)$ be a marked translation group
 and $\theta_*:V\to\SL(n+1,\RR)$ as above. Then
 $\chi(\theta_*)$ determines $\chi(\theta)$ and vice versa.
 \end{proposition}
 \begin{proof} The characteristic polynomial $c_{\theta}$ is determined by
 $\chi(\theta)$  using (\ref{charpoly}).
The constant term of $c_{\theta}$ determines $\det{\theta}:V\to\RR$, and therefore
$\chi(\theta_*)=\chi(\theta)\left(\det\right)^{-1/n+1}$ is determined. Conversely, given $\chi(\theta_*)$
the  characteristic polynomial $c_{\theta_*}$ is determined by (\ref{charpoly}), and
so the Lie-algebra weights $\{\weight_i:1\le i\le n\}$ of $\theta_*$ are determined. Thus the middle weight
 $\weight$ of $\theta_*$ is  determined by $\chi(\theta_*)$, and $\theta=\exp(-\weight)\theta_*$ has middle weight $0$.
  \end{proof}

\begin{theorem}\label{completeinvt}
If $\theta,\theta':V\to\Aff(n)$ are marked translation groups, then $\cinvt(\theta)=\cinvt(\theta')$ if
and only if $\theta$ and $\theta'$ are conjugate in $\Aff(n)$.
\end{theorem}
\begin{proof} It is clear that the complete invariant is a conjugacy invariant.
We show that if  $\cinvt(\theta)=\cinvt(\theta')$ then $\theta$ and $\theta'$ are conjugate.
By (\ref{charpoly}) $\chi(\theta)$ determines the characteristic polynomial and weights of $\theta$, counted
with multiplicity.
The type of $\theta$ is the maximum over $v\in V$ of the number of eigenvalues of $\theta(v)$
that are not equal to $1$. {\editD This is determined by $\chi(\theta)v$, so}  $\chi(\theta)$ determines $\type(\theta)$. In particular
$\type(\theta)=\type(\theta')$.

The first case is that $\type(\theta)=n$ so $\theta$ is diagonalizable.  Since $\type(\theta')=n$ then
$\theta'$ is also diagonalizable, and therefore semi-simple. The character of a 
semisimple representation determines the representation up to conjugacy, see for example \cite[pp. 650]{Lang}. Hence $\theta$ and $\theta'$ are conjugate in $\GL(n+1,\RR)$.
 This implies they are conjugate in $\Aff(n)$. {\editD If $\type=0$ then the generalized cusps
 are equivalent to cusps in hyperbolic manifolds. It is well known that these are determined by the 
 Euclidean similarity structure
 on the boundary, and hence by $[\beta]$.}

 Now assume $0<\type(\theta)<n$. By (\ref{completeinvt}) every marked translation group
  is conjugate in $\Aff(n)$ to some $\zeta_{\psi}\circ B$
where 
$B\in \SLpm V$ and $\psi\in A_n(\Psi)$. 
 After conjugacies in $\Aff(n)$ we may assume $\theta=\zeta_{\psi}\circ B$ and 
 $\theta'=\zeta_{\psi'}\circ B'$ are both of this form.  
Observe that $\theta$ and $\theta'$
 are conjugate if and only if $\theta\circ (B^{-1})$ and $\theta'\circ (B)^{-1}$ are conjugate. Thus it suffices to assume that  $\theta=\zeta_{\psi}$  and $\theta'=\zeta_{\psi'}\circ B'$. 
 
% Since $\chi(\theta)=\chi(\theta')$ then by (\ref{charpoly}) $\theta$ and $\theta'$ have the same
% Lie-algebra weights.
 {\editA By (\ref{psiformula}) $\psi$ is determined by
 the complete invariant},  hence $\psi=\psi'$, {\editC so $\theta'=\theta\circ B$.
Thus $\eta(\theta')=\eta(\theta)\circ B'$.  We are given that $\eta(\theta)=\eta(\theta')$,} so it follows that $B'\in \O(\cinvt(\theta)).$ Then by Lemma (\ref{Opsi})
 $B'= \mu_{\theta,b}^{-1} P \mu_{\theta,b}$ for some $P\in \O(\Omega,b)$.

 {\bf Claim:} $\theta'=P\theta P^{-1}$.
  Since $\theta'=\theta\circ (\mu_{\theta,b}^{-1} P \mu_{\theta,b})$,
given $v\in V$, and recalling $b\in\bdy\Omega$ is the basepoint, 
and using $\mu_{\theta,b}(v)=(\theta v)(b)$ gives
  $$\theta'(v)=\theta(u),\qquad{\rm where}\quad u=\mu^{-1}_{\theta,b}\left( P((\theta v) b)\right)\in V$$
  Now $P\in\O(\Omega,b)$ fixes the basepoint $b$ so
  $$P((\theta v) b)=P((\theta v) P^{-1} b)= (P(\theta v)P^{-1})(b)$$
  Thus
  $$(\theta u)b=\mu_{\theta,b}(u)=P(\theta(v) b)=(P(\theta v)P^{-1})(b)$$
  Now $\theta(u)$ and $P\theta(v)P^{-1}$ are both in $T(\Omega)$  which acts freely on $\bdy\Omega$.
  Thus $\theta'(v)=\theta(u)=P\theta(v)P^{-1}$, so $\theta'=P\theta P^{-1}$ as claimed.
 \end{proof}

There is an interpretation of the complete invariant  as a geometric structure on the boundary
of a generalized cusp.

\begin{definition}   A {\em cusp geometry}  on a torus $T\cong\RR^{n-1}/\ZZ^{n-1}$ is $(\beta,\Ccal)$
where $\beta$ is a Euclidean metric on $T$ {\editB with volume $1$}, and $\Ccal\subset H^1(T;\RR)\setminus 0$.
The {\em type} of the geometry is $\type=|\Ccal|$.
 \end{definition}

{\editA If $\theta:V\to\Aff(n)$ is a marked translation group then there is a properly convex set $\Omega\subset\RR^n$
that is preserved by $\theta V$ and $C=\Omega/\theta(\ZZ^{n-1})$ is a generalized cusp. Given $b\in\bdy\Omega$ the orbit
map $\mu_{\theta,b}:V\to\bdy\Omega$ is a homeomorphism.
Let $\pi:\Omega\to C$ be projection. 
Then  $\pi_C:=\pi\circ\mu_{\theta,b}:V\to \bdy C$ can be regarded as the universal cover of $\bdy C$.
A cusp geometry   $(\beta,\{\alpha_1,\cdots,\alpha_{\type}\})$ of type $\type=\type(\theta)$ on $\bdy C$ is
defined
as follows.   }

The metric $\beta$ on $\bdy C$ is as defined above. The character $\chi(\theta)$ determines Lie-algebra
weights of
the representation $\weight_i:V\to\RR$ for $1\le i\le \type(\psi)$, and $\alpha_i=[\omega_i]\in H^1(\bdy C;\RR)$ is determined by $\pi^*\omega_i=\weight_i$.

Thus $\omega_i$ is the harmonic representative of the de-Rham class $\alpha_i$. 
Generalized cusps with type $\type<n$ correspond to choices of  non-zero cohomology classes that are orthogonal  with respect to the dual of $\beta$,
and all such cusp geometries are realized by generalized cusps. Those of type $\type=n$ are determined by (\ref{weightseqtn}).
Observe that one can recover the complete invariant from
the cusp geometry.

\begin{proposition} Suppose $\theta_1,\theta_2:V\to\Aff(n)$  are marked translation groups and $C_i=\Omega_i/\theta_i(V)$ 
are corresponding generalized cusps. Then $\theta_1$ and $\theta_2$
are conjugate if and only if there is a map ${\editC f:}\bdy C_1\to\bdy C_2$ 
 that preserves the cusp geometries defined above, and {\editD $f$ is in the correct homotopy class.}
\end{proposition}
 \begin{proof} The existence of  $\editC f$
  implies the two generalized cusps have the same complete invariant. {\editC Then $\theta_1$ and $\theta_2$ are conjugate by 
  (\ref{completeinvt})}.
  {\editC  Conversely, if $\theta_1$ and $\theta_2$ are conjugate, then $C_1$ and $C_2$ are equivalent cusps
  and so have the same cusp geometry.} \end{proof}

\section{New parameters}

In this section we define another family of translation groups in (\ref{lambdakappa}). First we motivate the definition
in dimension $n=4$. The reader may choose to replace $4$ by $n$ in what follows, and introduce
$\cdots$ in the formulae. 

The goal is to construct a connected algebraic family of Lie groups that give conjugates of
all the translation groups $\Tr(\psi)$, and such that
 the diagonalizable ones are dense. Recall that $\type=n$ is diagonalizable, and $\type<n$ is non-diagonalizable.
 
Refer to  (\ref{completepsi}) for the following discussion. {\editC If we reparameterize 
  $\zeta_{\psi}$  in the diagonal case using $t_i=\sqrt{\psi_i}v_i$ 
  then  $\beta(\zeta_{\ppsi})(t)=\|t\|^2+\delta^2$ where $\delta=\psi_n^{-1/2}\sum_{i=1}^{n-1}\sqrt{\psi_i}t_i$.
% On the subspace of $A_n^u(\Psi)$ 
When $\psi_n=\max_i\psi_i$ then $|\delta|\le n\|t\|$, so $\beta$ varies in a compact subset of $\PP\Sym^2V$. Hence, if the character remains bounded along a sequence in this subspace, there is a subsequence for which
 the complete invariants converge. Then, after a suitable conjugacy, the limit  should be a marked generalized cusp of smaller type.}
%   The character now involves terms
% $\exp(t_i/\sqrt{\psi_i})$, which suggests deforming so that some $\psi_i\to\infty$, instead of, as one might naively expect, $\psi_i\to %0$.
%Observe that, if $\psi_n\to\infty$ and $\psi_i/\psi_n\to0$, then the complete
% invariants of diagonalizable cusps converge to the complete invariants of non-diagonalizable generalized cusps.  
 To obtain an algebraic family
 set $\psi_i=1/\lambda_i^2$, then $v_i=\lambda_it_i$. 
The  diagonal group $\Tr(\psi)$  consists of the matrices
 $\exp(M)$, for those $M$ shown below, satisfying (\ref{lambdaeqtn}).
 \begin{equation}\label{Pconj}
M=\begin{pmatrix} \lambda_1t_1 & 0 & 0 & 0 & 0\\
0 & \lambda_2t_2 & 0 & 0 & 0\\
0 & 0 & \lambda_3t_3 & 0 & 0\\
0 & 0 & 0 & \lambda_4t_4 & 0\\
0 & 0 & 0 & 0 & 0\\
\end{pmatrix},\qquad P=\begin{pmatrix} 1 & -\lambda_2^{-1} & -\lambda_3^{-1} & -\lambda_4^{-1} & \lambda_1^{-2}\\
0 & 1 & 0 & 0 & \lambda_2^{-1}\\\
0 & 0 &1 & 0 & \lambda_3^{-1}\\
0 & 0 & 0 & 1 & \lambda_4^{-1}\\
0 & 0 & 0 & 0 & 1\\
\end{pmatrix}
\end{equation}
%Then $\zeta_{\psi}$ is the restriction of $\exp(\psi_n\Diag(v_1,\cdots,v_n))$ to the subspace defined by
 \begin{equation}\label{lambdaeqtn}
0=\sum\psi_iv_i=\sum (1/\lambda_i^2)(\lambda_i t_i)=\sum\lambda_i^{-1}t_i
\end{equation}

The orbits flatten in the directions for which  $\lambda_i\to 0$. To prevent this, conjugate $M$
 by the matrix $P$ in (\ref{Pconj}) to get:
\begin{equation}\label{newmatrix}R:=P^{-1}MP=
\begin{pmatrix} 0 & t_2 & t_3 & t_4 & 0\\
0 & \lambda_2t_2 & 0 & 0 & t_2\\
0 & 0 &\lambda_3 t_3 & 0 & t_3\\
0 & 0 & 0 & \lambda_4t_4 & t_4\\
0 & 0 & 0 & 0 & 0\\
\end{pmatrix}
+
\lambda_1t_1\begin{pmatrix}
1 & -\lambda_2^{-1} & -\lambda_3^{-1} &-\lambda_4^{-1} & 0\\
0 & 0 & 0 & 0 & 0\\
0 & 0 & 0 & 0 & 0\\
0 & 0 & 0 & 0 & 0\\
\end{pmatrix}
\end{equation}
Since $\psi_i$ decreases with $i$, it follows that $\lambda_i$ increases with $i$. We want 
this new family to contain only polynomials (rather than rational functions) in the parameters, so that they are defined whenever 
\begin{equation}\label{weylch}
0\le \lambda_1\le\lambda_2\le\lambda_3\le\lambda_4
\end{equation}
To do this we
introduce extra parameters  $\kappa_i$ for $2\le i\le 4$, and require 
\begin{equation}\label{lambdakappaeqtn}\lambda_i\kappa_i=\lambda_1\end{equation} 
%These new parameters record the direction
%of approach as $\lambda_1\to 0$. 
%In the language of algebraic-geometry, we are blowing up the subvariety 
%$\lambda_1=0$.
then
$$R=
\begin{pmatrix} 0 & t_2 & t_3 & t_4 & 0\\
0 & \lambda_2t_2 & 0 & 0 & t_2\\
0 & 0 &\lambda_3 t_3 & 0 & t_3\\
0 & 0 & 0 & \lambda_4t_4 & t_4\\
0 & 0 & 0 & 0 & 0\\
\end{pmatrix}
+
t_1\begin{pmatrix}
\lambda_1 & -\kappa_2 & -\kappa_3 &-\kappa_4 & 0\\
0 & 0 & 0 & 0 & 0\\
0 & 0 & 0 & 0 & 0\\
0 & 0 & 0 & 0 & 0\\
\end{pmatrix}
$$
 Using (\ref{lambdaeqtn}) we replace $t_1$ by $$t_1=-\lambda_1\left(\lambda_2^{-1}t_2+\lambda_3^{-1}t_3+\lambda_4^{-1}t_4\right)=-\left(\kappa_2 t_2+\kappa_3t_3+\kappa_4t_4\right)$$
and this gives a family of representations $$\Phi_{\lambda,\kappa}:\RR^3\to\Aff(4),\qquad \Phi_{\lambda,\kappa}(t_2,t_3,t_4)=\exp R$$ parameterized by those $(\lambda,\kappa)$ satisfying (\ref{weylch})
and (\ref{lambdakappaeqtn}). When $\lambda_1>0$ then $\kappa_i=\lambda_1/\lambda_i\in[0,1]$ so
$\lambda$ determines $\kappa\in[0,1]^3$. We will see that the conjugacy class of the image group
only depends on $\lambda$. Thus the same collection of conjugacy
classes of groups is obtained by restricting to $\kappa_i\in[0,1]$. Restricting $\kappa$ to a compact set
 helps later with  the point-set topology,
when we quotient out by this compact set.
Finally, since $t_1$ is expressed in terms of the other $t_i$, the
terms for $i=1$ are different to the other terms. 
Thus we replace the index set $1\le i\le 4$ by $0\le i\le 3$, to emphasize the special role of $\lambda_0$.
This leads to the following definitions.

\gap
%Let $\{e_1,\cdots,e_{n}\}$ be the standard basis of  $\RR^{n}$. 
Given $\lambda\in\Hom(\RR^{n},\RR)$ define $\lambda_{i-1}=\lambda (e_i)$.  The subspace
\begin{equation}\label{weylchamber}
A_n = \{(\lambda_0, ..., \lambda_{n-1})\  |\  0 \leq \lambda _0 \leq  \lambda_1 \leq \lambda_2 \leq \cdots \leq \lambda_{n-1}\}\subset\RR^{n}
\end{equation}
is called the {\em (closed) Weyl chamber}. It is a fundamental domain for the action by signed coordinate permutations on $\RR^n$. Observe that $\lambda_i=0$ if only if $\type<n$ and $i\le \ur(\lambda)$. 
%{\purple perhaps here use $A_n( \lambda)$ and say we write $\lambda$ or $\psi$ to distinguish the ordering}

The {\em blown up Weyl chamber} is
\begin{equation}\label{blownupweyl} \widetilde A_n =  \{ (\lambda, \kap) \in A_n \times [0,1]^{n-1}\ : \ \lambda_0 = \lambda _i \kappa_i \}\end{equation}
The projections $p_1 : \widetilde A_n \to A_n$ and $p_2:\widetilde A_n\to[0,1]^{n-1}$ are defined by $p_1 (\lambda, \kappa) = \lambda$ 
and $p_2(\lambda,\kappa)=\kappa$. Since $\lambda_i\ge\lambda_0$ it follows that $p_1$ is surjective.
When $\lambda_i\ne 0$ then $\kappa_i=\lambda_0/\lambda_i$ is determined 
by $\lambda_i$. 
However when $\lambda_i=0$ then $\lambda_0=0$ also, thus
 $\kappa_i\in[0,1]$ is arbitrary. One may regard  $\widetilde A_n$ as obtained from $A_n$ by a kind of {\em blowup}
 of the subset of $A_n$ where $\lambda_0=0$, and the $\kappa$ coordinates record certain tangent directions when some of the coordinates of $\lambda$ are zero.
 
 We make frequent use of the following {\em inverse function theorem}.
  
  \begin{lemma}\cite[Corollary 10.1.6]{Geoghegan}\label{prop_homeo} Let $f: X \to Y$ be a continuous bijection between locally compact spaces.  If $Y$ is Hausdorff and $f$ is a proper map, then $f$ is a homeomorphism.  
\end{lemma} 

{\editC Let  $D_n=\{ (\lambda, \kappa) \in (0,\infty)^n \times [0,1]^{n-1}\ : \ \lambda_0 = \lambda _i \kappa_i \}$. A point in $D_n$
determines a {\em diagonalizable} marked translation group via (\ref{lambdakappa}), however the coordinates of $\lambda$
are in  arbitrary order subject only to $\lambda_0=\min\lambda_i$, rather than non-increasing.}

  \begin{lemma}\label{fibers} Given $(\lambda,\kappa)\in\widetilde A_n$ set $\type=\type(\lambda)$ and
  $\ur=\ur(\lambda)$. If $\type=n$ 
   then $p_2(p_1^{-1}\lambda)=\kappa$. If $\type(\lambda)<n$ then
  $p_2(p_1^{-1}\lambda)=[0,1]^{\ur} \times {\bf 0}$ where 
  ${\bf 0}=  (0,\cdots,0)\in[0,1]^{n-1- \ur}$. Moreover
\begin{itemize}
\item[(a)] $\widetilde A_n\subset \cl D_n$ 
\item[(b)] $p_1$  has compact fibers
\item[(c)] $p_1:\widetilde A_n\to A_n$ is a quotient map.
\end{itemize}
   \end{lemma} 
  \begin{proof}  If $\type=n$ then all $\lambda_i>0$
  and $\lambda$ determines $\kappa$. Otherwise $\type<n$ and
   $\lambda_i=0$ if and only if $i\le \ur$.
For $i\ge 1$ then $\kappa_i$ is the set of solutions in $[0,1]$ of
  $0=\lambda_0=\kappa_i\lambda_i$. For $1\le i\le\ur$ then $\lambda_i=0$ and $\kappa_i\in[0,1]$ is arbitrary.
  For $\ur<i\le n-1$ then   $\lambda_i>0$, so
  $\kappa_i=0$. This gives the formula for $p_2(p_1^{-1}\lambda)$, and (b) is an immediate consequence.
  
   For (a), we prove there is a sequence $(\lambda(m),\kappa(m))\in  D_n$ that converges to $(\lambda,\kappa)\in\widetilde A_n$.
   If $\type=n$ then  $(\lambda,\kappa)\in D_n$ so a constant sequence suffices.
   Otherwise $\lambda_0=0$. Since $\kappa\in[0,1]^{n-1}$  there is  a sequence $\kappa(m)\in(0,1]^{n-1}$ that converges to $\kappa$. Now define
   $\lambda_0(m)=m^{-1}$ and $\lambda_j(m)$ by 
  $\lambda_0(m)=\lambda_j(m)\kappa_j(m)$ for $j>0$. {\editD  Then
  $(\lambda(m),\kappa(m))\in D_n$, and converges to $(\lambda,\kappa)$.} When $\lambda_i=0$ for $i\le \ur$ then the coordinates of  $\kappa$ need not be monotonic. {\editC This is where 
  we exploit that there is no ordering requirement for the $\lambda$ coordinates in $D_n$.}
 
  For (c), let $B=\widetilde A_n/\sim$ be the space of fibers of $p_1$ equipped with quotient topology. The map
  $f:B\to A_n$ induced by $p_1$ is a proper continuous bijection. Moreover $A_n$ is compact and Hausdorff.
 Also $B$ is  locally compact because $p_1^{-1}(K)$ is compact whenever $K$ is compact.
 Hence $f$ is a homeomorphism by
  Lemma \ref{prop_homeo}. 
  \end{proof}
  \begin{remark}
  	(c) is where $[0,1]^{n-1}$ is compact is needed. The reader might like to consider what
 $B$ becomes if $[0,1]^n$ is  replaced by $[0,\infty)^n$ in the definition of $\widetilde A_n$.
  \end{remark}

 We now define another family of Lie groups  $T(\lambda,\kappa)$ that varies continuously with $(\lambda,\kappa)\in\widetilde A_n$.
Theorem \ref{psilambda} show that the families of Lie groups  $T(\lambda,\kappa)$ and  $\Tr(\psi)$ are {\em conjugate}.

\begin{definition}\label{lambdakappa}\label{completelambdakappa} For each $(\lambda,\kappa)\in\widetilde A_n\cup D_n$ define 
$\Phi_{\lambda,\kappa}:=\exp\circ\phi_{\lambda,\kappa}:V\to\Aff(n)$ where
$\phi_{\lambda,\kappa}: V \to \mathfrak{aff}(n)$  is given by 
$$\phi_{\lambda,\kappa}(v) = 
\begin{pmatrix}
0 & v_1 & v_2 &  \cdots & v_{n-1} &0 \\
0 & \lambda_1 v_1&0  & \cdots & 0 & v_1\\
\vdots && \ddots &&& \vdots\\
&& && \lambda_{n-1} v_{n-1} & v_{n-1} \\
0 &&&& \cdots &0
\end{pmatrix}
+ \langle v, \kappa\rangle
\begin{pmatrix}
- \lambda_0 & \kap_1 & \cdots & \kap_{n-1} & 0 \\
0 & \cdots &&& 0 \\
\vdots &&&& \vdots \\
&&&&&\\
0 & \cdots &&& 0 \\
\end{pmatrix}
$$ 
and   $v=(v_1,\cdots,v_{n-1})\in V$,  and $\lambda=(\lambda_0,\cdots,\lambda_{n-1})$, and $\kappa=(\kappa_1,\cdots,\kappa_{n-1})$.
Also $\mathfrak{t}({\lambda, \kappa}):=\Image(\phi_{\lambda,\kappa})$ and 
$T(\lambda,\kappa):=\Image(\Phi_{\lambda,\kappa})$. 
  \end{definition} 
{\editD  If $(\lambda,\kappa)\in D_n$ then $\Phi_{\lambda,\kappa}$ is diagonalizable. It follows if 
$(\lambda,\kappa)\in\widetilde{A}_n$ then
 $\Phi_{\lambda,\kappa}$ is the limit of these diagonalizable representations  by (\ref{fibers})a. This fact is exploited to prove that $T(\lambda,\kappa)$ is a translation group. {\editC The proof of the following is routine and in the appendix.}
   
\begin{proposition}\label{phidiag} (a) Given $(\lambda,\kappa)\in D_n$ 
let $\psi_{i}=\lambda_i^{-2}$ for $1\le i\le n-1$ and $\psi_n=\lambda_0^{-2}$. Then there is $Q\in \SL(n+1,\RR)$
and  ${\editC{\mathfrak f}}\in\GL(V)$ given by ${\editC{\mathfrak f}}(v_1,\cdots,v_{n-1})=\lambda_0^2(\lambda_1v_1,\cdots,\lambda_{n-1}v_{n-1})$  such that 
 $Q\Phi_{\lambda,\kappa}Q^{-1}=\zeta_{\psi}\circ {\editC{\mathfrak f}}$, and
$Q T(\lambda,\kappa)Q^{-1}=\Tr(\psi)$ .

(b) $T(\lambda,\kappa)$ is a translation group, that preserves a convex set $\Omega(\lambda,\kappa)\subset\RR^n$
and $\bdy\Omega(\lambda,\kappa)=T(\lambda,\kappa)\cdot 0$.
Also $\eta(\Phi_{{\lambda,\kappa}})=(\chi_{_{\lambda,\kappa}},[\beta'_{\kappa}])$ where 
\begin{align*}
\beta'_{\kappa}= \Id+\kappa\otimes\kappa&=\left[\begin{matrix} 
\begin{matrix}1+\kappa_1^2 & \kappa_1\kappa_2 &\cdots & \kappa_1\kappa_\ur\\
\kappa_2\kappa_1 & 1+\kappa_2^2 & \cdots &\kappa_2\kappa_\ur\\
\vdots & &  & \vdots\\
\kappa_\ur\kappa_1 &\kappa_\ur\kappa_2 &\hdots & 1+\kappa_\ur^2
\end{matrix} & 0\\
0   & \Id_{\rank}
\end{matrix} \right]
\\
\chi_{_{\lambda,\kappa}}(v_1,\cdots,v_{n-1}) & = 
  1  +\exp\left(-\lambda_0\langle\kappa,v\rangle\right)+ \sum_{i=1}^{n-1}\exp(\lambda_iv_i)
 \end{align*}
{\editC Define $\varkappa=(1+\|\kappa\|^2)^{1/(n-1)}$ then $\det\beta'_{\kappa}=\varkappa^{n-1}$ and 
$\beta_{\kappa}=\varkappa^{-1}\beta'_{\kappa}$} is unimodular.
\end{proposition}

 \begin{definition}\label{squareroot} If $Q=\Id+M\in \GL(k,\RR)$ and $M^2=\alpha M$ then the {\em preferred square root of $Q$} is
  $$\widetilde{S}(Q)=\Id+\alpha^{-1}(\sqrt{1+\alpha}-1)M$$
  \end{definition}

This {\em is} a square root  since $(\Id+x M)^2=\Id +(2x+\alpha x^2)M=Q$ when $2x+\alpha x^2=1$.
 If $v\in\RR^k$ then $M=v\otimes v$ has rank $1$ and the condition holds with $\alpha=\|v\|^2$. Moreover, 
 if $Q$ is symmetric and positive definite, then so is $\widetilde{S}$.

 \begin{lemma}\label{Scts} If $\widetilde{S}=\widetilde{S}(\Id+\kappa\otimes\kappa)$ then $\widetilde{S}^{-1}:(V,\beta_0)\to(V,\beta'_{\kappa})$ is an isometry,
 where  $\beta'_\kappa$ is defined in (\ref{phidiag}). Moreover $\widetilde{S}^{-1}$ varies continuously with $\kappa$.
 \end{lemma}

This gives a  re-parameterization of $\Phi_{\lambda,\kappa}$ that  make the horosphere metric standard. 
 \begin{definition}\label{stdrep} %The {\em standardized representation} 
 $\Phi^{\perp}_{\lambda,\kappa}:V\to\Aff(n)$ is given by 
 $\Phi^{\perp}_{\lambda,\kappa}=\Phi_{\lambda,\kappa}\circ \widetilde{S}^{-1}$ where $\widetilde{S}=\widetilde{S}(I+\kappa\otimes\kappa)\in \GL(V)$.
  \end{definition}

If  $\type<n$  then $\kappa_i=0$ whenever $\lambda_i\ne0$, so this re-parameterization does not {\editA change} the character. However if $\type=n$
the character of $\Phi^{\perp}_{\lambda,\kappa}$ is more complicated. Fortunately we will not need an explicit formula for it in this case.
It follows from (\ref{Scts}) that
  \begin{corollary}\label{completestandardized} {\editD The map $\widetilde{A}_n\to\Hom(V,\Aff_n)$ given by 
  $(\lambda,\kappa)\mapsto \Phi^{\perp}_{\lambda,\kappa}$ is continuous.} The complete
invariant of $\Phi^{\perp}_{\lambda,\kappa}$ is given by $\beta(\Phi^{\perp}_{\lambda,\kappa})(v)=\langle v,v\rangle$, and
if $\type(\lambda)<n$ then
   $$\chi(\Phi^{\perp}_{\lambda,\kappa})(v_1,\cdots,v_{n-1})=2+\ur+\sum_{i=\ur+1}^{n-1}\exp(\lambda_iv_i)$$
 \end{corollary}

  The next result shows that the conjugacy classes of the family of groups 
  $\Tr(\psi)$ coincides with the conjugacy classes
  of the groups $T(\lambda,\kappa)$, and that the conjugacy class of $T(\lambda,\kappa)$
  only depends on $\lambda$.  Changing $\kappa$ but keeping $\lambda$ fixed changes
  the conjugacy class of $\Phi_{\lambda,\kappa}$ (as detected by the horosphere metric) 
  without changing the conjugacy class of  $T(\lambda,\kappa)$.

 \begin{theorem}\label{psilambda} Given $(\lambda,\kappa)\in \widetilde A_n$ 
 then $T(\lambda,\kappa)$ is conjugate to $\Tr(\psi)$ in $\Aff(n)$ where $\psi$ is defined as follows.

 Set $\ur=\ur(\lambda)$ and $\type=\type(\lambda)$. 
 When $\type=0$ then $\lambda=0$ and define $\psi=0$. When $\type=n$ define $\psi$ as in (\ref{phidiag}). 
 When $\editC 0<\type<n$
then  $\type+\ur=n-1$ and 
\begin{align*}
{\rm given}\ \ \ \ \ \ &\lambda=(\lambda_0,\cdots,\lambda_{n-1})=(0,\cdots,0,{\editC \lambda_{\ur+1},\cdots\lambda_{\ur+\type}})\in\RR^n\\
{\rm define}\ \ \ \ \  &\psi=(\psi_1,\cdots,\psi_n)=(\lambda_{\ur+1}^{-2},\cdots,\lambda_{\ur+\type}^{-2},0,\cdots,0) \in\RR^n
\end{align*}
  \end{theorem} 
   
 \begin{proof}  When $\type=n$ this follows (\ref{phidiag}). When $\type=0$ then $\zeta_0=\Phi_{0,0}$ and the result follows.
 This leaves the case $1\le \type<n$.
  Define $F,C\in\GL(V)$ by $F(v_1,\cdots,v_{n-1})=(v_{\ur+1},\cdots, v_{\ur+\type},v_1,\cdots,v_{\ur} )$
and $C=F\cdot\Diag({c_1,\cdots,c_{n-1}})$, {\editC where the $c_i$ are determined below.}
 From (\ref{completestandardized}) 
  \begin{align}\label{eq44}
  \beta(\Phi^{\perp}_{\lambda,\kappa}\circ C)(v)\sim\sum_{i=1}^{n-1}{\editC c_i^2}v_i^2
   ,\qquad \chi(\Phi^{\perp}_{\lambda,\kappa}\circ C)(v)=2+\ur+\sum_{i=1}^{\type}\exp({\editC \lambda_{\ur+i} c_i }v_i)
   \end{align}
 By (\ref{completepsi}) 
  $$
\beta(\zeta_{\ppsi}) (v)\sim
\sum_{i=1}^{\type} \psi_iv_i^2+{\editC \psi_{\type}^{-1}}\sum_{i=\type+1}^{n-1}v_{i}^2,\qquad
     \chi(\zeta_{\ppsi}) (v) =
 2+\ur+  \sum_{i=1}^{\type} \exp({\editC \psi_{\type}}v_i)$$
We will now show how to choose $C$ so that $\zeta_{\psi}$ and $\Phi^{\perp}_{\lambda,\kappa}\circ C$ have the same complete
invariant,  then they are conjugate by (\ref{completeinvt}).
The characters are equal if $ \lambda_{\ur+i} c_i =\psi_{\type}$ for $i\le\type$. Now $\psi_i=\lambda_{\ur+i}^{-2}$ when $i\le\type$
thus $c_i=\psi_{\type}/\lambda_{\ur+i}$ for $i\le\type$, hence $c_i^2=\psi_{\type}^2\psi_i$. Then from (\ref{eq44})
$$\beta(\Phi^{\perp}_{\lambda,\kappa}\circ C)(v)\sim\psi_{\type}^2\sum_{i=1}^{\type}\psi_iv_i^2+\sum_{i=\type+1}^{n-1} c_i^2v_i^2$$
For $i>\type$ define $c_i=\sqrt{\psi_{\type}}$ then
$$\beta(\Phi^{\perp}_{\lambda,\kappa}\circ C)(v)\sim\psi_{\type}^2\sum_{i=1}^{\type}\psi_iv_i^2+\sum_{i=\type+1}^{n-1}\psi_{\type}v_i^2\sim\sum_{i=1}^{\type}\psi_iv_i^2+\psi_{\type}^{-1}\sum_{i=\type+1}^{n-1}v_i^2\sim\beta(\zeta_{\ppsi}) (v)$$
\end{proof}
 
It is messy to directly construct a conjugating matrix,  since it varies continuously only when the type 
does not change.
 In general
 the representations
 $\Phi_{\lambda,\kappa}$ and $\Phi_{\lambda,\kappa'}$ are not conjugate if $\kappa\ne\kappa'$ because they have different complete invariants. However:

{\editD \begin{corollary}\label{newrep} If $\theta:V\to\Aff(n) $ is a marked translation group then there are $B,C\in\SLpm V$ and $(\lambda,\kappa)\in\widetilde A_n$ such that $\theta$ is equivalent to $\Phi_{\lambda,\kappa}\circ B$ and to $\Phi_{\lambda,\kappa}^{\perp}\circ C$.
\end{corollary}
\begin{proof} The first claim follows from (\ref{classification})(b) and (\ref{psilambda}) and the second claim from this and  (\ref{stdrep}).
\end{proof}
}

 \begin{corollary}\label{Tconj}  If $s>0$ and $(\lambda,\kappa),(s\cdot\lambda,\kappa')\in\widetilde A_n$ then $T(\lambda,\kappa)=T(s\cdot\lambda,\kappa')$ are conjugate 
 subgroups of $\Aff(n)$. In particular, if $\type(\lambda)<n$ then $T(\lambda,\kappa)$ is conjugate to $T(\lambda,0)$. 
 \end{corollary}
 \begin{proof} By (\ref{psilambda})  $T(\lambda,\kappa)$ and $T(\lambda,\kappa')$ are conjugate. Let $f:V\to V$ be $f(v)=sv$.
 By (\ref{completelambdakappa}) $\chi(\Phi_{s\lambda,\kappa})=\chi(\Phi_{\lambda,\kappa})\circ f$ and 
 $\beta(\Phi_{s\lambda,\kappa})=\beta(\Phi_{\lambda,\kappa})$. Now $\beta(\Phi_{\lambda,\kappa}\circ f)\sim s^2 \beta(\Phi_{\lambda,\kappa})\sim\beta(\Phi_{\lambda,\kappa})$. Thus $\Phi_{s\lambda,\kappa}$ and $\Phi_{\lambda,\kappa}\circ f$ are marked translation groups with the same complete invariant.
 Thus they are conjugate by (\ref{completeinvt}).
  The second statement follows because, if $\type(\lambda)<n$, then $\lambda_0=0$ so $(\lambda,0)\in \widetilde A_n$.
 \end{proof}
 It is interesting that  in the non-diagonalizable case we may choose $\kappa=0$,  and then $\phi_{\lambda,0}$ has a simple form as given in (\ref{lambdakappa}), {\editD however the diagonalizable ones are more complicated.}
 \section{Topology of the Moduli Space}  
 
 { Recall that $\Rep_n$ is the space of conjugacy classes of holonomy representations of marked generalized torus cusps. \editD First we establish that $\Rep_n$ is a quotient of $\widetilde{A}_n\times\SLpm V$, and
 that the complete invariant provides an embedding of $\Rep_n$. We use this to prove
 that the holonomy map is a homeomorphism $\hol:\Modsp_n\to\Rep_n$. Finally we compute the stratification of $\Modsp_n$
 and prove (\ref{notmanifold}).}

It follows from (\ref{classification}) and (\ref{psilambda}) that every marked translation group is conjugate to $\Phi_{\lambda,\kappa}^{\perp}\circ A$ for
some $(\lambda,\kappa)\in\widetilde A_n$ and $A\in \SLpm V$. 
Moreover if $\type(\lambda)<n$ then it suffices to use $\kappa=0$
so $\Phi_{\lambda,0}^{\perp}=\Phi_{\lambda,0}$. 

\begin{lemma}\label{Psicts} The map $\widetilde\Psi:\widetilde{A}_n\times \SLpm V\to\editD\Rep_n$
given by $\widetilde\Psi((\lambda,\kappa),B)=[\Phi_{\lambda,\kappa}^{\perp}\circ B]$ is continuous, and
covers a continuous surjection
\editC$\Psi:{A}_n\times \SLpm V\to\Rep_n$.
\end{lemma}
\begin{proof} Continuity of $\widetilde\Psi$ follows from (\ref{completestandardized}).
To prove $\Psi$ is well defined we must show that $\Phi^{\perp}_{\lambda,\kappa}\circ B$
is conjugate to $\Phi^{\perp}_{\lambda,\kappa'}\circ B$. To do this, it suffices
to show they have the same complete invariant. Clearly it suffices to do this when $B=\Id$.
{\editD This now follows from (\ref{completestandardized}). }
 
 Recall $p_1:\widetilde A_n\to A_n$ and we have $\Psi\circ p_1=\widetilde\Psi$.  If $U\subset\Rep_n$ is open
 then, since  $\widetilde\Psi$ is continuous,
  $\widetilde\Psi^{-1}(U)=p_1^{-1}(\Psi^{-1}(U))$ is open. Since $p_1$ is a quotient map by
  (\ref{fibers})(c), it follows that $\Psi^{-1}(U)$ is open, so $\Psi$ is continuous.
   \end{proof}
   
   \if0 Recall that  a sequence in a topological space is \emph{bounded} if it is contained in some compact subset. Let $C(X,Y)$
   be the set of continuous functions on $X$ with the compact-open topology. Given $x\in X$ the evaluation map
   $C(X,Y)\to Y$ that sends $f$ to $f(x)$ is continuous. Hence if $f_n$ is a bounded sequence
   in  $C(X,Y)$ then $f(x$) is bounded in $Y$.
   \fi

In what follows use $\beta\in\Pcal$ in place of $[\beta]\in\PP\Pcal$.
Recall the compete invariant $\eta: \Rep_n\to \chi(V)\times\Pcal$ and the codomain is {\editD homeomorphic to a subspace of Euclidean space.  In particular a closed subset of the codomain is locally compact.

\begin{lemma}\label{propercompleteinvt} $\cinvt\circ\Psi:A_n\times \SLpm V\to \chi(V)\times\Pcal$ is proper and continuous, {\editD and $X_n=\cinvt(\Rep_n)$ is a closed subset of $  \chi(V)\times\Pcal$, and is homeomorphic to a closed subset of Euclidean space.}
\end{lemma}
\begin{proof}Continuity of $\cinvt\circ\Psi$ follows from (\ref{iotacts}) and (\ref{Psicts}).
 Suppose  $((a_j,\kappa_j),B_j)\in \widetilde{A}_n\times\SLpm V$ 
and  
$$(\chi_j,\beta_j)=\cinvt(\Psi(a_j,B_j))=\cinvt(\Phi_{a_j,\kappa_j}^{\perp}\circ B_j)$$  
 is a bounded sequence in  $\chi(V)\times\Pcal$. 
%By (\ref{completestandardized}) $(\chi_j,\beta_j)$ does not depend on the choice of $\kappa_j$ subject to $(a_j,\kappa_j)\in\widetilde A_n$. 
Then $\beta_j=B_j^tB_j$
 is bounded.  The map  $\theta: \SLpm V\to \SL V$ given by $\theta(B)= B^t B$ is proper, thus $B_j$ is bounded.
 {\editD After passing to a subsequence we may assume $\lim B_j=B\in\SLpm V$.}
 By (\ref{stdrep}) $$\Phi_{a_j,\kappa_j}^{\perp}=\Phi_{a_j,\kappa_j}\circ \widetilde S_j^{-1}$$
where $\kappa_j\in[0,1]^{n-1}$
so $\widetilde S_j=\widetilde{S}(I+\kappa_j\otimes\kappa_j)$  is bounded.
Since the map that sends an element of $ \SLpm V$ to its inverse is proper, $B_j^{-1}$ is bounded.
 Thus $(B_j^{-1}\circ\widetilde S_j)$ is bounded. Also $\chi_j$ is bounded, so
 $$\chi_j\circ(B_j^{-1}\circ\widetilde S_j)=\trace((\Phi_{a_j,\kappa_j}\circ\widetilde S_j^{-1})\circ B_j)\circ (B_j^{-1}\circ\widetilde S_j)=\trace(\Phi_{a_j,\kappa_j})$$
 is bounded.  Let $\mu_j$ be the last component
of $a_j$, then $\mu_j$ is the largest component of $a_j$.
Referring to (\ref{lambdakappa}) we see that $\Phi_{a_j,\kappa_j}(e_{n-1})$ has
an eigenvalue of 
$\exp\mu_j$ in the $(n,n)$ entry and all other eigenvalues equal to 1. 
Since $\trace(\Phi_{a_j,\kappa_j})$ is bounded, and $\mu_j>0$, it follows that $\mu_j$ is bounded.
Thus $a_j$ is bounded. Hence $\cinvt\circ\Psi$ is proper. {\editD After taking a subsequence
$\lim a_j=a$ and $a\in A_n$ because  $A_n$ is a closed subset of $\RR^n$. Thus 
 $\lim (a_j,B_j)=(a,B)\in A_n\times\SLpm V$, and $\lim\eta\circ\Psi(a_j,B_j)= \eta\circ\Psi(a,B)\in\Im\eta$. 
 Thus $\Im\eta\circ\Psi$ is closed in $ \chi(V)\times\Pcal$.
 By (\ref{Psicts}) $\Im\Psi=\Rep_n$ thus
  $\Im\eta\circ\Psi=\eta(\Rep_n)$ is closed}.
\end{proof}

{\editA By (\ref{completeinvt}), if $B,B'\in\SLpm V$ 
then
  $\Phi^{\perp}_{\lambda,\kappa}\circ B$ and $\Phi^{\perp}_{\lambda,\kappa}\circ B'$ represent the same point in $\Rep_n$
 if and only if they have the same complete invariant.
 By definition (\ref{metricdef})  this is equivalent to
$B'\in B\cdot O(\Phi^{\perp}_{\lambda,\kappa})$. Let $\pi:A_n\times \SLpm V\ \to \left(A_n\times \SLpm V\right) /\sim$ be projection, where $(\lambda,B)\sim(\lambda',B')$ if and only if
$\lambda=\lambda'$ and $B'\in B\cdot O(\eta(\Phi^{\perp}_{\lambda,\kappa}))$ for some $\kappa$ with 
$(\lambda,\kappa)\in\widetilde A_n$. It follows there is an injective function  $$\Psi_*: \left(A_n\times \SLpm V\right) /\sim \ \ \longrightarrow  \editC\Rep_n$$  
such that $\Psi=\Psi_*\circ\pi$. Equip the {\editC domain} with the quotient topology, then $\Psi_*$ is continuous by (\ref{Psicts}). Surjectivity of $\Psi_*$  follows from (\ref{classification}) and (\ref{psilambda}). Theorem (\ref{completeinvtvar}) follows from {\editD (\ref{holishomeo}) and}:

\begin{theorem}\label{quotienthomeo}  
$\Psi_*$ is a homeomorphism and  {\editD $\cinvt:\Rep_n\to X_n$ is a  homeomorphism, 
and  $\Rep_n$  is homeomorphic to a closed subset of Euclidean space.}
\end{theorem}
\begin{proof}   
 By (\ref{propercompleteinvt}) $\cinvt\circ\Psi_*: \left(A_n\times \SLpm V\right) /\sim\  \longrightarrow \editD X_n$ is  proper and continuous.
Since $\editD X_n$ is homeomorphic to a closed subspace of Euclidean space, it is Hausdorff and locally compact.
Given  $x=(\lambda,A)\in A_n\times \SLpm V$ there are compact neighborhoods $L\subset A_n$ of $\lambda$ and $K\subset \SLpm V$ of $A$. 
Then $U=L\times\left(O(n-1)\cdot K\right)\subset A_n\times \SLpm V$ is compact
because $O(n-1)$ is compact.
Since  $O(\eta(\Phi^{\perp}_{\lambda,\kappa}))\subset O(n-1)$ it follows that $\pi(U)$ is a compact neighborhood of $\pi x$, thus
$ \left(A_n\times \SLpm V\right) /\sim$ is  locally compact.
Hence $\cinvt\circ\Psi_*$ is an embedding by (\ref{prop_homeo}). It follows that $\cinvt$ is an embedding,
 and $\Psi_*$ is a homeomorphism.
{\editD The second conclusion follows from (\ref{propercompleteinvt}).}
\end{proof} 
}

 In (\ref{classification}) generalized cusps were classified and shown to be equivalent to ones with holonomy 
 in $\Tr(\psi)$  for some 
$\psi\in A_n(\Psi)$.
Recall that $\psi_i=1/\lambda_i^2$ {\editD when $\lambda_i>0$}. In \cite[Theorem 0.2(v)]{BCL}  gives a bijection $\Theta$ that is essentially the same as $\Psi_*$, but the topology on the domain is different.
It follows from the above that, if the {\em reciprocals} of the coordinates of $\psi$ converge suitably,
then the conjugacy class of $\Tr(\psi)$ has a limit that is another translation group.

{\editA
\if0
\begin{proof}[Proof of Theorem (\ref{completeinvtvar})] After (\ref{quotienthomeo}) it only remains to shown $X_n=\Im(\eta)$ is a semi-algebriac set.
By (\ref{propercompleteinvt})  $X_n=(\cinvt\circ\Psi)(A_n\times \SLpm V)$.
Moreover $f:\widetilde A_n\times\GL V\to  \chi(V)\times\Sym^2V$ given by $f((\lambda,\kappa)$,
 and $\cinvt\circ\Psi$ is a polynomial map. Moroever $A_n$
is defined by linear inequalities and is therefore semi-algebraic. Thus the domain is semi-algebraic. By the Tarski-Seidenberg theorem, \cite{MR1633348},
the image of a semi-algebraic
set under a polynomial map is semi-algebraic. 
\end{proof}
\fi
\editB
 {\editA Informally, two generalized cusps are {\em close} if, after shrinking them, they are nearly affine isomorphic.
 It turns out this is equivalent to their holonomies being close up to conjugacy.
  Our definition of {\em marked moduli space} is based on the notion of {\em developing maps}
 as is done in \cite[Sec 1]{CLT2}. Recall $C=(V/\ZZ^{n-1})\times[0,\infty)$ so $\widetilde C=V\times[0,\infty)$ {\editC is the universal cover}.
 
 Let $\widetilde\Modsp_n$ be the space of developing maps $\dev: \widetilde C\to\AA^n$ for marked generalized cusps
 with underlying space $C$. We endow $\widetilde\Modsp_n$ with the {\editD compact-open topology.}
 %smooth weak topology, where a basis consists of maps that are close in $C^{\infty}$ on a compact set. 
 There is an equivalence relation on 
 $\widetilde\Modsp_n$ that is generated by restricting to a smaller cusp, {\em homotopy}, and composition with an affine isomorphism. The quotient space
 is $\Modsp_n$. 
 
 {\editC When $n\le 3$ homotopy implies isotopy for homeomorphisms of   $ T^n$. However when $n\ge 5$,   there are infinitely many
 isotopy classes homotopic to the identity, see \cite[Theorem 4.1]{hatcher}. We have used {\em homotopy} rather than {\em isotopy} in the definition of $\Modsp_n$
 in  order to obtain the following.}

 % {\editB Thus a neighborhood basis of $[\dev]\in\Modsp_n$  consists of sets of the form $U(K,\delta,Z)$, consisting of
  %all $[\dev']$ such there is  $\alpha\in\Aff(n)$, and $h:C\to C$ homotopic to the identity, and a lift $\widetilde h:\widetilde C\to\widetilde C$,
%such that  $\alpha\circ\dev'\circ\widetilde h$ and $\dev$ are 
%$\delta$-close in $ C^{\infty}(Z,\RR^n)$ where $Z\subset\widetilde C$ is a compact set that is disjoint from $V\times[0,K]$.} 

{\editC \begin{theorem}\label{holishomeo}   The holonomy $\hol:\Modsp_n\to\Rep_n$ is a homeomorphism.
 \end{theorem}
 \begin{proof} {\editD First we define $\hol$}.
 Suppose $\dev:\widetilde C\to\AA^n$ is the developing map of a generalized cusp.
Then $g\in\ZZ^{n-1}=\pi_1C$ acts on $\widetilde C=V\times[0,\infty)$ by $g\cdot(v,t)=(v+g,t)$
so
 the extended holonomy $\rho$ can be recovered from $\dev$ using
that {\editC for $x\in\Im (\dev)$ $$(\rho g)(x) = \dev((g,0)+\dev^{-1}(x))$$}
It follows that there is a map $\widetilde\hol:\widetilde\Modsp_n\to\Hom(V,\Aff(n))$. Moreover this formula shows {\editD $\rho$ is determined
by the restriction of $\dev$ to a compact set. Since $\widetilde\Modsp_n$ has the compact-open topology, it follows
that} $\widetilde\hol$ is continuous.
{\editD It is clear that $\rho$ is the holonomy, and is therefore well defined on the equivalence class $[\dev]\in\Modsp_n$.}
Thus $\widetilde\hol$ covers a continuous map $\hol:\Modsp_n\to\Rep_n$.

{\editD Next, we construct an inverse to $\hol$. By (\ref{quotienthomeo}) $\Psi_*$ is a homeomorphism
so we may replace $\Rep_n$ by $\left(A_n\times \SLpm V\right) /\sim$.
Given $\Phi_{\lambda,\kappa}^{\perp}\circ B\in \widetilde{A}_n\times \SLpm V$}, define
 $f=f_{\lambda,\kappa,B}:V\times[0,\infty)\to\RR^n\times{1}=\AA^n$ by
 $$f(v,z)=(\Phi_{\lambda,\kappa}^{\perp}(Bv))(z,0,\cdots,0,\editD 1)\in\AA^n$$
 Observe that $f(V,0)$ is the orbit  of the origin, so $\Im(f)=\Omega(\lambda,\kappa)$ defined in (\ref{phidiag}).
It follows that $f$ is the developing map for a generalized cusp with holonomy $\Phi_{\lambda,\kappa}^{\perp}\circ B$, thus
 $f\in\widetilde\Modsp_n$.
 
 Define $\widetilde F:\widetilde{A}_n\times \SLpm V\to\widetilde\Modsp_n$ by $\widetilde F((\lambda,\kappa),B)=f_{\lambda,\kappa,B}$.
 Clearly $\widetilde F$ is continuous. By properties of the quotient topology, $\widetilde F$ covers a continuous map 
 $F:\left(A_n\times \SLpm V\right) /\sim\ \ \to \Modsp_n$.
%By (\ref{quotienthomeo}), $F$ is a  bijection.  
Since $\hol$ has a continuous inverse $F\circ\Psi_*^{-1}$, it follows  $\hol$ is a homeomorphism.
 \end{proof}

{\editE
\begin{proof}[Proof of (\ref{nottorus})] If $C$ is a torus then $\widetilde C=C$ and the result follows from (\ref{holishomeo}).
It only remains to prove that
the holonomy of $\widetilde C$ uniquely determines the holonomy of $C$.
Now $\rho|\widetilde C$ determines the extended holonomy $\sigma:V\to\Aff^n$.
We claim $\sigma$ determines the rotational part $R:\pi_1C\to\O(n)$ and therefore determines $\rho:\pi_1C\to\AA^n$.
This follows from the observation that $R$
is uniquely  determined by the action of $\rho\pi_1C$ on $\sigma V$ by conjugacy. This in turn
is determined by the action by conjugacy of $\pi_1C$ on $\pi_1\widetilde C$.
\end{proof}
}
In the sequel we will use $\hol$ to identify these two spaces. 
If $\dev$ is the developing map for a generalized cusp with holonomy $\rho$  then we
identify $[\dev]\in \Modsp_n$ 
with  $[\rho]=\hol[\dev]\in \Rep_n$. {\editE It follows from the above that:}

{\editE \begin{theorem}\label{modspquotient}  $\hol^{-1}\circ\Psi^*:\left(A_n\times \SLpm V\right) /\sim \ \ \longrightarrow  \Tcal_n$ is a homeomorphism.
\end{theorem}}

\begin{definition} Given $0\le \type\le n$ the {\em stratum of type $\type$ of $\Modsp_n$} is the subspace  
$\Modsp_n(\type)\subset \Modsp_n$
that consists of all marked cusps with holonomy
conjugate into some $\Tr(\psi)$ with $\type(\psi)= \type$. \end{definition}

The holonomy of a generalized cusp is conjugate to $\zeta_{\ppsi}\circ A$   where $(\ppsi,A)\in A_n\times\SLpm V$. 
The coordinates of $\ppsi$ are ordered.
Below we show each stratum is a manifold by
showing it is the quotient of a smooth manifold by a compact group that acts  freely.
To do this involves enlarging the set of pairs $(\psi,A)$
 by relaxing the ordering and using $\ppsi\in A_n^u(\Psi)$.
 The equivalence relation on $A_n^u(\Psi)\times\SLpm V$ is then given by a free action of the  $\Sigma_{\type}\times O(u)$. 
 This technique can only be employed with individual strata, but not all of
 $\Modsp_n$, since the dimension of  $O(u)$ changes with type. We will see that $\Modsp_n$ is not a manifold with boundary when $n\ge 3$.  The proof
 of the next result actually determines the topology of each stratum.
 
\begin{proposition}\label{strata} For each $0\le\type\le n$ 
the stratum $\Modsp_n(\type)\subset\Modsp_{n}$ is a connected smooth manifold without boundary {\editE  and 
$\dim \Tcal_n(\type)<\dim \Tcal_n(\type+1)$}.
 Moreover $\cl(\Modsp_n(\type))=\cup_{i\le \type}\ \Modsp_n(i)$. If $n\ge 3$ then the fundamental group $\pi_1(\Modsp_n(n))$ is not trivial.
\end{proposition}
\begin{proof} Let
$W_{\type}=(0,\infty)^{\type}\times \SLpm V$.
By (\ref{classification}) there is  a surjective map
$\pi_{\type}:W_{\type}\to\Modsp_n(\type)$ given by $\pi_{\type}(\psi,A)=[\zeta_{\psi'}\circ A]$
where $\psi=(\psi_1,\cdots,\psi_{\type})$ and $\psi'=(\psi_1,\cdots,\psi_{\type},0,\cdots,0)\in A_n$.

The first case  is $\type<n$, so $\type+\ur=n-1$. Recall  $V=D\oplus U$ from (\ref{DplusU})
where $D=\RR^{\type}\oplus 0$ and $U=0\oplus\RR^{\ur}$.  
Let $\Sigma_{\type}\subset O(\type)$ be the subgroup that permutes the coordinates axes of $\RR^{\type}$,
and $G_{\type}=\Sigma_{\type}\oplus O(\ur)\subset \SLpm V$.
There is an action of $\alpha\in G_{\type}$
on $(\psi,A)\in W_{\type}$ given by
$$\alpha (\psi,A)= \left(\sigma^*\psi,\alpha A\right),\qquad {\rm where}\ \alpha=\bpmat \sigma & 0\\ 0 &  B\epmat$$
Here we regard $\psi\in D^*$ and $\sigma^*$ is the dual action. 
The marked translation groups given by $(\psi,A)$ and $\alpha(\psi,A)$ 
are conjugate {\editE because they have the same complete invariant}. 
Thus $\pi(\alpha(\psi,A))=\pi(\psi,A)$. 

We claim that $\pi^{-1}(\psi,A)=G_{\type}\cdot (\psi,A)$. Suppose $\pi(\psi',A')=\pi(\psi,A)$. There is $\sigma\in\Sigma_{\type}$
so that the coordinates of $\sigma\psi$ are non-increasing. Thus we may assume $\psi$ and $\psi'$ both have this property. By
the classification \cite[1.44 \& 0.2(v)]{BCL} it then follows that $\psi=\psi'$ and $A'\in O(\ur)\cdot A$. The claim follows. Hence $\Modsp_n(\type)$ is
homeomorphic to $W_{\type}/G_{\type}$. 
%Since $G_{\type}$ is compact, and acts smoothly and freely on $W_{\type}$, it follows that $W_{\type}/G_{\type}$ is a smooth manifold. 
Moreover the subgroup $O(\ur)$ acts trivially on the first factor of $W_{\type}$, and by left multiplication on the second factor,
so $$\Modsp_{n}(\type)\cong \left[(0,\infty)^{\type}\times\left(O(\ur)\backslash\SLpm V\right)\right]/\Sigma_{\type}$$
Now $O(\ur)\backslash\SLpm V$ is a symmetric space. Since $\Sigma_{\type}$ is finite and acts freely, 
 it follows that $\Modsp_{n}(\type)$
is a manifold.
%and $\pi_1(\Modsp_{n}(\type))$ surjects $\Sigma_{\type}$ if $\ur>0$, and surjects the alternating subgroup if $\ur=0$.

This leaves the case $\type=n$, in which case $\ur=0$. Let $\Mono(V,\Aff(n))\subset\Hom(V,\Aff(n))$ be the subspace
of injective maps.
Define $f:W_n\to\Mono(V,\Aff(n))$ by
$f(\psi,A)=\zeta_{\psi}\circ A$. Then  $f$ is injective and we use it to identify $W_n$ with $Z=f(W_n)$.  Let $\Sigma_n\subset \Aff(n)$
 be the subgroup that permutes the standard basis of $\RR^n$. Then $\Sigma_n$ acts freely by conjugacy
on $\Mono(V,\Aff(n))$. 

{\bf Claim}: this action preserves $Z$. 
We identify $\Sigma_n$ with the group of permutations of $\{1,\cdots,n\}$. Suppose $\sigma\in\Sigma$. If $\sigma(n)=n$
then the action of $\sigma$ on $W_n$ is as above. In particular the subgroup $\Sigma_{n-1}\subset\Sigma_n$ that fixes $n$
preserves $Z$.

Let $\sigma\in\Sigma$ be the transposition $\sigma=(n-1,n)$. Since $\Sigma_{n-1}$ and $\sigma$ generate $\Sigma_n$, it suffices to show that
$\sigma$ preserves $Z$.
{\editC Given $\psi=(\psi_1,\cdots,\psi_n)\in (0,\infty)^n$  then $\sigma\zeta_{\psi}\sigma^{-1}=\zeta_{\psi'}\circ B$
where 
$$B=\bpmat 1 & & & &\\ & 1 &  & &\\
& & \cdots \\
& & & 1 &\\
  -\psi_1/\psi_n & -\psi_2/\psi_n & \cdots & & -\psi_{n-1}/\psi_n
\epmat,\qquad \psi'=(\psi_n/\psi_{n-1})(\psi_1,\cdots,\psi_{n-2},\psi_n,\psi_{n-1})$$
Let $\delta=|\det B|^{1/(n-1)}$ then $\delta^{-1} B\in\SLpm V$, and $\zeta_{\delta\psi'}(v)=\zeta_{\psi'}(\delta v)$ 
 by (\ref{classification})a,
thus $\zeta_{\psi'}\circ B=f(\delta\psi',\delta^{-1} B)\in Z$.} This proves the claim.

If two elements of $Z$ are conjugate, then they are conjugate by an element of $\Sigma_n$.
{\editC This is because both representations are diagonal, so a conjugacy must preserve the coordinate axes. Thus
the conjugacy is by a signed permutation matrix. However a signed permutation matrix
 is the product of a permutation matrix and a 
diagonal matrix with $\pm1$ on the diagonal. But diagonal matrices centralize these representations,
so they are conjugate via a coordinate permutation.}
 
 Hence $\Modsp_n(n)\cong Z/\Sigma_n$.
Now $W_n$ has two components, and these are swapped by every odd element of $\Sigma_n$. Thus
$$\Modsp_{n}(n)\cong \left((0,\infty)^n\times\SL V\right)/\operatorname{Alt}$$ where $\operatorname{Alt}\subset\Sigma_n$ is the alternating subgroup. 
In particular $\pi_1\Modsp_{n}(n)$ surjects to $\operatorname{Alt}$, and $\operatorname{Alt}$ is non-trivial if $n\ge 3$,
the last claim in the theorem follows.
 {\editC Moreover $\cl(\Modsp_n(\type))=\cup_{i\le \type}\ \Modsp_n(i)$ follows from the corresponding fact for {\editE the Weyl chamber} $A_n$. 
Finally $\dim \Modsp_{\type}=\dim W_{\type}-\dim G_{\type}=({\type}+\dim\SL V) -\dim \O(\ur)$
and $\dim V=n-1$.
}\end{proof}
}

{\editB
\begin{proof}[Proof of (\ref{notmanifold})] There is a deformation retraction $\Modsp_n\to \Modsp_n(0)$ given by scaling $\lambda$, and $\Modsp(0)\cong\Pcal$
is homeomorphic to Euclidean space of dimension $n(n-1)/2$. Thus $\Modsp_n$ is contractible.
 In \cite[Prop 6.2]{BCL}  $\Modsp_2$ was parameterized as $\{(x,y)\in\RR^2: 0\le x\le y\}$ and is thus
a manifold with boundary. 
Suppose $\Modsp_n$ is a manifold $M$ with boundary and $n\ge 3$. Let $\mathcal{N} \subset\Modsp_n$ be the subspace of non-diagonalizable generalized cusps.

We claim that $\bdy M=\mathcal{N}$. Since $Y=\Modsp_n\setminus \mathcal{N}$ is the stratum of diagonalizable generalized cusps, it follow from (\ref{strata}) 
that $Y$ is a manifold without boundary, and $\dim Y=\dim\Modsp_n$ so $\bdy M\cap Y=\emptyset$. Thus $\bdy M\subset\mathcal N$. If $\lambda\in\bdy A_n$ and
$\type(\lambda)=n-1$ then $\lambda$ has exactly one zero coordinate. Let $Z\subset \mathcal{N}$ be the  subset of $[\rho]$ with $\rho=\Phi_{\lambda,\kappa}\circ A$
with $\type(\lambda)=n-1$ and all the coordinates of $\lambda$ are distinct. Then no element of $\Sigma_{n-1}$ fixes $[\rho]$ because if $\sigma\in\Sigma_{n-1}$ and
$\sigma\lambda=\lambda$ then $\sigma=\Id$. It follows a neighborhood of $[\rho]$ in $M$ is homeomorphic to a neighborhood $U$ 
of a point in $A_n\times\SLpm V$ that projects to $U$. But $\rho$ is in the boundary of $A_n\times\SLpm V$ so $[\rho]$ is in the boundary of the quotient.
Thus $Z\subset\bdy M$. But $Z$ is dense in $\mathcal{N}$ and $\bdy M$ is closed in $M$ so $\mathcal{N}\subset\bdy M$. This proves the claim.

Since $M$ is contractible $\pi_1M=0$. Also $\pi_1M=\pi_1 Y$ because a manifold and its interior have the same fundamental group. By (\ref{strata})
 $\pi_1Y\ne 1$ when  $n\ge 3$. This contradicts that $\Modsp_n$ is a manifold.
\end{proof}
}

  \if0
  Observe that
$\Modsp_n(0)\cong \SO(n-1)\backslash\SL(n-1,\RR)$ is the space of standard cusps, thought of as
 the  space of marked conformal classes of Euclidean tori of dimension $(n-1)$.

\begin{theorem} There is a map $\pi:\Modsp_n\to\Modsp_n(0)$ that is a trivial bundle, and 
$\Modsp_n(0)$
is the space $SO(n-1)\backslash \SL(n-1,\RR)$ of Euclidean similarity structures on $T^{n-1}$. The fiber of $\pi$ is contractible.
\end{theorem}
\begin{proof} There is a deformation retraction $\Modsp_n\to \Modsp_n(0)$ given by scaling $\lambda$. The KAN decomposition is $\SL (n-1,\RR)=\SO(n-1)\times \Bcal$
where $\Bcal$ is the Borel subgroup of upper triangular matrices with determinant $1$ and positive entries on the diagonal. Thus
$\Modsp_n(0)\cong \Bcal\cong \RR^{\dim \Bcal}$. There is an action of $B\in\Bcal$ on $\Modsp(V)$ given by $[\theta]\mapsto[\theta\circ B]$.
This action preserves fibers of $\pi$ and shows that $\pi$ is a trivial bundle.
\end{proof}
\fi

\section{The weights data $\nu$}
In this section we prove  (\ref{weightspace}).
 {\editE There is an action of $A\in\GL V$ on $\Rep_n$ given by $A\cdot[\rho]=[\rho\circ A^{-1}]$.
If $\nu(\rho)=([\xi_1,\cdots,\xi_n],[\beta])$, then 
\begin{equation}\label{nuaction}\nu(\rho\circ A^{-1})=([\xi_1\circ A^{-1},\cdots,\xi_n\circ A^{-1}],[\beta\circ A^{-1}])\end{equation}}
This restricts to an action on $\Im\nu$ that covers a transitive action on $\Pos$.

\begin{lemma}\label{rootslambda}  If $\rho=\Phi_{\lambda,\kappa}\circ A$ with $A\in \SLpm V$  then  $\nu(\rho)=([\xi_0,\cdots,\xi_{n-1}],\beta)$
where
$$\langle \xi_i,\xi_j\rangle_{_{\beta^*}}={\editC\varkappa}\lambda_i^2\delta_{ij}-\varpi={\editC\varkappa}\lambda_i\lambda_j\delta_{ij}-\varpi$$ {\editE and $\varkappa=(1+\|\kappa\|^2)^{1/(n-1)}$}  and $\varpi={\editB \lambda_0^2{\editC\varkappa}^{\editC (2-n)}}$.  \end{lemma}
\begin{proof}
  Let $\beta=\beta(\rho)$, and $\langle\cdot,\cdot\rangle_{_\beta}$ is the inner 
 product on $V$ corresponding to $\beta\in \Sym^2V$, and $\|\cdot\|_{_\beta}$ the associated norm. 
 Let $\|\cdot\|$ be the standard norm on $V$ for which the standard basis is orthonormal and $\langle\cdot,\cdot\rangle$ the associated inner product.
 We may assume $\rho=\Phi_{\lambda,\kappa}$.
 Then the matrix of $\beta$ in the standard basis is given $\editC Q={\editC\varkappa}^{-1}(\Id+\kappa\otimes\kappa)$.
   The matrix of the dual form $\beta^*$ on $V^*$ with respect to the dual basis is then $\editC Q^{-1}$.

From (\ref{lambdakappa}) the Lie algebra weights for $\rho$ are $\xi_0,\cdots,\xi_{n-1}\in V^*$ where
 \begin{equation}\label{xi0}
 \xi_0(v)=-\lambda_0 \langle \kappa,v\rangle,\qquad \dxi_i=\lambda_ie_i^*\ \ for\ \ 1\le i\le n-1\end{equation}
 For the following, refer to the discussion after (\ref{squareroot}). Now $Q=\editC {\editC\varkappa}^{-1}(\Id+M)$ where $M=\kappa\otimes\kappa$, then $M^2=\|\kappa\|^2 M$, so $Q^{-1}=\editC{\editC\varkappa}(\Id+M)^{-1}={\editC\varkappa}(\Id-(1+\|\kappa\|^2)^{-1}M)$.

If $1\le i,j\le n-1$ then 
\begin{equation}\label{weightseq}
\langle \dxi_i,\dxi_j\rangle_{_{\beta^*}}=\langle \lambda_ie_i^*,\lambda_je_j^* \rangle_{_{\beta^*}}={\editB{\editC\varkappa}}\lambda_i\lambda_j\left( \delta_{ij}-(1+\|\kappa\|^2)^{-1}\kappa_i\kappa_j\right)={\editB{\editC\varkappa}}\lambda_i\lambda_j\delta_{ij}-\varpi
\end{equation}
Now $\lambda_i\kappa_i=\lambda_0$ so $\varpi ={\editB{\editC\varkappa}}\lambda_0^2(1+\|\kappa\|^2)^{-1}=\editC\lambda_0^2{\editC\varkappa}{\editC\varkappa}^{1-n}$. 

We claim  (\ref{weightseq}) holds in all cases: $0\le i,j\le n-1$. If $\lambda_0=0$ then $\xi_0=0$ and (\ref{weightseq}) holds in all cases.
 Otherwise $\lambda_0>0$ and using  $\lambda_i\kappa_i=\lambda_0$ then (\ref{xi0}) implies 
 \begin{equation}\label{xieqtn}
 \xi_0=-\lambda_0^2\sum_{i=1}^{n-1}\lambda_i^{-2}\dxi_i,\qquad {\rm so}\qquad  \sum_{i=0}^{n-1}\lambda_i^{-2}\dxi_i=0
 \end{equation}
To compute $\langle \xi_0,\dxi_j\rangle_{{\beta^*}}$, replace $\xi_0$ by the above and then use  (\ref{weightseq}) in the  case $i,j\ge 1$ already established. Some algebra 
 then shows  (\ref{weightseq}) holds in all cases. 
\end{proof}

{\editD The lemma implies the inner product of distinct weights is always $-\varpi$. This has a geometric interpretation.
Consider a set of $n$ vectors $\{v_1,\cdots,v_n\}$ in $V=\RR^{n-1}$ equipped with the standard inner product
such that for some $\varpi\ge0$  the vectors satisfy the equation 
\begin{equation}\label{sigmaeqtn}\forall\ i\ne j\qquad \langle v_i,v_j\rangle=-\varpi\end{equation}
If $\varpi=0$ this just says the vectors are pairwise orthogonal, and for dimension reasons at least one is zero. If $\varpi>0$
then set  $\RR^n=V\oplus\RR e_n$ with the standard inner product. The equations (\ref{sigmaeqtn}) are equivalent to
the pairwise orthogonality of the vectors $\{u_i=v_i+\sqrt{\varpi}e_n\}$ in $\RR^n$. In this case
the $\{u_i\}$ are an orthogonal basis of $\RR^n$ that represent points in the hyperplane $x_n=\sqrt{\varpi}$, and
the $\{v_i\}$ are the images of these vectors under orthogonal projection into $V$.} 

 \begin{proof}[Proof of Theorem (\ref{weightspace})] We  will abuse notation by identifying $\Modsp_n=\Rep_n$ and
  write $\rho$ instead of $[\rho]$ 
 for a point  in $\Rep_n$. 
Suppose $\nu(\rho)=([\xi_0,\cdots,\xi_{n-1}],\beta)$. Then 
 $\chi=\chi_{\rho}:V\to\RR$ is given by $\chi(v)=\sum_{i=1}^n\exp\dxi_i v$. 
 Thus the complete invariant $\eta(\rho)=(\chi,\beta)$ is  a continuous function of $\nu(\rho)$.
 {\editE 
 By  (\ref{completeinvt}) $\eta$ is injective hence $\nu$ is injective.}

Recall $\Rcal_n\subset\SP^n(V)\times\Pos$ is the subset of all $([\xi_0,\cdots,\xi_{n-1}],\beta)$ such that 
\begin{equation}\label{Weqtn}
\exists\ \varpi\ge 0\qquad \forall\ i\ne j\quad\langle \xi_i,\xi_j\rangle_{_{\beta^{\editD*}}}=-\varpi
\end{equation}
We must show that $\Im(\nu)=\Rcal_n$. By (\ref{rootslambda})  $\Im \nu\subset \Rcal_n$.   It remains show that $\Rcal_n\subset\Im\nu$. 
    In what follows we will always choose an ordering  for $x=([\xi_0,\cdots,\xi_{n-1}],\beta)\in\Rcal_n$ so that $\beta^*(\xi_i)$ is a non-decreasing function of $i$.
Define  $k$ by $\xi_i\ne0$ if and only
 if $i\ge k$, and define $\varpi=-\langle \xi_0,\xi_1\rangle_{_{\beta^*}}$.

\textbf{Case 1:} $\varpi=0$.  Then (\ref{Weqtn}) is equivalent
 to requiring the $\xi_i$ are pairwise orthogonal with respect to $\beta^{\editD *}$. Since $\dim V=n-1$ it follows that $\xi_0=0$.
 Define 
 \begin{equation}\editD\label{lambdadef}\lambda_{i}=\sqrt{\beta^*(\xi_i)},\qquad\kappa=(0,\cdots,0)
 \end{equation} then $\lambda_0=0$. {\editD Observe
 these values are consistent with (\ref{rootslambda}).} From (\ref{completelambdakappa})
 the weight data is
 \begin{equation}\label{eq37}
 \nu(\Phi_{\lambda,0})=([0,\cdots,0,\lambda_ke_{k}^*,\cdots,\lambda_{n-1}e_{n-1}^*],\beta_0),\qquad \beta_0(x)=\|x\|^2
 \end{equation}  %where $\beta_0$ is defined in (\ref{completelambdakappa}).
 
  {\editD Now $\beta_0^*(\lambda_ie_i^*)=\lambda_i^2=\beta^*(\xi_i)$. Then $\lambda_ie_i^*$ are obviously pairwise $\beta_0^*$-orthogonal, and $\xi_i$ 
 are  pairwise  
 $\beta^*$-orthogonal since $\varpi=0$. Thus there is
 an isometry $A:(V,\beta)\rightarrow (V,\beta_0)$ with $$(\lambda_ie_i^*)\circ A^{-1}=A^*(\lambda_ie_i^*)=\xi_i,\qquad \beta=\beta_0\circ A^{-1}$$ 
Then  $x\in\Im\nu$ because
 applying   (\ref{nuaction}) to  (\ref{eq37}) gives 
$$\nu(\Phi_{\lambda,0}\circ A^{-1})=([\xi_0,\cdots,\xi_{n-1}],\beta)=x$$}
\if0Let $L_{\beta}:(V,\beta)\to (V^*,\beta^*)$ be the  natural  isometry given by
 $(L_{_{\beta}}v)w=\langle v,w\rangle_{_{\beta}}$ and let $r_i=L_{\beta}^{-1}\xi_i$ {\editD be the vectors dual to the weights}. 
Recall $V=D\oplus U$ is an orthogonal decomposition with respect to $\beta$.
  Then ${\editD D}=\langle r_k,\cdots, r_{n-1}\rangle$ and $\beta|{\editD D}=\beta'|{\editD D}$, 
  because $\{ r_k,\cdots, r_{n-1}\}$ is
 an orthogonal basis of ${\editD D}$  for both inner products, and 
  $\beta'(r_i)=\|A^{-1}r_i\|^2$ and
  $\beta(r_i)={\editD \beta^*(\xi_i)}=\lambda_i^2$ by (\ref{lambdadef}). Thus  the identity on ${\editD D}$ extends to an isometry
 from $\beta$ to $\beta'$. Composing $A$ with this isometry gives a new choice for $A$  and proves the claim.
 \fi

\textbf{Case 2:} $\varpi>0$.  Identify $V$ with the subspace of $\RR^n$ where $x_n=0$, and extend 
$\beta$ to $\RR^n$  so that $\beta(e_n)=1$ and $e_n$ is orthogonal to $V$. 
Let $L_{\beta}:(V,\beta)\to (V^*,\beta^*)$ be the  natural  isometry given by
 $(L_{_{\beta}}v)w=\langle v,w\rangle_{_{\beta}}$ and let $r_i=L_{\beta}^{-1}\xi_i$ {\editD be the vectors dual to the weights}. Then (\ref{Weqtn}) is equivalent 
to the pairwise orthogonality of the vectors $$\{u_i=r_i+\sqrt\varpi e_n:0\le i\le n-1\}\subset\RR^n$$  
Since $\|u_i\|\ge \varpi>0$  this
is a basis of $\RR^n$. Moreover $\langle e_n,u_i\rangle =\sqrt{\varpi}$. {\editD Writing $e_n$ in terms of this orthogonal basis}
  $e_n=\sum_{i=0}^{n-1}\mu_iu_i$ with  $\mu_i=\sqrt\varpi/\|u_i\|^2_{_{\beta}}>0$.
Thus {\editC $\sum\mu_ir_i=0$  and $\xi_i=L_{\beta}(r_i)$ so $\sum\mu_i\dxi_i=0$}. Set $\lambda_i=\mu_i^2$ and $\kappa_i=\lambda_0/\lambda_i$ then 
$\sum\lambda_i^{-2}\dxi_i=0$. {\editD Define}
$$([\xi'_0,\cdots,\xi'_{n-1}],\beta'):=\nu(\Phi_{\lambda,\kappa})$$
Then $\sum\lambda_i^{-2}\dxi_i'=0$ {\editD by (\ref{xieqtn})}, and by (\ref{weightseq}) $\langle \dxi_i',\dxi_j'\rangle_{_{\beta{'^*}}}={\editD {\editC\varkappa}}\lambda_i^2\delta_{ij}-\varpi$. 
Now $\dxi_i'=\lambda_ie_i^*$ for $i>0$, so in particular $\{\dxi_i':1\le i\le n-1\}$ is a basis of $V^*$. There is a unique $A\in \GL V$
such that $A^*\dxi_i'=\dxi_i$ for $i\ge 1$. Since $\xi_0'=-\lambda_0^{2}\sum_{i=1}^{n-1}\lambda_i^{-2}\dxi_i'$ and
$\xi_0=-\lambda_0^{2}\sum_{i=1}^{n-1}\lambda_i^{-2}\dxi_i$ it follows that $A^*\xi_0'=\xi_0$. Now
$$\langle\dxi_i',\dxi_j'\rangle_{_{\beta'^*}}={\editD {\editC\varkappa}\lambda_i^2\delta_{ij}-\varpi}=\langle\dxi_i,\dxi_j\rangle_{_{\beta^*}}$$
and it follows that $A^*$ is an isometry between the metrics $(\beta')^*$ and $\beta^*$ on $V^*$. 
Thus $\nu$ is surjective.

{\editE We have shown that $\nu$ is a  bijection. 
Let $\Upsilon=\eta\circ\nu^{-1}:\Rcal_n\to X_n$. By  (\ref{quotienthomeo}) $\eta$ is a homeomorphism, so
 $\Upsilon$ is a bijection. Above we showed that $\eta(\rho)$ is a continuous function of $\nu(\rho)$, and it follows
 that $\Upsilon$ is continuous. 
 
 We claim $\Upsilon$ is proper. Suppose $\nu(\rho_m)=(x_m,\beta_m)$ is unbounded,
and suppose for contradiction that $\Upsilon(\nu(\rho_m))=\eta(\rho_m)=(\chi(\rho_m),\beta_m)$ is bounded. Then there is a component $\xi_{m,i}\in V$ of $x_m=(\xi_{m,0},\cdots,\xi_{m,n-1})\in\SP^n V^*$ that is unbounded.
Thus $\chi(\rho_m)=\sum_i\exp\xi_{m,i}$ is unbounded, a contradiction.
This proves the claim.

By (\ref{propercompleteinvt}) $X_n$ is locally compact,  and $\Rcal_n$  is a closed subset of Euclidean space and thus locally compact. By (\ref{prop_homeo}) $\Upsilon$ is a homeomorphism.
Since $\eta$ is a homeorphism it follows that
$\nu=\Upsilon^{-1}\circ\eta$ is a homeomorphism.}
\end{proof}

   \section{Cubic differentials}

In this section we will show that {\editC when $n\ge 3$} a generalized cusp $C\cong T^{n-1}\times[0,\infty)$
 is uniquely determined {\editC up to equivalence by the projective class $[J]$} called the {\em shape invariant} of a certain polynomial $J=q+c$ where $q,c:\RR^{n-1}\to \RR$
are homogeneous polynomials of degree $2$ and $3$ respectively. One may
regard $q$ as a similarity structure (Euclidean structure up to scaling) on $T^{n-1}$,  and $c$ as a cubic differential on $T^{n-1}$.
When $n=2$ then the shape invariant does not determine the cusp, but the moduli space is described in \cite{BCL} Section 6.

 \begin{definition} A {\em calibrated vector space} is a pair $(V,\vartheta)$ where $V$ is a vector space  and
 $\vartheta:V\to\RR$ is a  function, called the {\em calibration}. {\editA A linear isomorphism $f:V\to V'$
  is an {\em isometry} between the calibrated vector spaces $(V,\vartheta)$ and $(V',\vartheta')$
 if   $\vartheta= \vartheta'\circ f$.
 The group of self isometries of $(V,\vartheta)$ is written $\O(\vartheta)$.  Two calibrations $\vartheta,\vartheta'$ are {\em similar}
 if there is $\lambda>0$ with $\vartheta'=\lambda\vartheta$, and this is written $\vartheta\sim\vartheta'$.
}
 \end{definition}

A calibration can be viewed as an interesting generalization of  a norm. For example, 
  there is a calibrated vector space $(\RR^{248},\vartheta)$ with $\vartheta$ an octic polynomial such that the compact form of the exceptional Lie group $E_8$  is the identity component of 
$\O(\vartheta)$, see \cite{Gur}, where they use the term {\em stabilize} instead of {\em isometry}. We follow \cite{Asch} in using the term {\em isometry}. 
%We will see that $\O(\lambda,V)=\O(\phi_{\lambda,\kappa})$ where:

   \begin{definition}\label{cuspspacedef} A {\em cusp-space} is a calibrated vector space, $(V,\vartheta)$,  that
 {\editB is similar}  to some $(\RR^{n-1},\vartheta_{\lambda,\kappa})$ where
 $\vartheta_{\lambda,\kappa}:\RR^{n-1}\to\RR$  is given by 
$$\vartheta_{\lambda,\kappa}(v_1,\cdots,v_{n-1})=\left(\langle v,v\rangle +
 \langle v,\kappa\rangle^2\right)  +\frac{1}{3}\left(-\lambda_0\langle v,\kappa\rangle^3+ \sum_{i=1}^{n-1} \lambda_i v_i^3\right)$$
 and $(\lambda,\kappa)\in\widetilde{A}_n$ and $\langle\cdot,\cdot\rangle$ is the standard inner product on $\RR^{n-1}$. \end{definition}

{\editC In the non-diagonalizable case when $\kappa=0$, this simplifies to
$\vartheta_{\lambda,0}=\langle v,v\rangle   +\frac{1}{3}\sum_{i=1}^{n-1} \lambda_i v_i^3$.}

   \begin{definition}\label{cuspspacestructures}
The {\em space of cusp-space structures} on the vector space $V$ is
 $$\Jcal(V)=\{[\vartheta]\ :\ (V,\vartheta)\ \text{is a cusp-space}\}\subset \PP\left(\Sym^2V\oplus\Sym^3V\right)$$
equipped with the subspace topology, and $\Jcal_n=\Jcal(\RR^{n-1})$. 
 \end{definition}

{\editC 
  If $f:\RR^n\to\RR$ is a smooth function, the {\em $k$-Jet} is the polynomial
given by the truncated Taylor expansion of $f$ around $0$ consisting of all terms of total degree at most $k$.}

\begin{definition}\label{calibratedcusp}  {\editA Suppose $T$ is a translation group, and $W$ is a real vector
space, and  $\theta:W\to T$ is an isomorphism. The {\editE\em shape invariant} for $\theta$ is $[J]$ where
$h$ is a height function for $T$,
and
 $J=J(\theta)$ is the 3-Jet of $h$ at $0$, and $[J]\in\PP(\Sym^2W\oplus\Sym^3 W)$.}\end{definition}

{\editC The height function $h$ is unique up to multiplication by a positive real, thus the projective class $[J]$ of $J$ is well defined. 
Moreover the terms of degree $0$ and $1$ in $J$ vanish, so $J=q+c$ with $q\in\Sym^2W$ and $c\in\Sym^3W$.
{\editE When $W=V$ then $\det q$ is defined using the standard basis of $V$, and $\beta(\theta)=\gamma q$ is unimodular 
where
$\gamma=(\det q)^{{\editD-}1/\dim V}$. We use the map $F:\Jcal(V)\rightarrow\Pcal\oplus\Sym^3V$ given by
$F[q+c]=\gamma(q+c)$} to identify $\Jcal(V)$ with a subspace of $\Pcal\oplus\Sym^3V$. 
}

{\editA It is easy to check that if $B\in\Aff(n)$  then $J(B\theta B^{-1})=J(\theta)$,
and that if $A\in\GL V$ then $J(\theta\circ A)=J(\theta)\circ A$.}
  Consider the diagonal translation subgroup $G=\Tr(\ppsi)$   
   {\editA where $\ppsi=\sum_{i=1}^n\ppsi_ie_i^*$ with all $\ppsi_i>0$ as in  (\ref{psigroup}).} 
 Let $D(n)\subset\GL(n+1,\RR)$ be the subgroup of positive diagonal matrices with $1$ in the bottom right corner.
 Then $G$   is  a codimension-1 subgroup of $D(n)$. To compute the calibration for $\zeta_{\psi}$
  we avoid choosing a basis of the Lie algebra, $\mathfrak g$, of $G$, but instead work with the natural basis of $D(n)$.
  
   Let $\AA=\RR^{n}$ be the  $\RR$-algebra with addition and multiplication defined componentwise,
  so
  $$(a_1,\cdots,a_n)(b_1,\cdots,b_n)=(a_1b_1,\cdots,a_nb_n)$$ 
 This multiplication is called the {\em Hadamard product}. 
 Observe that $p=(1,\cdots,1)$ is the multiplicative identity in $\AA$, and for $n>0$ then
  $a^n\in\AA$ is the element obtained by raising each component of $a$ to power $n$.
Let $\AA_+\subset\AA$ be the subset with all coordinates strictly positive, made into a group using Hadamard multiplication. 
    The map  $\AA\to\mathfrak{gl}(n+1,\RR)$ given by $(x_1,\cdots,x_n)\mapsto\Diag(x_1,\cdots,x_n,0)$ is used to identify the Lie algebra 
    $\AA$ (with zero Lie bracket) to the Lie algebra of
    $D(n)$,  and the group homomorphism $\delta:\AA_+\to \GL(n+1,\RR)$ given by $\delta(x_1,\cdots,x_n)=\Diag(x_1,\cdots,x_n,1)$ identifies  $\AA_+$ (with Hadamard multiplication) to  $D(n)$.
  Regarding $\AA$ as the Lie algebra of $\AA_+$ then  $\exp:\AA\to\AA_+$
is coordinate-wise exponentiation.
    Define an inner product on $\AA$ by
  $$\langle x,y\rangle_{\ppsi} =\psi(xy)=\sum_{i=1}^n \ppsi_ix_iy_i$$
  Then $\langle xy,z\rangle_{\ppsi}=\langle x,yz\rangle_{\ppsi}$ so  $\langle x,y\rangle_{\ppsi}=\langle p,xy\rangle_{\ppsi}$,
   and $$\mathfrak g=\ker\ppsi=p^{\perp}:=\{x\in\AA:\ \langle p,x\rangle_{\ppsi}=0\ \}$$ may be regarded as the Lie algebra of $\Tr(\ppsi)$.
  
   \begin{lemma}\label{diagonalcalibration}     If $\type(\ppsi)=n$ then $\delta\circ\exp:\mathfrak g\to\Tr(\ppsi)$ is a marked translation group, 
 {\editC and the shape invariant is $[J(\delta\circ\exp)]=[J_{\ppsi}]$} where  $$J_{\ppsi}(x)=\ (1/2)\langle p,x^2\rangle_{\ppsi}+(1/6)\langle p,x^3\rangle_{\ppsi}$$\end{lemma} 
\begin{proof} Since $\mathfrak g=\ker\ppsi$
 $$\delta\circ\exp(\mathfrak g)=\{ \Diag(\exp(x_1),\cdots,\exp(x_n),1):\ \sum\ppsi_ix_i=0\ \}=\Tr(\ppsi)$$
 Let $\bdy\Omega\subset\RR^n$ be the orbit of $p$  under $\Tr(\ppsi)$ then the tangent space to $\bdy\Omega$ at $p$ is $p^{\perp}$.
{\editC We use the height function $h=h_\psi \circ\exp:\mathfrak g\to\RR^n$ where $h_\psi(y)=\psi(y)-\ppsi(p)$  then
$$h(x)= -\ppsi(p)+\ppsi(\exp(x))=-\ppsi(p)+\sum_{i=0}^{\infty} \frac{1}{n!}\langle p,x^n\rangle_{\ppsi}$$
The  terms of degree $0$ and $1$ vanish, because $\langle p,x^0\rangle_{\ppsi}=\psi(p)$, and  $\langle p,x\rangle_{\ppsi}=0$ since $x\in p^{\perp}$.}\end{proof}

{\editD The  proof of the following is in the appendix.}
\begin{proposition}\label{calibrationiscuspspace}   If $(\lambda,\kappa)\in A_n$ then
   $[J(\Phi_{\lambda,\kappa})]=[\vartheta_{\lambda,\kappa}]\in  \Jcal(V)$. 
   {\editC Moreover in the diagonalizable case $\lambda_0>0$, and $(V,J(\Phi_{\lambda,\kappa}))$ is similar
   to $(\mathfrak g,J_{\ppsi})$ where $\ppsi$ is determined in the proof.} 
\end{proposition}

{\editD The following lemma shows how $\nu$ determines the calibration.}
The cubic term $c$ in the $3$-Jet $J=q+c$ is  a weighted sum of the cubes of the weights $\xi_i$, see (\ref{ccformula}) below.
Later we will see that one can recover these weights from $[J]$. {\editC See \cite{MR870934}
 and
Theorem (1.4) in \cite{MR3105781} for a uniqueness statement 
concerning the expression of a cubic as a sum of cubes. The proof of the following is in the appendix.}

\begin{lemma}\label{cubicrootdata} {\editC If $n\ge 3$ then} there is a  map  $\Kappa:\Rcal_n\to\editC \Pcal \oplus\Sym^3V$ such that $\Kappa\circ\nu=[J]$ and $\Kappa$ is continuous and proper.
If $\rho$ is a marked translation group and $x=\nu(\rho)=([\xi_0,\cdots,\xi_{n-1}],\beta)\in \Rcal_n$, then 
$\Kappa(x)=\beta(\rho)+c(\rho)$
with 
$$c(\rho)=(1/3)\sum_{i=0}^{n-1}\dxi_i^3\left(\langle \xi_i,\xi_i\rangle_{_{\beta^*}}+\varpi\right)^{-1},\qquad \varpi={\editC-}\langle \xi_1,\xi_2\rangle_{_{\beta^*}}$$
\end{lemma}

\begin{corollary}\label{Jproper} $J:\Modsp_n\to \Pcal \oplus\Sym^3V$ is {\editE continuous and} proper.
\end{corollary}
\begin{proof} By (\ref{holishomeo}) $\hol^{-1}:\Rep_n\rightarrow\Modsp_n$ is a homeomorphism  and by
 (\ref{cubicrootdata}) $\Kappa:\Rcal_n\rightarrow\editC\Pcal \oplus\Sym^3V$ is {\editE continuous and} proper
 and $\nu:\Rep_n\rightarrow\Rcal_n$ is homeomorphism by (\ref{weightspace})
 thus $J=\Kappa\circ\nu\circ\hol^{-1}$ is {\editE continuous and}  proper.\end{proof}

It remains to show that the  {\editC shape invariant} $[J]=[q+c]$ determines a unique generalized cusp. The method used is to show
that the local maxima of the cubic, $c$,
 restricted to the unit sphere of the quadratic, $q$, enable one to determine $\psi$.
This follows from Lemmas (\ref{localmaxlemmadiag}) for the diagonalizable case, and (\ref{localmaxlemmanotdiag}) in the non-diagonalizable case. 
{\editC The proofs are in the appendix}.
%The following shows that in the diagonalizable case that $J$ determines $\ppsi$ up to scaling.  

\begin{lemma}[Diagonalizable case]\label{localmaxlemmadiag} {\editC Assume $n\ge 3$. Given $\ppsi\in A_n$  let $(\RR^n,J_{\ppsi}=q+c)$ be the calibrated vector space 
with $J_{\ppsi}(x)=\ (1/2)\langle p,x^2\rangle_{\ppsi}+(1/6)\langle p,x^3\rangle_{\ppsi}$. Let
$\mathfrak g=\{x\in\RR^n:\langle p,x\rangle_{\psi}=0\}$, and $S=\{ v\in\mathfrak g: \langle v,v\rangle_{\ppsi}=1\}$, and
 $\mathfrak{s}=\sum\ppsi_i$. } For $1\le i\le n$ define
 $v_i=(\mathfrak{s} e_i-\ppsi_ip)/\|\mathfrak{s} e_i-\ppsi_ip\|_{\ppsi}$. Then $$K=\{x\in S: (c|S)\ \text{ has a local maximum at } x\}=\{v_i:1\le i\le n\}$$
{\editC Moreover  $ i\ne j\Rightarrow \alpha_{ij}:=\langle v_i,v_j\rangle_{\ppsi}<0$. If  $1\le i,j,k\le n$ and
 $i,j,k$ are pairwise distinct then
$$\ppsi_i/\mathfrak{s}=\alpha_{ij}\alpha_{ik}/(\alpha_{ij}\alpha_{ik}-\alpha_{jk}),\qquad\qquad 6c(v_i)=\frac{1}{\sqrt{\ppsi_i}}\frac{1-2\ppsi_i/\mathfrak{s}}{\sqrt{1-\ppsi_i/\mathfrak{s}}}$$
Also $|K^+|\ge n-1$ where $K^+=\{v\in K:\ c(v)>0\ \}$.
}
\end{lemma}

For the corresponding result in the non-diagonalizable case, it is more convenient to work with $\editC\Psi_{\lambda,0}$ instead of $\zeta_{\ppsi}$,
since the calibration is $J=\|v\|^2+(1/3)\sum\lambda_iv_i^3$.
\begin{lemma}[non-diagonalizable case]\label{localmaxlemmanotdiag}  Given $\lambda=({\editC 0,}\lambda_1,\cdots,\lambda_{n-1})\editC \in A_n$, {\editC let $J(v)=\|v\|^2+c(v)$ where} $c=(1/3)\sum\lambda_iv_i^3$ and
$S=\{v\in V:\sum v_i^2=1\}$. {\editC Then $J(\Phi_{\lambda,0})=\vartheta_{\lambda,0}=\editE J(v)$ and}
$$K^+=\{v\in S: (c|S) \text{ has a local max at } v, \text{ and } c(v)>0\}=\{e_i:\lambda_i>0\}$$ 
Moreover  $c(e_i)=\lambda_i/3$ for $e_i\in K^+$, {\editC and if $\ a\ne b\in K^+$ then $\langle a,b\rangle=0$, and $|K^+|=\type-1\le n-1$.}\end{lemma}

{\editC The subgroup $\O(\Omega,b)\subset G(\Omega)$ that stabilizes $b\in\bdy\Omega$ is conjugate to the subgroup $\O(\eta)\subset\GL V$ that preserves $V$, by (\ref{Opsi}). The following shows that the latter is the same as the subgroup that preserves
$J$. These results are keys steps in showing  $\eta$ and $[J]$ are powerful invariants. }
    \begin{lemma}\label{Olambdalemma} If $\theta:V\to T$ is a {\editC marked translation group} then
    $\O(J(\theta))=\O(\eta(\theta))$.
    \end{lemma}
    \begin{proof}  {\editD This is easy when $\type=0$ since the generalized cusp is standard, and the cubic term in
    $J$ is $0$. Thus we may assume $\type>0$ and then }
    {\editC by (\ref{Opsi}) $\O(\eta(\theta))\subset\GL V$ is conjugate to the stabilizer of the basepoint in 
    $\PGL\Omega$. Since $J(\theta)$ is preserved by the latter
$\O(\eta(\theta))\subset \O(J(\theta))$. To show the
    reverse inclusion,
    by (\ref{classification}), every marked translation group is given by $B(\zeta_{\psi}\circ A)B^{-1}$ for some $A\in \SLpm V$ and $B\in \Aff(n)$.
    Now $\O(J(\theta))$ and $O(\eta(\theta))$ are both unchanged under conjugation by $B$.
    Moreover $J(\theta\circ A)=J(\theta)\circ A$ and $\eta(\theta\circ A)=\eta(\theta)\circ A$.
     Thus
  is suffices to prove the result for $\theta=\zeta_{\psi}$.}
  
    Suppose $J=q+c$ where $q=\beta$ is the horosphere metric on $V$ given by
   $\zeta_{\psi}$ and $\O(q)\subset\GL(V)$ is the subgroup that
    preserves $q$. Set $\type=\type(\ppsi)$. Let $\Wcal=\{\weight_i\in V^*:1\le i\le \type\}$ be the set of {\editC non-zero}
     Lie algebra weights for $\zeta_{\psi}$.
   Then $\O(\eta(\zeta_{\psi}))$ is the subgroup of $\O(q)$ that preserves
   the character $\chi=\chi(\zeta_{\psi})$, and  
    $\O(J(\zeta_{\psi}))$ is the subgroup of $\Sim(q)$ that preserves
      the cubic $c$. {\editD Arguing as in (\ref{Opsi}) $\O(J(\zeta_{\psi}))\subset\O(q)$  since $\type>0$.}  The result will follow by showing that
      preserving $\chi$ is equivalent to preserving $\Wcal$ is equivalent to preserving $c$.

  {\editC {\editD By (\ref{charpoly}) preserving $\chi$
  is equivalent to preserving the characteristic polynomial $G=c_{\zeta_{\psi}}$. }
     Let $\Wcal^+\supset\Wcal$ be the multiset  of all Lie-algebra weights of the linear part of $\zeta_{\ppsi}.$ Then
    $|\Wcal^+|=n$ and $\Wcal^+$ contains the zero weight with multiplicity $n-\type$.
    The coefficients of $\editD G$ are the elementary symmetric functions
    of the elements of $\Wcal^+$. }Thus preserving $\editD G$ is equivalent to preserving
   $\Wcal$. {\editE By (\ref{ccformula}) $c=\editD(1/3{\editC\varkappa})\sum\lambda_i^{-2}\weight_i^3$. Thus if $\Wcal$} is preserved, then $c$ is preserved.

    For the converse, suppose $c$ is preserved.
    When $\type(\lambda)<n$ then
    by (\ref{localmaxlemmanotdiag})  $\O(J(\zeta_{\psi}))$
    preserves $K^+=\{e_i:\ppsi_i>0\}$ and since $c(e_i)=\psi_i/3$ it follows that $\O(J(\zeta_{\psi}))$ preserves $S=\{\ppsi_ie_i:1\le i\le n\}$. It follows that $\Wcal$ is preserved in this case.

     This leaves the case $\type(\lambda)=n$.
       By (\ref{localmaxlemmadiag}) $\O(J(\zeta_{\psi}))$ preserves $K$ and therefore permutes
     the coordinates of $\ppsi$. Moreover the formula for $c(v_i)$ in (\ref{localmaxlemmadiag}) shows that $c(v_i)=c(v_j)$
     if and only if $\ppsi_i=\ppsi_j$. {\editE Comparing this to (\ref{phidiag}) one sees that the weights are preserved.}
     Thus  $\O(J(\zeta_{\psi}))$ preserves $\Wcal$.
          \end{proof}

{\editC We now have the ingredients to show that $[J]$ determines $\ppsi$.
\begin{lemma}\label{Jdeterminesppsi}  If $n\ge 3$, and $A,A'\in\SLpm V$, and 
$[J(\zeta_{\ppsi}\circ A)]=[J(\zeta_{\ppsi'}\circ A')]$  then
$\ppsi=\ppsi'$.
\end{lemma}
\begin{proof}
       In what follows we scale $J=q+c$ so that $q$ is unimodular, and talk about {\em this} calibration instead
       of its projective class.
          Let $S=\{v\in V: q(v)=1\}$  and $K \subset S$
  the set of points at which $c|S$ has a local maximum, and let $K^+\subset K$ be the subset where $c>0$.
 Observe that $|K^+|$
 is an invariant of the similarity class of a cusp space. 

  Let $\langle\cdot,\cdot\rangle_q$ be the inner product on $V$ determined by $q$.
Then the set $\{\langle a,b\rangle_q:a,b\in K^+\}$
 is also an invariant of the similarity class. By (\ref{calibrationiscuspspace}) the  calibration on a marked translation
 group is similar to some $\vartheta_{\lambda,\kappa}$, and in the diagonalizable case also to some $J_{\ppsi}$. 
   First suppose $|K^+|\ge 2$
 and choose two distinct elements $a,b\in K^+$.
 
 {\bf Case 1} $\langle a,b\rangle_q=0$.
  Then (\ref{localmaxlemmadiag}) implies that
  $\type<n$, and (\ref{localmaxlemmanotdiag}) implies the coordinates of $\lambda$ are given by $c(v)$ as $v$ ranges over $K^+$.
  Moroever $\psi_i=1/\lambda_i^2$ so $\psi$ is determined by $[J]$ in this case. 
  
{\bf Case 2}  $\langle a,b\rangle_q\ne 0$.
Then (\ref{localmaxlemmanotdiag}) implies $\type=n$, so $(V,J(\zeta_{\ppsi}\circ A))$ is similar to $(\mathfrak g, J_{\psi})$.
\if0
The  $\alpha_{ij}$ defined in (\ref{localmaxlemmadiag})  are an invariant of the similarity class of $J_{\psi}$.
There is an isometry $$f:(\RR^n, \langle\cdot,\cdot\rangle_{\ppsi})\rightarrow(\RR^{n},\langle\cdot,\cdot\rangle)$$
given by $f(x_1,\cdots,x_n)=(x_1\sqrt{\ppsi_1},\cdots,x_{n}\sqrt{\ppsi_{n}})$. Recall $p=(1,\cdots,1)\in\RR^n$ and
$\mathfrak g=p^{\perp_{q}}$. Define  $\nu=f(1,\cdots,1)=(\sqrt{\ppsi_1},\cdots,\sqrt{\ppsi_{n}})$, 
then $V:=\nu^{\perp}=f(\mathfrak g)$. 

Observe that $f$
preserves the set of coordinate axes in $\RR^n$. Hence the $q$-orthogonal projections into $\mathfrak g$
of the coordinate axes  in $\RR^n$ map, by $f$, to the standard orthogonal projections
of the coordinate axes in $\RR^n$ into $V$. Thus the constants
$\alpha_{ij}$ in (\ref{angleslemma}) are then the same as in (\ref{localmaxlemmadiag}).
\fi
 It follows from (\ref{localmaxlemmadiag}) that $J$
determines determines $\ppsi$ up to multiplication by a positive scalar.
\if0 Let $f:V\to\mathfrak g$ be the linear isomorphism given by
\begin{align*}\label{feqtn}f(v_1e_1+\cdots+v_{n-1}e_{n-1})=(\psi_nv_1,\cdots,\psi_nv_{n-1},-\psi_1v_1\cdots-\psi_{n-1}v_{n-1})\end{align*}
Then $\zeta_{\ppsi}=\delta\circ\exp\circ f:V\to\Aff(n)$. Thus $J(\zeta_{\psi})=J(\delta\circ\exp)\circ f=J_{\ppsi}\circ f$.

Then  $[J(\zeta_{\psi}\circ A)]$ determines $\psi/\|\psi\|$ by (\ref{localmaxlemmadiag}) and (\ref{angleslemma}).\fi

  Thus we may assume $\psi'=s\psi$ with $s>0$.
  By (\ref{classification}a) $\zeta_{s\psi}=\zeta_{\psi}\circ ((s \Id_{\rank})\oplus \Id_{\ur})$. If 
  $[J(\zeta_{\ppsi}\circ A)]=[J(\zeta_{\ppsi'}\circ A')]$ it follows that
  $[J(\zeta_{\ppsi})]=[J(\zeta_{\ppsi}\circ B)]$ where $B=((s \Id_{\rank})\oplus \Id_{\ur})A' A^{-1}$.
  Thus $B\in O(J(\psi))$, so $\det B=\pm1$. Since $|\det A|=|\det A'|=1$ it follows that $\det((s \Id_{\rank})\oplus \Id_{\ur})=s^{\rank}=\pm1$.
 Thus $s=1$, so $\psi'=\psi$.
  
  {\bf Case 3} $|K^+|\le 1$. If $\type=n$ then   $|K^+|\ge n-1$ by (\ref{localmaxlemmadiag}). Since $n\ge 3$
   it follows that $\type<n$ {\editD which contradicts $\type=n$}. The result now follows from 
  (\ref{localmaxlemmanotdiag}) as before.
  \end{proof}
}

{\editC
\begin{lemma}\label{Jinjective} Suppose $\rho,\rho':V\to\Aff(n)$ are marked translation groups and $n\ge 3$. If $[J(\rho)]=[J(\rho')]$ then
$\rho$ and $\rho'$ are conjugate.
\end{lemma}
\begin{proof} We may assume  $\rho=\zeta_{\psi}\circ f$ and $\rho'=\zeta_{\psi'}\circ f'$ with $f,f'\in\SLpm V$.
        It follows from (\ref{Jdeterminesppsi}) that $\psi=\psi'$.
Then $[J(\rho)]=[J(\rho')]$ implies $f^{-1}\circ f'\in \O(J(\zeta_{\psi}))$, thus $f^{-1}\circ f'\in\O(\eta(\zeta_{\psi}))$
by (\ref{Olambdalemma}). Hence  $\rho$ and $\rho'$ have the same complete invariant, and so are conjugate by (\ref{completeinvt}).
 \end{proof}
}

     \begin{theorem}\label{Jcomplete} {\editA Suppose $n\ge 3$. Let $\Modsp_n$ be the space of marked generalized cusps homeomorphic to $T^{n-1}\times[0,\infty)$.
      The map 
     $J:\Modsp_n\longrightarrow \Jcal_n$ is a homeomorphism. Moreover $\Kappa:\Rcal_n\to\Jcal_n$   is a homeomorphism.}
   \end{theorem}
   \begin{proof}   By (\ref{Jinjective}) $J$ is injective. {\editE By (\ref{Jproper})  $J$ is continuous and proper.}
   The image of $J$ is contained in $\Jcal_{n}$ by (\ref{calibrationiscuspspace}), and
  surjectivity follows from the proof of (\ref{calibrationiscuspspace}). 
  Moreover $\Jcal_n$ is a subspace of Euclidian space and is therefore locally compact and Hausdorff. Also
 $\Modsp_n$ is locally compact by (\ref{quotienthomeo}),  so $J$ is a homeomorphism by  (\ref{prop_homeo}).
Now $J=\Kappa\circ\nu$, and $\nu$ is a homeomorphism by (\ref{weightspace}), thus $\Kappa$ is a homeomorphism. 
   \end{proof}

  \subsection{The Affine Normal}  A {\editC reference for this is chapter 1 of  \cite{Nom}, 
  see also \cite{Lof1} and  \cite[Lemma 4.1]{Klartag}. }
  Suppose $S\subset\RR^n$ is a smooth strictly convex hypersurface and $p$ is a point in $S$.  Then the tangent hyperplane
to $S$ at $p$ 
 intersects $S$
  only at $p$ and $S$ lies on one side of $P$. 
  An {\em affine normal} to $S$ at $p$ is vector $0\ne\nu=\nu(p)\in\RR^n$ with the following property. Given $\delta>0$
  let $P(\delta)$ be the hyperplane parallel to $P$ on the side of $P$ that contains $S$, and distance $\delta$ from $P$.
  %For $\delta$ small $S\cap P(\delta)$ is the boundary of a compact convex set $C(\delta)\subset P(\delta)$. 
{\editA  Let $x(\delta)$ be
  the center of mass of $S\cap P(\delta)$. }Then $(x(\delta)-p)/\delta$ converges to a non-zero multiple of $\nu$. We also
  require that $\nu$ points to the convex side of $S$. Then $\nu$ is
  defined up to positive scalar multiples.

It follows from this that affine normals are preserved by affine maps:   if $A$ is an affine map of  $\RR^n$ then $A(\nu(p))$ is an affine normal to $A(S)$. Since affine maps are not conformal, the affine normal is not in general orthogonal to $S$ at $p$. A  convex
hypersurface in $\RR^n$ is an {\em affine sphere} if there is a point $b\in\RPn$ such that every affine normal passes through $b$.

  There is a decomposition  $\Sym^3(\RR^n)=\harm_n\oplus\rad_n$  
into the  {\em harmonic cubics}  $\harm_n$, and the {\em radial cubics} $\rad_n$ given by
$$\harm_n=\{p\in \Sym^3(\RR^n):\ \Delta p=0\},\qquad \rad_n=\{\|x\|^2\langle v,x\rangle:\ v\in\RR^n\}$$
 The group $\O(n)$ acts on $\Sym^3(\RR^n)$ preserving this decomposition, and
  by \cite[Theorem 0.3]{That} the action on each summand is irreducible. 
  
  The material from here to (\ref{affsphharmonic}) is not used in this paper, so we have omitted the proofs.
  It is included to avert a possible misperception.
  The map $\pi:\Sym^3(\RR^n)\to \RR^n$ given by $\pi(p)=(2n+4)^{-1}\nabla(\Delta p)$ is projection onto $\rad_n$ followed by the map 
  $\|x\|^2\langle v,x\rangle\mapsto v$. 
   More generally, if $\beta$ is a positive definite quadratic form on $\RR^n$ then
  there is an isometry $L\in\GL(n,\RR)$ from $\|\cdot\|^2$ to $\beta$. Hence  $L(\harm_n)$ and $L(\rad_n)$ are preserved by $\O(\beta)$
  and
  $\pi_{_{\beta}}=L\circ\pi\circ L^{-1}:\Sym^3(\RR^n)\to\RR^n$.
 The following says that  
 the affine
  normal is the radial  part of the cubic term in a Taylor expansion.
 
 \begin{proposition}\label{affinenormalencode} Suppose $U\subset\RR^n$ is a neighborhood of $0$ and $f:U\to\RR$
     is $C^3$. Let $S\subset\RR^{\editC n+1}$ be the graph of $f$ and suppose {\editD $f(x)=\beta(x)+c(x) +o(\|x\|^3)$ and $\beta\in \Sym^2\RR^n$ 
     is positive definite, and
  $c\in\Sym^3\RR^n$.   Then  an
     affine normal to $S$ at $0$ is $e_{n+1}-(2n)^{-1}\pi_{_{\beta}}(c)$. }
     \end{proposition}

{\editC  This can  be deduced from formula (3.4) on page 48 of
  \cite{Nom}. 
   This formula  goes back at least to 1923, see Blaschke 
  \cite{Blaschke}.
  }

    {\editD    
     
            Recall that the   {\em radial flow} $\Phi:\RR\to\Aff_n$ for a generalized cusp lie group $G(\Omega)$
  centralizes it, and $\Phi_t(\Omega)\subset\Omega$ whenever $t\le 0$,  see \cite[(1.11)]{BCL}.
  If $\theta=\Phi_{\lambda,\kappa}$ the radial flow
 is $\Phi_t(x)=x-te_1$ if $\type<n$, and otherwise $\type(\lambda)=n$ and
 $\Phi_t(x)=e^{-t}(x-C)+C$ where $C\in\RR^n$ is the {\em center} of $\Phi$.
 {\editC Refer to (\ref{heightfn}) for the definition of $\tau$ and $H_b$ in the following.}  Now we may assume that
 $\Omega=\Omega(\lambda,\kappa)$ in (\ref{phidiag}) and
  $b=0$ and $H_b$ is $x_1=0$. Then $\tau(x_1,\cdots,x_n)=\alpha x_1$  for
     some $\alpha>0$.
  
    It is more convenient
 in the following to redefine the radial flow when $\type=n$ to be $\Phi:(-1,\infty)\rightarrow\Aff_n$ given by
 $\Phi_t(x)=(t+1)^{-1}\cdot (x-C)+C$. Then $\Phi_0$ is always 
 the identity and $I=\RR$ or $(-1,\infty)$ is the domain of $\Phi$ as appropriate.
 
 Then $F=\theta\times \Phi:V\times I\rightarrow\RR^n$ are coordinates on a subset of $\RR^n$ that contains $\Omega$.
   {\em In these coordinates} the height function 
 $h_{\theta}$
 {\em describes (an open subset of) $H_b$ as a graph over $\bdy\Omega$}
    rather than  {\em vice-versa}, as one might n\"aively imagine.
    \begin{lemma}\label{htlemma} Scale $\tau$ so that if $\type<n$ then $\tau(x_1,\cdots,x_n)=x_1$ and if
    $\type=n$ then $\tau(c)=-1$. Then
    $F(V\times 0)=\bdy\Omega$ and $F(\{(v,t):\ t=h_{\theta}(v)\})\subset H_b$.
    \end{lemma}
    If $J(\theta)=\editE [\beta+c]$ then (\ref{affinenormalencode}) implies that  $\beta$ and the radial-cubic part of $c$  determines the affine normal to $F^{-1}(H_b)$.}

\begin{proposition}\label{affsphharmonic} Let $\|\cdot\|$ be the standard inner product on $V$ and let
$\theta:V\rightarrow T$ be a marked
translation group and $S$ a horosphere for $T$, and with radial flow $\Phi$ and $J(\theta)=\beta+c$ with $\beta$ unimodular.
The following are equivalent\begin{itemize}
\item[(a)] flow lines of $\Phi$ are   affine normals to $S$.
\item[(b)] $S$ is an affine sphere.
\item[(c)] {\editC $c$ is harmonic with respect to $\beta$ i.e. $\pi_{_{\beta}}(c)=0$}
\item[(d)] $T$ is conjugate to $\Tr(s,\cdots,s)$ with $s\ge 0$.
\end{itemize}
\end{proposition}
 \begin{proof} Flow lines of $\Phi$ limit on the center of the radial flow, so $(a)\Rightarrow(b)$.
 For the converse,  assume $S$ is an affine sphere with center $w\in\RPn$. Then $T$ fixes $w$.
  If the affine normals to $S$ are parallel, then $S$ is an elliptic paraboloid, \cite{calabi2}, \cite{Pog}.
 In this case $T$ is conjugate to $\Tr(0,\cdots,0)$, and $w$ is the center of $\Phi$. Otherwise $w\in\RR^n$.
 Thus $T$ is diagonalizable. We may assume $T=\Tr(\psi)$ with all the coordinates of $\psi>0$ and $w=0$.
 Again $w$ is the center of $\Phi$. Thus $(b)\Rightarrow(a)$. In this case we claim $\psi=(s,\cdots,s)$.
 This is because $S$ is an affine sphere asymptotic to the sides of a simplex, and by \cite{MR437805}
 it follows that $S$ is unique up to affine maps preserving the simplex. Thus $(b)\Rightarrow(d)$.
 For $(d)\Rightarrow(b)$ when $s=0$ then $S$ is an elliptic paraboloid and when $s>0$ then $S$ is
 defined by $\prod x_i=1$. These are well known affine spheres. 

{\editD It remains to show $(c)\Leftrightarrow(d)$. Using (\ref{cuspspacedef}) we may assume
$$J=[\beta+c],\qquad \beta(v)=\|v\|^2+\langle v,\kappa\rangle^2,\qquad 3c(v)=-\lambda_0\langle v,\kappa\rangle^3+\sum_{i=1}^{n-1}\lambda_iv_i^3$$
If $\lambda_0=0$ we may choose $\kappa=0$ then $c$ is harmonic with respect to $\beta(v)=\|v\|^2$  if and only if $\lambda=0$,
showing $(c)\Leftrightarrow(d)$ in this case. 
Otherwise $\lambda_0>0$. {\editE First we perform a linear change of coordinates on $V$ so that $\beta(v)=\|v\|^2$.}

Let $T\in \GL (V)$ be defined by $T(v)=v +\alpha\langle v,\kappa\rangle \kappa$ where 
$\alpha=\|\kappa\|^{-2}({\editE -}1+(1+\|\kappa\|^2)^{-1/2})$ then  $\beta(T v)=\|v\|^2$. Now we compute the cubic $c\circ T$ using the Hadamard product on $V$, and $\kappa_i=\lambda_i/\lambda_0$ .
\begin{align*} 3\lambda_0^{-1}(c\circ T)v&=-\langle Tv,\kappa\rangle^3 +\langle\kappa^{-1},(Tv)^3\rangle\\
&= \gamma\langle v,\kappa\rangle^3 + 3\alpha\langle v,\kappa\rangle\|v\|^2+\langle\kappa^{-1},v^3\rangle\\
{\rm where}\quad \gamma&=-\left(1+\alpha\|\kappa\|^2\right)^3 + 3\alpha^2+\alpha^3\|\kappa\|^2
=\frac{-(2+\|\kappa\|^2)+2\sqrt{1+\|\kappa\|^2}}{\|\kappa\|^{\editE 4}\sqrt{1+\|\kappa\|^2}}
\end{align*}
Set $m=\dim V$ then 
\begin{align*}3\lambda_0^{-1}\nabla^2(c\circ T)&=\left(6\gamma\|\kappa\|^2 + 
3\alpha(2m+4)\right)\langle v,\kappa\rangle+6\langle \kappa^{-1},v\rangle\\
 &=\langle 6u,v\rangle\\
\text{where}\qquad u&= -\left(\frac{m}{\|\kappa\|^2}+\frac{\|\kappa\|^2-m}{{\editE\|\kappa\|^2}\sqrt{1+\|\kappa\|^2}}\right)\kappa+\kappa^{-1}
\end{align*}
Then $c\circ T$ is harmonic with respect to $\|\cdot\|^2$  if and only if {\editC $u=0$}. Since $u$
is a linear combination of $\kappa$ and $\kappa^{-1}$ it follows that 
 $\kappa=s(1,\cdots,1)$
for some $s\in[0,1]$. Then $\|\kappa\|^2=ms^2$ and  {\editC $u=0$ implies}
$$\left(s^{-2}+{\editE m(s^2-1)/ms^2\sqrt{1+ms^2}}\right)s=s^{-1}$$
This implies $\editE s^2-1=0$. Hence {\editE $s=1$ and}  $(c)\Leftrightarrow(d)$ when $\lambda_0>0$. }
   \end{proof}

  \section{Three Dimensions}\label{3D} 
  
  {\editE In dimension $3$ every generalized cusp is
  equivalent to $\Omega_{\lambda,\kappa}/\Gamma$ for some lattice
  in $\Gamma\subset T(\lambda,\kappa)$, and $\bdy\Omega_{\lambda,\kappa}$
 is the orbit of $0$ under  $T(\lambda,\kappa)$. 
 From the proof of  (\ref{phidiag}) one sees that  in dimension $3$ that $\bdy\Omega_{\lambda,\kappa}$  is  the graph graph $y=f_{\lambda}(x_1,x_2)$ in $\RR^3$ shown below
 where for $\type<3$ we have chosen $\kappa=0$. 
 
 \begin{center}
 \begin{tabular}{|c | c |} 
 \hline
$\type$\ \   &   $f_{\lambda}(x_1,x_2)$\\
\hline
3 & $\lambda_1^{-1}x_1+\lambda_2^{-1}x_2+\lambda_0^{-2}\left(-2+(1+\lambda_1x_1)^{-(\lambda_0/\lambda_1)^2}+(1+\lambda_2x_2)^{-(\lambda_0/\lambda_2)^2}
\right)$\\
\hline
2  &  $(x_1+x_2)/2-\lambda_1^{-2}\log(1+\lambda_1x_1)+\lambda_2^{-2}\log(1+\lambda_2x_2)$\\
\hline
1  & $x_1^2/2+\lambda_2^{-2}\log(1+\lambda_2x_2)$\\
\hline
0 & $(x_1^2+x_2^2)/2$\\
\hline
\end{tabular}
\end{center}
 
\if0 \begin{align*}
 \type \qquad &   f_{\lambda}(x_1,x_2)\\
3\qquad  & \lambda_1^{-1}x_1+\lambda_2^{-1}x_2+\lambda_0^{-2}\left(-2+(1+\lambda_1x_1)^{-(\lambda_0/\lambda_1)^2}+(1+\lambda_2x_2)^{-(\lambda_0/\lambda_2)^2}
\right)\\
2\qquad  & (x_1+x_2)/2-\lambda_1^{-2}\log(1+\lambda_1x_1)+\lambda_2^{-2}\log(1+\lambda_2x_2)\\
1\qquad  &x_1^2/2+\lambda_2^{-2}\log(1+\lambda_2x_2)\\
0\qquad &(x_1^2+x_2^2)/2
 \end{align*}
 \fi
 The function $f_\lambda$ varies continuously with $\lambda$  on the subspace $\lambda_0=0$, and is also continuous
 when $\lambda_1,\lambda_2>0$ are constant as $\lambda_0\to0$, but is not continuous in general.
 This family of surfaces only varies continuously with $\lambda$ subject to these constraints.} 
 
{\editE Using the horosphere metric $\beta$ we may identify a Lie-algebra weight in $V^*$  with a vector in $V$.
Then  a generalized cusp in a 3-manifold is specified by a parallelogram of area one in $V=\RR^2$, together with three vectors
$a,b,c$ in $V$ satisfying $\langle a,b \rangle=\langle b,c\rangle=\langle c,a\rangle=\varpi\le 0$.
The Lie algebra weights
of the holonomy are given by $\xi(x)=\langle v,x\rangle_{\beta}$ where $v\in\{a,b,c\}$.

Two such collection of data define equivalent cusps if and only if there is an isometry of $\RR^2$
taking one parallelogram to the other and that permutes the set of vectors $\{a,b,c\}$.
The type of the generalized cusp is the number of these vectors that are non-zero. \begin{figure}[ht]
 \begin{center}
\includegraphics[scale=0.45]{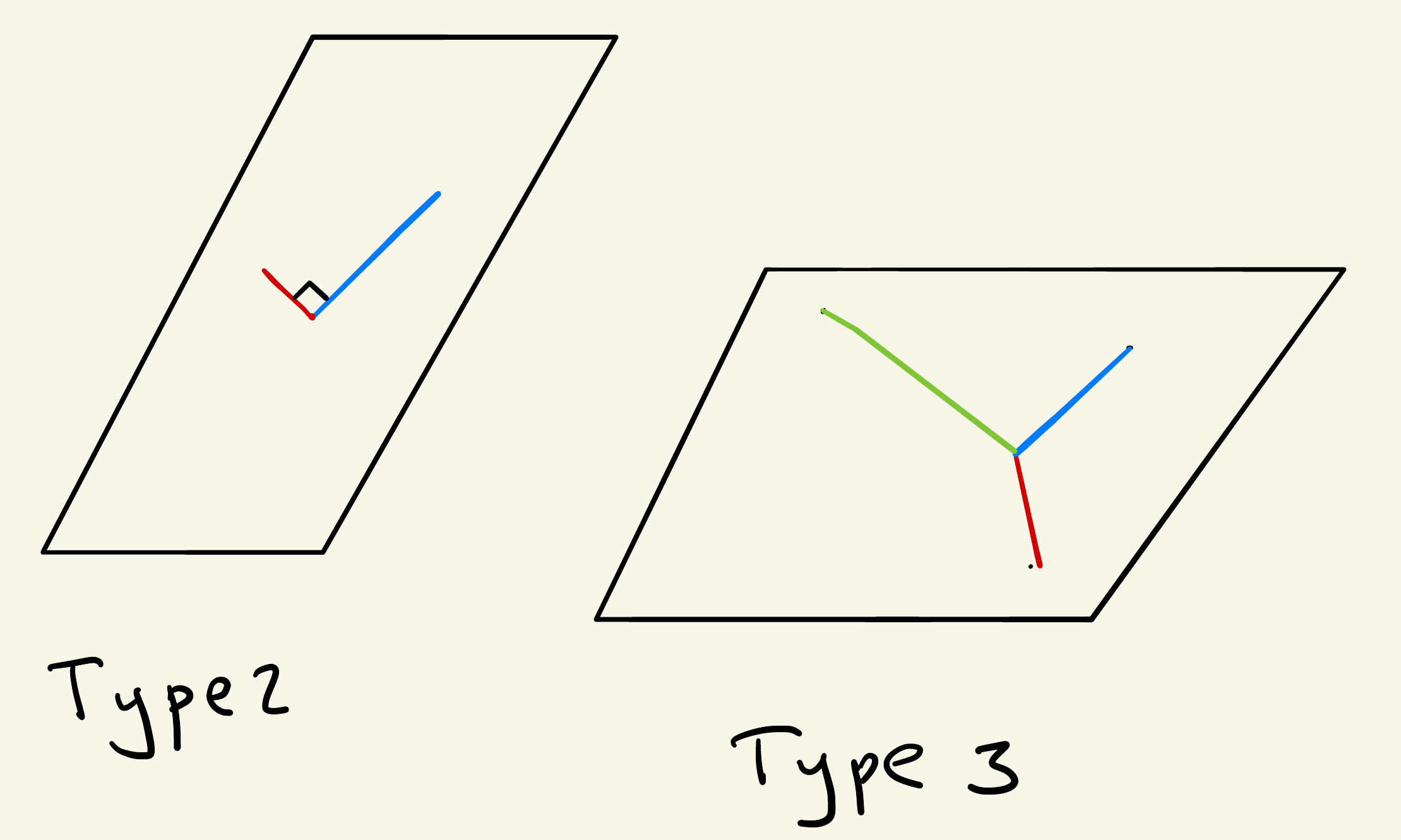}
	 \end{center}
 \caption{Generalized cusps in dimension 3} \label{weights}
\end{figure}
 }

There is a decomposition of $\Sym^3(\RR^2)=\harm_2\oplus\rad_2$ is given by
$$\harm_2=\langle x(x^2-3y^2),y(y^2-3x^2)\rangle,\qquad \rad_2=\langle x(x^2+y^2),y(x^2+y^2)\rangle$$ 
with coordinate projections $\pi_{_{\harm}}$ and $\pi_{_{\rad}}$.
 {\editC By (\ref{affsphharmonic}) the cubic is harmonic {\editD with respect to $\beta$} if and only if the holonomy is conjugate into
  $\Tr(s,s,s)$ for some $s\ge 0$.}
 %This type has been shown to frequently appear as deformations of hyperbolic cusped manifolds. {\red citation ?}
 
Regarding $V=\RR^2\cong\CC$ via $z=x+iy$, and recalling that the real part of a holomorphic function is harmonic, it follows that
$$\harm_2=\{\operatorname{Re}(h z^3):h\in\CC\},\qquad \rad_2=\{\operatorname{Re} (r z|z|^2) :r\in\CC\}$$
This gives an isomorphism of real vector spaces $\theta:\CC^2\to\Sym^3\RR^2$ given by $\theta(h,r)= \operatorname{Re}(h z^3+r z  |z|^2)$.
The action of $\SO(2)\cong U(1)=\{\omega\in\CC:|\omega|=1\}$ on $\Sym^3\RR^2$ is then
$\omega.\theta(h,r)=\theta(\omega^3 h,\omega r)$.
The standard Euclidean structure on $\CC^2$ gives an inner product  on $\Sym^3\RR^2$ given by $\|\theta(h,r)\|^2=|h|^2+|r|^2$,
and $\SO(2)$ acts by isometries. {\editC Let $\beta_0$ be the quadratic form $x^2+y^2$ on $\RR^2$}.

{\editC
\begin{theorem}\label{3mfdspace} The image of the embedding $J:\Modsp_3\to\Pcal(\RR^2)\times\Sym^3\RR^2$ is
$$
J(\Modsp_3)= \{(A^tA,c\circ A)\in \Pcal(\RR^2)\times\Sym^3\RR^2\ :\ |\pi_{_{\rad}}c|\le 3|\pi_{_{\harm}}c|,\ \ A\in\SL(2,\RR)\ \}$$
Moroever $|\pi_{_{\rad}}c|= 3|\pi_{_{\harm}}c|$ gives the subspace of non-diagonalizable generalized cusps.
\end{theorem}
\begin{proof} 
 In this proof we identify $\Sym^3\RR^2\equiv\CC^2$ using $\theta$ and {\editE $\Tcal_3\equiv\Rep_3$ using the holonomy}. The action of $A\in\SLpm V$ on $\Modsp_3$ defined in (\ref{nuaction}) is conjugate by $J$ to the action on
$\Pcal(\RR^2)\times\Sym^3\RR$ given by $A\cdot(\beta,c)=(\beta\circ A^{-1},c\circ A^{-1})$. This action  preserves the product structure.
The stabilizer of  $\beta_0$ is $O(2)$. 

{\editE {\bf Claim 1} Suppose $(\lambda,\kappa)\in\widetilde A_n$ and $\lambda=(0,\lambda_1,\lambda_2)$ and $\kappa=0$ and
$c=c(\Phi_{\lambda,\kappa})$. 
Then $\pi_{_{\harm}}c=z$ and $\pi_{_{\rad}}c=3\overline{z}$ where $z=(\lambda_1+i\lambda_2)/12$.}

From the definition (\ref{lambdakappa}), the Lie algebra weights are $\xi_i=\lambda_ie_i^*$ for   $i\in\{1,2\}$, and $\xi_0=0$. 
Using Lemma (\ref{cubicrootdata}) then $\kappa=0$ so $\varkappa=1$ and formula (\ref{ccformula})  gives
 $$3c=\lambda_{1}^{-2}\xi_1^3+\lambda_{2}^{-2}\xi_2^3=\lambda_1(e_1^*)^3+\lambda_2(e_2^*)^3=\lambda_1x^3+\lambda_2y^3$$
 Expressing this in terms of the generators of $\harm_2$ and $\rad_2$ gives
$$12c=\lambda_1[x(x^2-3y^2)+3x(x^2+y^2)]+\lambda_2[y(y^2-3x^2)+3y(x^2+y^2)]$$
So \begin{align}
12h&=\lambda_1x(x^2-3y^2)+\lambda_2y(y^2-3x^2)\qquad  &12r &=3\lambda_1x(x^2+y^2)+3\lambda_2y(x^2+y^2) \nonumber \\
&=\operatorname{Re}\left((\lambda_1+i\lambda_2)z^3\right) \qquad  &  &=\operatorname{Re}\left( 3(\lambda_1-i\lambda_2)z |z|^2\right) \nonumber\\
\therefore\quad 12\pi_{\harm}(c)&=\lambda_1+i\lambda_2 & 12\pi_{\rad}(c)&=3(\lambda_1-i\lambda_2) \nonumber
\end{align}
This proves claim 1.

Now $B=\Modsp_3\setminus\Modsp_3(3)$ consists of all marked generalized cusps with non-diagonalizable holonomy.
Let $\pi:\Pcal(\RR^2)\times\Sym^3\RR\rightarrow\Pcal(\RR^2)$ be projection and
consider the subspace of $N=B\cap (\pi\circ J)^{-1}\beta_0$ of non-diagonalizable holonomies for the
standard quadratic form $\beta=\|\cdot\|^2$.

{\bf Claim 2} $J(N)=\{(\beta_0,h,r):|r|=3|h|\}$. 

If $[\rho]\in N$ then $[\rho]=[\Phi_{\lambda,0}\circ A]$ with $\type(\lambda)<3$ and $A\in O(2)$.
Under the identification $V=\CC$, 
the action of $\SO(2)$ on $V$ is 
given by the action of $U(1)$ on $\CC$. If $J(\Phi_{\lambda,0})=(z,3\overline{z})$, and $A$ is rotation by $\theta$, 
and $\omega=\exp(i\theta)$
  then  $J(\Phi_{\lambda,0}\circ A)=(\omega^3 z,3\overline{\omega}\overline{z})$.   
  Moreover if $A\in O(2)$ is given by $A(x,y)=(x,-y)$ then
 $J(\Phi_{\lambda,0}\circ A)=(\overline{z},3z)$. Given $h,r\in\CC$ with $|r|=3|h|$ there are $z,\omega\in\CC$ with $|\omega|=1$
 such that $(h,r)=(\omega^3 z,3\overline{\omega} \overline{z})$. This proves claim 2.

Using the action of $\SLpm V$ on $\Modsp_3$ it follows that
$$
J(B)= \{(A^tA,c\circ A)\in \Pcal(\RR^2)\times\Sym^3\RR^2\ :\ |\pi_{_{\rad}}c|=3|\pi_{_{\harm}}c|,\ \ A\in\SLpm V\}$$
Consider $f:\Pcal\times\CC^2\to\RR$ given by $f(\beta,h,r)=3|h|-|r|$, and
set $P=\Modsp_3(3)$. 
When $\lambda=(1,1,1)$ then (\ref{affsphharmonic}) implies the cubic is harmonic so $r=0$, and $h\ne0$ thus
$J(P)$ contains a point where $f>0$. Since $J$ is injective, 
$J(P)\subset\CC^2\setminus f^{-1}(0)$. By (\ref{strata})  $P$ is connected, so
 $f\circ J(P)> 0$.  
 
By (\ref{Jproper})  $J:\Modsp_3\rightarrow\Pcal \times\CC^2$ is proper, and the domain and codomain are locally compact,
  thus $J(\Modsp_3)$ is closed.  By (\ref{strata}), $P$  is a 6-manifold without boundary. Since 
$J:P\rightarrow\Pcal \times\CC^2$ is an embedding, and $\Pcal\times\CC^2$ is a 6-manifold, $J(P)$ is open by invariance of domain. Hence $J(P)=f^{-1}(0,\infty)$.
\end{proof}

\begin{proof}[Proof of (\ref{3manifold})] Let $U\subset\SL(2,\RR)$ be the subspace of upper-triangular 
matrices with positive eigenvalues. Then $g:U\rightarrow \Pcal$ given by $g(A)=A^tA$ is a homeomorphism. Let $G=g^{-1}$ and
$$\Ccal=\{c\in\Sym^3\RR^2:\ |\pi_{_{\rad}}c|\le 3|\pi_{_{\harm}}c|\ \}\equiv\{(r,h)\in\CC^2:\ |r|\le3|h|\ \}$$ Define $f:\Pcal\times\Ccal\to \Pcal\times\Sym^3\RR^2$ by
$f(Q,c)=(Q,c\circ G(Q))$. Then $f$ is an embedding, {\editD since it has inverse $f^{-1}(Q,c)=(Q,c\circ (G(Q))^{-1})$.}
If $A=G(Q)$ then $f(Q,c)=(Q,c\circ A)$ and $$A\cdot f(Q,c)=((A^t)^{-1}QA^{-1},(c\circ A)\circ A^{-1})=(I,c)$$ Thus the image of $f$ is $J(\Modsp_3)$, so $f^{-1}\circ J:\Modsp_3\rightarrow \Pcal\times\Ccal$ is a homeomorphism. There is a homeomorphism
$h:\Pcal\to\Hcal=\{z\in\CC:\ \Im z>0\ \}$ given by $h(Q)= \alpha_A(i)$ where $A=G(Q)$ and $ \alpha_A$ is the M\"obius transformation
corresponding to $A$. Then $\Theta=(h\times Id)\circ f^{-1}\circ J:\Modsp_3\rightarrow\Hcal\times\Ccal$ is a homeomorphism.
\end{proof}

Now we describe the strata of $\Modsp_3$. Let $\pi:\Modsp_3\to\Pos$ be projection.
The fiber $\pi^{-1}(\beta_0)$  is the cone
$F=\{(h,r)\in\CC^2: |r|\le 3|h| \}$ stratified
 as follows. For $k\in\{0,1,2,3\}$, let $T_k=\Modsp_3(k)\cap\pi^{-1}(\beta_0)$. Then
$T_0=(0,0)\in\CC^2$ is the cone point, and $T_1=\{ ( w^3|w|^{-2},3w):\  w\in\CC\setminus 0  \}$
 is the open cone of a $(3,1)$ curve in $S^1\times S^1$ {\editD  because $c$ is the cube of a linear polynomial},
 and  $T_2=\bdy F-(T_1\cup T_0)$, and $T_3=\interior(F)$. The stratification is preserved by
 the action of $\SL(2,\RR)$ which also preserves the fibering and acts transitively on the base space $\Pcal$. }

 \section{Appendix: routine proofs.}
 
 \begin{proof}[Proof of (\ref{completepsi})]  The character and Lie-algebra weights can be read off from Definition (\ref{psigroup}). 
To compute $\beta$ we use (\ref{betaeq}) with
 basepoint $b=(e_1+\cdots+e_{\type})+e_{n+1}$.  When $\type<n$ from (\ref{psigroup})
 $$\mu_{\theta,b}(v)-b=\sum_{i=1}^{\type}(\exp({\editC \psi_{\type}}v_i)-1)e_i+\sum_{i=\type+1}^{n-1}v_{i}e_{i+1}+\left(-\sum_{i=1}^{\type} \psi_iv_i+(1/2)\sum_{i=\type+1}^{n-1} v_i^2\right) e_{\type+1}$$
Computing $u_i=(\partial\mu_{\theta,b}/\partial v_i)_{v=0}$  gives
$$(u_1,\cdots,u_{n-1})=( \psi_{\type}e_1-\psi_1e_{\type+1},\psi_{\type}e_2-\psi_1e_{\type+1},\cdots,{\editC \psi_{\type}}e_{\type}-\psi_{\type}e_{\type+1},e_{\type+2},\cdots,e_{n})$$
By  (\ref{geqtn1}) $h_{\theta}(v)={\editD\pm}\det(u_1,\cdots,u_{n-1},\mu_{\theta,b}(v)-b)$ gives 
\begin{align*}
h_{\theta}(v)  &=\pm\det
\bpmat \psi_{\type} & & & & & & & \exp(\psi_{\type}v_1)-1\\
 &  \psi_{\type}& & & & & &\exp(\psi_{\type}v_2)-1\\
 & &\ddots & & & & &\vdots\\
 &  & & \psi_{\type}& & & &\exp(\psi_{\type}v_{\type})-1\\
-\psi_1 &-\psi_2 &\cdots &-\psi_{\type} & 0&\cdots & 0& -\sum_{i=1}^{\type}\psi_iv_i +(1/2)\sum_{\type+1}^{n-1}v_i^2\\
 &  & & & 1&  &  & v_{\type+1}\\
 &  & & & & \ddots &  & \vdots\\
 &  & & & &  &  1 & v_{n-1}\\
\epmat\\
&=\det
\bpmat \psi_{\type} & & & &  \exp(\psi_{\type}v_1)-1\\
 &  \psi_{\type}& & & \exp(\psi_{\type}v_2)-1\\
 & &\ddots & & \vdots\\
 &  & & \psi_{\type}& \exp(\psi_{\type}v_{\type})-1\\
-\psi_1 &-\psi_2 &\cdots &-\psi_{\type} &  -\sum_{i=1}^{\type}\psi_iv_i +(1/2)\sum_{\type+1}^{n-1}v_i^2\\
\epmat\\
&={\editC \psi_{\type}}^{\type-1} \sum_{i=1}^{\type}\psi_i(\exp({\editC \psi_{\type}}v_i)-1-{\editC \psi_{\type}}v_i)+(1/2){\editC \psi_{\type}}^{\type}\sum_{i=\type+1}^{n-1}v_i^2\\
\end{align*}
Taking the second derivative at $v=0$ yields
$$\widetilde\beta={\editC \psi_{\type}^{\type+1}}\sum_{i=1}^{\type}\psi_idv_i^2+{\editC \psi_{\type}}^{\type}\sum_{i=\type+1}^{n-1}dv_i^2$$
{\editA  Observe that the matrix of $\beta'={\editC \psi_{\type}}^{-(\type+1)}\widetilde\beta$ is diagonal  in the standard basis and $\det\beta'$ is as claimed.
It is clear that the Lie algebra weights $\xi_i$ are pairwise $\beta$-orthogonal. Thus so are their duals.}

{\editD By definition (\ref{psigroup}), the non-zero Lie algebra weights are $\xi_i=\psi_{\type}e_i^*$ with $1\le i\le\type$.
Now $\langle x,e_i\rangle_{\beta'}=\psi_ie_i^*(x)$. Let $\gamma=\det(\beta')^{-1/n-1}$, then
 $\beta=\gamma\cdot\beta'$ so $\langle x,e_i\rangle_{\beta}=\gamma\psi_ie_i^*(x)$.
 Thus the dual of $\xi_i$ with respect to $\beta$ is $(\gamma\psi_i)^{-1}\psi_{\type}e_i$ and
\begin{align}\label{dualwteqtn}
\beta^*(\xi_i)&=\beta((\gamma\psi_i)^{-1}\psi_{\type}e_i)\nonumber\\
&=((\gamma\psi_i)^{-1}\psi_{\type})^{2}\beta(e_i)\nonumber\\
&= (\gamma^{-2}\psi_i^{-2}\psi_{\type}^{2})\gamma\beta'(e_i)\nonumber\\
&=(\gamma^{-2}\psi_i^{-2}\psi_{\type}^{2})\gamma\psi_i\nonumber\\
&=\gamma^{-1}\psi_{\type}^2\psi_i^{-1}\end{align}}

If $\type=n$ choose basepoint $b=e_1+\cdots+e_{n+1}$ 
then $$\mu_{\theta,b}(v)-b=\sum_{i=1}^{n-1}(\exp({\editC \psi_n}v_{i})-1)e_{i}+\left(\exp\left(-\sum_{i=1}^{n-1} \psi_iv_i\right) -1\right)e_{n}$$ thus
$u_i=(\partial\mu_{\theta,b}/\partial v_i)_{v=0}={\editC\psi_n}e_i-\psi_ie_n$. Then  (\ref{geqtn1}) gives 
\begin{align*}
h_{\theta}(v)& =\det
\bpmat 
\psi_n      & & & & \exp(\psi_n v_1)-1\\
 & \ddots&  &  &\vdots\\
 & & & \psi_n&\exp(\psi_nv_{n-1})-1\\
 -\psi_1 &-\psi_2&\cdots&-\psi_{n-1} & \exp(-\sum_{i=1}^{n-1}\psi_iv_i)-1
\epmat
\\
&={\editC \psi_n^{n-2}}\sum_{i=1}^{n-1}{\editC \psi_i}\left(\exp({\editC \psi_n}v_i)-1\right)+{\editC \psi_n^{n-1}}\left(\exp\left(-\sum_{i=1}^{n-1}\psi_iv_i\right) -1\right)
\end{align*}
Taking the second derivative at $v=0$ gives
$$\widetilde\beta={\editC \psi_n^{n}}\left(\sum_{i=1}^{n-1}\psi_i dv_i^2+{\editC \psi_n^{-1}}\left(-\sum_{i=1}^{n-1}\psi_idv_i\right)^2\right)$$
Then $\beta'=\psi_n^{\editC-n}\widetilde\beta$ gives the form shown in the proposition.
$$\psi_n\beta'=\bpmat
\psi_1(\psi_n+\psi_1) & \psi_1\psi_2 &\cdots & & \psi_1 \psi_{n-1}\\
\psi_2\psi_1 & \psi_2(\psi_n+\psi_2) & \psi_2\psi_3 &\cdots & \psi_2\psi_{n-1}\\
\vdots\\
\psi_{n-1}\psi_1 &\cdots & & \psi_{n-1}\psi_{n-2} & \psi_{n-1}(\psi_n+\psi_{n-1})
\epmat$$
The determinant of this matrix is a polynomial of degree $2(n-1)$. Row $i$ has a factor of $\psi_i$. The sum of the rows
is a multiple of $\psi_1+\psi_2\cdots+\psi_n$. Setting $\psi_n=0$ gives a matrix of rank $1$ so $\psi_n^{n-2}$ is a factor.
Hence 
$$\det(\psi_n\beta')=\alpha \psi_1\cdots \psi_{n-1} \psi_n^{n-2}(1+\psi_1+\cdots \psi_{n-1})$$
for some constant $\alpha$. Equating coefficients of  $\psi_n^{n-1}$ gives $\alpha=1$. Thus $$\det\beta'= \psi_1\cdots \psi_{n-1} \psi_n^{-1}(1+\psi_1+\cdots \psi_{n-1})$$
  \end{proof}

 \begin{proof}[Proof of (\ref{phidiag})] (a) Given $v=(v_1,\cdots,v_{n-1})\in V$ define $v_0\in V$ by $\lambda_0^{-1}v_0+\cdots\lambda_{n-1}^{-1}v_{n-1}=0$.
 Let
 $$P=\begin{pmatrix} 1 & -\lambda_1^{-1} & \cdots & -\lambda_{n-1}^{-1} & \lambda_0^{-2}\\
0 & 1 & 0 & 0 & \lambda_1^{-1}\\\
0 & 0 &\ddots & 0 & \vdots\\
0 & 0 & 0 & 1 & \lambda_{n-1}^{-1}\\
0 & 0 & 0 & 0 & 1\\ 
\end{pmatrix},
\qquad
r=\begin{pmatrix} \lambda_0 v_0 &  0 & \cdots & &0\\
0 & \lambda_1v_1 &0  &\cdots  & 0\\
0 & 0 &\ddots & 0 & \vdots\\
0 & \cdots & 0 & \lambda_{n-1}v_{n-1} & 0\\
0 & \cdots &  &  & 0\\
\end{pmatrix}$$
 then
$$P^{-1}rP=
\begin{pmatrix} 0 & v_1  & \cdots & v_{n-1} &0\\
0 & \lambda_1v_1 &0  &\cdots  & v_1\\
0 & 0 &\ddots & 0 & \vdots\\
0 & 0 & 0  & \lambda_{n-1}v_{n-1}& v_{n-1}\\
0 & 0 & 0 & \ldots & 0\\ 
\end{pmatrix}
+\lambda_0v_0
\begin{pmatrix} 1 & -\lambda_1^{-1}  & \cdots & -\lambda^{-1}_{n-1} &0\\
0 & \cdots &  &  & 0\\
\vdots &   &  &   & \vdots\\
 &  &   &  &  \\
0 & \cdots &  &  & 0\\ 
\end{pmatrix}
$$
Now $\lambda_0^{-1}v_0=-\lambda_1^{-1}v_1\cdots-\lambda_{n-1}^{-1}v_{n-1}$ so $v_0=-(\kappa_1v_1+\cdots+\kappa_{n-1}v_{n-1})=-\langle v,\kappa\rangle$ where $\kappa_i=\lambda_0/\lambda_i$. Then
\begin{align*}P^{-1}rP&=
\begin{pmatrix} 0 & v_1  & \cdots & v_{n-1} &0\\
0 & \lambda_1v_1 &0  &\cdots  & v_1\\
0 & 0 &\ddots & 0 & \vdots\\
0 & 0 & 0  & \lambda_{n-1}v_{n-1}& v_{n-1}\\
0 & 0 & 0 & \ldots & 0\\ 
\end{pmatrix}
+\langle v,\kappa\rangle
\begin{pmatrix} -\lambda_0 & \kappa_1  & \cdots & \kappa_{n-1} &0\\
0 & \cdots &  &  & 0\\
\vdots &   &  &   & \vdots\\
 &  &   &  &  \\
0 & \cdots &  &  & 0\\ 
\end{pmatrix}\\
&=\phi_{\lambda,\kappa}(v)
\end{align*}
Set $R=\exp r$ then $P\Phi_{\lambda,\kappa}P^{-1}=R$. From (\ref{psigroup})
\begin{align*}
\zeta_{\psi}({\editC{\mathfrak f}} v)
= \exp\begin{pmatrix} \psi_n\lambda_0^2\lambda_1v_1  &  0 & \cdots & & & 0\\
0 & \psi_n\lambda_0^2\lambda_2v_2 &0  &\cdots  & & 0\\
0 & 0 &\ddots & 0 & \vdots\\
\vdots & \vdots & 0 & \psi_{n}\lambda_0^2\lambda_{n-1}v_{n-1} & & 0\\
 &  &  &  & -\sum_{i=1}^{n-1}\psi_i\lambda_0^2\lambda_iv_i& 0\\
0 & \cdots &  &  & & 0\\
\end{pmatrix}\end{align*}
Using $\psi_n\lambda_0^2=1$ and $\psi_i\lambda_i=\lambda_i^{-1}$ gives
\begin{align*}
\zeta_{\psi}({\editC{\mathfrak f}} v)
=\exp\begin{pmatrix} \lambda_1 v_1 &  0 & \cdots & & &0\\
0 & \lambda_2v_2 &0  &\cdots  & & 0\\
0 & 0 &\ddots &  & & \vdots\\
\vdots &\vdots & 0 & \lambda_{n-1}v_{n-1} & & 0\\
 &  &  &  & -\lambda_0^{2}\sum_{i=1}^{n-1}\lambda_i^{-1}v_i& 0\\
0 & \cdots &  &  & & 0\\
\end{pmatrix}
\end{align*}
Now $-\lambda_0^{2}\sum_{i=1}^{n-1}\lambda_i^{-1}v_i=\lambda_0v_0$. Let $M\in \GL(n+1,\RR)$ be 
defined by $$M(x_1,\cdots,x_{n+1})=\editE (x_n,x_1,\cdots,x_{n-1},x_{n+1})$$ Then
$$ M^{-1}\zeta_{\psi}({\editC{\mathfrak f}} v)M=\\
\exp \begin{pmatrix} \lambda_0 v_0 &  0 & \cdots & & &0\\
0 & \lambda_1v_1 &0  &\cdots  & & 0\\
0 & 0 &\ddots &  & & \vdots\\
\vdots &\vdots & 0 & \lambda_{n-1}v_{n-1} & & 0\\
0 & \cdots &  &  & & 0\\
\end{pmatrix}=R$$
so 
 \begin{equation}\label{centereq}M^{-1}(\zeta_{\psi}\circ{\editC{\mathfrak f}})M=R=P\Phi_{\lambda,\kappa}P^{-1}\end{equation}
Set $Q=MP$ then $Q\Phi_{\lambda,\kappa}Q^{-1}=\zeta_{\psi}\circ {\editC{\mathfrak f}}$ as asserted.

To prove (b) we exploit the fact that every $\Phi_{\lambda,\kappa}$ is a limit of the diagonalizable ones above.
Given an integer $k\ge 0$ define $f_k:\RR^2\to\RR$ by
 $$f_k(s,t)=\sum_{j=k}^{\infty}s^{j-k}t^j/j!$$
 This is analytic and $f_0(s,t)=\exp(st)$, and for $s\ne0$ 
 $$f_1(s,t)=s^{-1}(e^{st}-1),\qquad f_2(s,t)=s^{-2}(e^{st}-1-st)$$
 Also $f_1(0,t)=t$ and $f_2(0,t)=t^2/2$.
For $s\ge 0$ the map  $f_1(s,-):\RR\to(-s^{-1},\infty)$  is a diffeomorphism when we interpret $-0^{-1}=-\infty$, and $f_2(s,-):\RR\to\RR$ is convex and proper. Then
$$P^{-1}RP=
\begin{pmatrix} e^{\lambda_0v_0} &*  & \cdots & * & \sum_{i=0}^{n-1}\lambda_i^{-1}f_1(\lambda_i,v_i)\\
0 & e^{\lambda_1v_1} &0\cdots  & 0  & f_1(\lambda_1,v_1)\\
0 & 0 &\ddots & 0 & \vdots\\
\vdots & \vdots & 0  & e^{\lambda_{n-1}v_{n-1}}& f_1(\lambda_{n-1},v_{n-1})\\
0 & \cdots &  0& 0 & 1\\ 
\end{pmatrix}
$$
Set $x_i=f_1(\lambda_i,v_i)$ and $y=\sum_{i=0}^{n-1} {\editE\lambda_i^{-1}}f_2(\lambda_i,v_i)$. {\editD Write} the last column of $P^{-1}RP$ as 
$(y,x_1,\cdots,x_{n-1},1)^T$. Now 
\begin{equation}\label{v0eqtn}\lambda_0^{-1}v_0=-(\lambda_1^{-1}v_1+\cdots+\lambda_{n-1}^{\editE-1}v_{n-1})\end{equation}
Observe that $s^{-1}f_1(s,t)=f_2(s,t)+s^{-1}t$. Thus 
\begin{align}\label{f1eqtn}
\lambda_0^{-1}f_1(\lambda_0,v_0)
= f_2(\lambda_0,v_0)+\lambda_0^{-1}v_0
=f_2(\lambda_0,v_0)-\sum_{i=1}^{n-1}\lambda_i^{-1}v_i
\end{align}Then
\begin{align}
y=&\sum_{i=0}^{n-1}\lambda_i^{-1}f_1(\lambda_i,v_i)\nonumber\\
=&\lambda_0^{-1}f_1(\lambda_0,v_0)+\sum_{i=1}^{n-1}\lambda_i^{-1}f_1(\lambda_i,v_i)\nonumber\\
=&\left(f_2(\lambda_0,v_0)-\sum_{i=1}^{n-1}\lambda_i^{-1}v_i\right)+\sum_{i=1}^{n-1}\left(f_2(\lambda_i)+\lambda_i^{-1}v_i\right) & using\  (\ref{f1eqtn})\nonumber\\
=&\sum_{i=0}^{n-1} f_2(\lambda_i,v_i)\label{f2eqtn}
\end{align}
The orbit of the origin under 
$T(\lambda,\kappa)$ is a  hypersurface $S=S(\lambda,\kappa)$ in $\RR^n$ that is the locus of the points
$(y,x_1,\cdots,x_{n-1})$ as $v$ varies in $V$. Solving $x=f_1(\ell,v)$ for $v$ gives
\begin{equation}\label{heqtn}v=h(\ell,x):=\ell^{-1}\log(1+\ell x)\end{equation}
This defines $h(\ell,x)$ whenever $1+\ell x>0$ and $\ell\ne 0$. Observe that $h(\ell,x)=x+\ell\cdot \O(x^2)$, so
if we define $h(0,x)=x$ then $h$
is analytic on the subset of $\RR^2$ where $1+\ell x>0$.
{\editC Define $g(\ell,x)=\ell^{-2}(\ell x-\log(1+\ell x))$ for $1+\ell x>0$ and $\ell\ne 0$. Observe that $g(\ell,x)=x^2/2+ \O(x^3)$,
thus if we define  $g(0,x)=x^2/2$,
then $g$ is analytic for $1+\ell x>0$.}
Then
\begin{align}
f_2(\ell,v)&=f_2(\ell,\ell^{-1}\log(1+\ell x)) \nonumber\\ 
&=\ell^{-2}(e^{\ell\ell^{-1}\log(1+\ell x)}-1-\log(1+\ell x))\nonumber\\
&=\ell^{-2}(\ell x-\log(1+\ell x)\nonumber)\\
&=g(\ell,x) \label{geqtn}
\end{align}
The hypersurface $S=S(\lambda,\kappa)$ is given by
\begin{align*}y&=\sum_{i=0}^{n-1} f_2(\lambda_i,v_i) & by\ (\ref{f2eqtn})\\
&=f_2\left(\lambda_0,-\lambda_0\sum_{i=1}^{n-1}\lambda_i^{-1}v_i\right)+\sum_{i=1}^{n-1}g(\lambda_i,x_i)&
\because\ v_0=-\lambda_0\sum_{i=1}^{n-1}\lambda_i^{-1}v_i\ \&\ f_2(\lambda_i,v_i)=g(\lambda_i,x_i)\\
&=f_2\left(\lambda_0,-\sum_{i=1}^{n-1}\kappa_ih(\lambda_i,x_i)\right)+\sum_{i=1}^{n-1}g(\lambda_i,x_i)&
\because\ v_i=h(\lambda_i,x_i)\ \&\ \kappa_i=\lambda_0/\lambda_i\\
&=:F(\lambda,\kappa,x) & {\rm definition}
\end{align*}
Here $x=(x_1,\cdots,x_{n-1})$. Up to this point we have assumed $(\lambda,\kappa)\in D_n$
so every $\lambda_i>0$. However the function $F$ is defined and analytic whenever $(\lambda,\kappa)\in\widetilde{A}_n\cup D_n$
and $1+\lambda_ix_i>0$ for all $i$. It follows that $y=F(\lambda,\kappa,x)$ defines a hypersurface $S(\lambda,\kappa)$
for each $(\lambda,\kappa)\in\widetilde{A}_n$. 

Also $S(\lambda,\kappa)$ is the orbit of $0$ under $T(\lambda,\kappa)$
whenever $(\lambda,\kappa)\in D_n$. Since $\widetilde{A}_n\subset\cl D_n$ and $\Phi_{\lambda,\kappa}$
is a continuous function of $(\lambda,\kappa)$ it follows that $S(\lambda,\kappa)$ is the orbit of $0$ under
 $T(\lambda,\kappa)$
whenever $(\lambda,\kappa)\in\widetilde{A}_n$. For fixed $(\lambda,\kappa)$ 
\begin{align*}
h(\lambda_i,x_i)&=x_i+\O(x^2)& by\ (\ref{heqtn})\\
\therefore\qquad\sum_{i=1}^{n-1}\kappa_ih(\lambda_i,x_i)&=\sum_{i=1}^{n-1}\kappa_i(x_i+\O(x_i^2))\\
&=\langle\kappa,x\rangle +\O(\|x\|^2)
\end{align*}
Using this and $f_2(\lambda,x)=x^2/2 + \O(x^3)$ gives
\begin{align*}
f_2\left(\lambda_0,-\sum_{i=1}^{n-1}\kappa_ih(\lambda_i,x_i)\right)&=(1/2)\langle\kappa,x\rangle^2+\O(\|x\|^3)
\end{align*}
Also
\begin{align*}
g(\lambda_i,x_i)&=x_i^2/2+\O(x_i^3) & by\ (\ref{geqtn})\\
\therefore\qquad F(\lambda,\kappa,x)&=f_2\left(\lambda_0,-\sum_{i=1}^{n-1}\kappa_ih(\lambda_i,x_i)\right)+\sum_{i=1}^{n-1}g(\lambda_i,x_i)\\
&=(1/2)\left(\langle\kappa,x\rangle^2+\|x\|^2\right)+ \O(\|x\|^3)
\end{align*}
It follows that $S(\lambda,\kappa)$ is strictly convex at $0$. Since $T(\lambda,\kappa)$ acts transitively by affine maps
$S(\lambda,\kappa)$ is strictly convex everywhere. {\editE One checks} that  $F(\lambda,\kappa,x)$ is a proper
function of $x\in\{(x_1,\cdots,x_{n-1})\ :\ 1+\lambda_ix_i>0\ \}$ for fixed $\lambda,\kappa$. Hence $S(\lambda,\kappa)$
is properly embedded, and therefore bounds a convex domain $\Omega(\lambda,\kappa)\subset\RR^n$ that
is preserved by $T(\lambda,\kappa)$.

Now $v_i=x_i+\O(x_i^2)$ thus $\Phi_{\lambda,\kappa}(v)=(y,x_1,\cdots,x_{n-1})$ where
$$y=(1/2)\left(\langle\kappa,v\rangle^2+\|v\|^2\right)+ \O(\|v\|^3)$$
which gives the formula for $\beta'$. The formula for $\chi_{_{\lambda,\kappa}}$ follows immediately from the definition
(\ref{lambdakappa} )
as the sum of the exponentials of the diagonal terms. It only remains to compute $\det\beta'$. Now
$$\beta'(v)=\langle\kappa,v\rangle^2+\|v\|^2$$
Choose an orthonormal basis with respect to $\|\cdot\|^2$ of $V$ that contains $\kappa/\|\kappa\|$.
In this basis $\beta'$ is diagonal,
and the only diagonal entry that is not $1$ is $1+\langle\kappa,\kappa/\|\kappa\|\rangle^2=1+\|\kappa\|^2$. Hence
$\det\beta'=1+\|\kappa\|^2$.
 \end{proof}
 
\begin{proof}[Proof of (\ref{calibrationiscuspspace})]   Suppose $(\lambda,\kappa)\in A_n$.
  First consider the diagonalizable case. {\editC By (\ref{phidiag}),  $\Phi_{\lambda,\kappa}$ is conjugate to   
$\zeta_{\ppsi}\circ{\editC{\mathfrak f}}$,  where
$\psi_n=\lambda_0^{-2}$ and
$\ppsi_i=\lambda_i^{-2}$ for $1\le i\le n-1$. This defines a linear map $\psi:\AA\to\RR$ {\editE and we have ${\mathfrak g}=\ker\psi$.}
Since $J$ is an invariant of conjugacy classes, we may replace $\Phi_{\lambda,\kappa}$
by  $\zeta_{\ppsi}\circ{\editC{\mathfrak f}}$. }
In this proof summation is over the integers from $1$ to $n-1$. Consider the linear map $f:\RR^{n-1}\to\AA$  given by 
 $$x:=f(v_1,\cdots,v_{n-1})=\left(\psi_n v_1,\cdots,\psi_nv_{n-1},-\sum\psi_iv_i\right)$$
{\editC Then define $\mathfrak g:=\Image f=\ker \ppsi$}. 
{\editC Recall ${\editC{\mathfrak f}}(v_1,\cdots,v_{n-1})=\lambda_0^2(\lambda_1v_1,\cdots,\lambda_{n-1}v_{n-1})$. Thus
\begin{align}\label{ff}f\circ{\editC{\mathfrak f}}(v)&=\lambda_0^2(\psi_n\lambda_1v_1,\cdots,\psi_n\lambda_{n-1}v_{n-1},-\sum\psi_i\lambda_iv_i)\nonumber\\
&=(\lambda_1v_1,\cdots,\lambda_{n-1}v_{n-1},-\lambda_0^2\sum\lambda_i^{-1}v_i)& \because\lambda_0^2\psi_n=1\nonumber\\
&=(\lambda_1v_1,\cdots,\lambda_{n-1}v_{n-1},-\lambda_0\sum\kappa_iv_i) & \because\lambda_0\lambda_i^{-1}=\kappa_i
\end{align}
{\editA It follows from  Definition (\ref{psigroup}) that $\editC \zeta_{\ppsi}=\delta\circ\exp\circ f$}.  
The calibration, $J_{\psi}$, on $\mathfrak g$ is given by (\ref{diagonalcalibration})
\begin{equation}\label{Jpsieqtn}J_{\psi}(x)=\ (1/2)\langle p,x^2\rangle_{\ppsi}+(1/6)\langle p,x^3\rangle_{\ppsi}\end{equation}
The calibration 
$J=J(\zeta_{\ppsi}\circ\mathfrak f)=J(\delta\circ\exp\circ f\circ\mathfrak f)$. {\editE By (\ref{diagonalcalibration}) $J_{\psi}=J(\delta\circ\exp)$, so $J=J_{\ppsi}\circ f\circ{\editC{\mathfrak f}}$}. This calibration  on $V$ is obtained from this by using (\ref{ff}) to substitute $x=f({\editC{\mathfrak f}} v)$ into (\ref{Jpsieqtn}).
\begin{align*}
\langle p,x^2\rangle_{\ppsi} &=\sum\ppsi_i(\lambda_iv_i)^2+\psi_n\left(-\lambda_0\sum\kappa_iv_i\right)^2\\
&=\sum v_i^2 + \left(\sum\kappa_iv_i\right)^2&\because\psi_i\lambda_i^2=1
\end{align*}
Let $\langle\cdot,\cdot\rangle$ denote the standard inner product on $\RR^{n-1}$ then
\begin{equation}\label{eqa1}
\begin{array}{rcl}\langle p,x^2\rangle_{\ppsi} &=&\langle v,v\rangle + \langle v,\kappa\rangle^2\hspace{1.05in}\;
\end{array}
\end{equation}
and 
\begin{align}\label{eqa2}
\langle p,x^3\rangle_{\ppsi} &=\sum\ppsi_i(\lambda_iv_i)^3+\psi_n\left(-\lambda_0\sum\kappa_iv_i\right)^3  \nonumber\\
& = \sum \lambda_iv_i^3 -\lambda_0\langle \kappa,v\rangle^3 & \because\ \ \psi_i\lambda_i^2=1\ \&\ \psi_n\lambda_0^{2}=1
\end{align}
Then (\ref{eqa1}) and (\ref{eqa2}) give
\begin{align*}
J(v) &=  (1/2)\langle p,x^2\rangle_{\ppsi}+(1/6)\langle p,x^3\rangle_{\ppsi}\\
& = (1/2)\left(\langle v,v\rangle + \langle v,\kappa\rangle^2\right)+(1/6)\left( -\lambda_0\langle \kappa,v\rangle^3+\sum\lambda_iv_i^3 \right)
\end{align*}
 This gives the result in the diagonalizable case. 
 
 By (\ref{Tconj})} in the non-diagonalizable case we may assume $\theta=\Phi_{\lambda,\kappa}=\exp\circ\phi_{\lambda,\kappa}$
 with $\lambda_0=0$ and $\kappa=0$. {\editE For $1\le i\le n-1$ define $v'_i=v_i+\langle v,\kappa\rangle\kappa_i$.} Then
$\phi_{\lambda,\kappa}(v_1,\cdots,v_{n-1})=D+N$ where
$$D = 
\begin{pmatrix}
0 & 0 &0 &  \cdots &  &0 \\
0 & \lambda_1 v_1&0  & \cdots &  & \\
\vdots && \ddots &&& \vdots\\
&& && \lambda_{n-1} v_{n-1} & 0 \\
0 &&&& \cdots &0
\end{pmatrix}\qquad
N = 
\begin{pmatrix}
0 &  v_1^{\editE '} & v_2{\editE '} &  \cdots &v_{n-1}{\editE '}&0 \\
0 & &0  & \cdots & 0 & v_1\\
\vdots && \ddots &&& \vdots\\
&& &&  & v_{n-1} \\
0 &&&& \cdots &0
\end{pmatrix}
$$
Relabel the standard basis of $\RR^{n+1}$ as $e_0,\cdots,e_n$.
Then $\bdy\Omega$ is the orbit in affine space $\RR^n\oplus e_n\subset\RR^{n+1}$ of 
 $0\oplus e_{n}$ under this group. We compute the series expansion for $\exp(D+N)e_{n}$ to degree $3$.
$$\exp(D+N)=I +(D+N) +(1/2)(D+N)^2+(1/6)(D+N)^3 + O(\|v\|^4)$$
Using  that $De_{n}=0$ and $N^3=0$ and $DN^2e_{n}=0$ gives
\begin{equation}\label{DNeeqtn}
\exp(D+N)e_{n}= \left(I +N +(1/2)(DN+N^2)+(1/6)(D^2N +NDN)\right)e_{n} + O(\|v\|^4)
\end{equation}In the following summation is over integers from $1$ to $n-1$ 
\begin{align*}
\begin{aligned}
Ne_n&=\sum v_ie_i,  & N^2e_{n}&={\editE \left(\|v\|^2+\langle v,\kappa\rangle^2\right)}e_0\\
DNe_n& =  \sum\lambda_iv_i^2e_i, & NDNe_{n}&=\left(\sum\lambda_iv_i^3\right)e_0, & D^2Ne_n&=\sum\lambda_i^2v_i^3e_i
\end{aligned}
\end{align*}
The only term linear in $v_i$ is $Ne_n$, so
the supporting hyperplane to $\bdy\Omega$ at $0$ is the coordinate hyperplane $v_0=0$ in $\RR^n$.
Thus {\editC in the definition of $J$} we may take the height function $\tau$ to be the $v_0$-coordinate, and
it follows that $J$ is the coefficient of $e_0$ {\editD in (\ref{DNeeqtn})}
$$J=(1/2)N^2e_n+(1/6)NDNe_n=(1/2){\editE \left(\|v\|^2+\langle v,\kappa\rangle^2\right)} +(1/6)\sum\lambda_iv_i^3$$
This is the calibration {\editE $\vartheta_{\lambda,\kappa}$  as claimed.}
\end{proof}

\begin{proof}[Proof of (\ref{cubicrootdata})] {\editB We claim the formula holds  when $\rho=\Phi_{\lambda,\kappa}$.
 By (\ref{calibrationiscuspspace}) $${\editB{\editC\varkappa}}J(\Phi_{\lambda,\kappa})(v)=\left(\langle v,v\rangle +
 \langle v,\kappa\rangle^2\right)  +\frac{1}{3}\left(-\lambda_0\langle v,\kappa\rangle^3+ \sum_{i=1}^{n-1} \lambda_i v_i^3\right)$$
 {\editC Then $J=q+c$ and $q$ is unimodular.}
 Thus
\begin{align}\label{cformula1}c:=c(\rho)=\frac{1}{3\editB{\editC\varkappa}}\left(-\lambda_0\langle v,\kappa\rangle^3+ \sum_{i=1}^{n-1} \lambda_i v_i^3\right)\end{align}
  By (\ref{lambdakappa}) the weights of $\Phi_{\lambda,\kappa}$ are $\xi_i(v)=\lambda_ie_i^*$ for $1\le i\le n-1$
and $\xi_0(v)=-\lambda_0\langle v,\kappa\rangle$. Then (\ref{cformula1}) becomes:
\begin{align}\label{ccformula}{\editD 3{\editC\varkappa}}c=\sum_{i=0}^{n-1}\lambda_i^{-2}\dxi_i^3\end{align}
 By (\ref{rootslambda}) ${\editB {\editC\varkappa}}\lambda_i^2=\langle \xi_i,\xi_i\rangle_{_{\beta^*}}-\langle \xi_1,\xi_2\rangle_{_{\beta^*}}$.
 This proves the claim.
 
If $A\in\SLpm V$ then $c(\rho\circ A)=c(\rho)\circ A$ and $\xi_i(\rho\circ A)=\xi_i(\rho)\circ A$.
It follows that the formula holds for $\rho=\Phi_{\lambda,\kappa}\circ A$. Every marked translation group is conjugate to such $\rho$,
and both sides are conjugacy invariants, so the formula holds in general.}

{\editE The formula shows}  $\Kappa$ is continuous. It only remains to show $\Kappa$ is proper. Suppose $J(\rho_m)$ is a bounded sequence,
 {\editD then we must show $\kappa(\rho_m)=([\xi_1(\rho_m),\cdots,\xi_n(\rho_m)],\beta(\rho_m))$ is bounded.}
Now $\beta(\rho_m)$ {\editC and $c(\rho_m)$} are both bounded, 
so it only remains to prove the weights $\xi_i(\rho_m)$ are bounded.  Suppose
for a contradiction that  some $\xi_i(\rho_m)$ is unbounded. {\editC We show that this implies $c(\rho_m)$ is
not bounded, which gives a contradiction.}

We may assume $\rho_m=\Phi_{\lambda(m),\kappa(m)}^{\perp}\circ B_m$ with $B_m\in\SLpm V$. The matrix
of $\beta(\rho_m)$ in the standard basis of $V$ is $B_m^tB_m$. Since this is bounded, $B_m$ is bounded, and we
may subsequence so $B_m$ converges to $B_{\infty}\in\SLpm V$. It follows that the  weights
of $\Phi_{\lambda(m),\kappa(m)}^{\perp}$ are unbounded. Hence $\lambda(m)=(\lambda_{m,0},\cdots,\lambda_{m,n-1})$ is unbounded
and therefore $\lambda_{m,n-1}\to\infty$. 

{\editE The matrix of $\beta(\Phi_{\lambda(m),\kappa(m)})$ is $\Id+\kappa\otimes\kappa$}
and $\kappa\in[0,1]^{n-1}$ is bounded. Thus the  weights
of $\Phi_{\lambda(m),\kappa(m)}$ are unbounded. 
{\editD The weights determine the cubic via (\ref{ccformula})} 
and $\xi_i=\lambda_ie_i^*$ for $1\le i\le n-1$, and $\xi_0(v)=-\lambda_0\langle v,\kappa\rangle=-\lambda_0\sum_{i=1}^{n-1}\lambda_0\lambda_i^{-1}v_i$.
We now  evaluate this cubic at the point 
 $v_m=e_{n-1}-t_m(e_1+\cdots+e_{n-2})\in V$   where $t_m$ is chosen so  that  $\langle\kappa,v\rangle=0$. This simplifies
 the first summand  {\editD in (\ref{ccformula})} to  $\xi_0(v_m)=0$. {\editE If $\type<n$ then $\kappa=0$ and we choose $t_m=0$.

  If $\type=n$}
 since $\lambda_{m,n-1}\ge\lambda_{m,i}$ for all $i$
and $\kappa_i=\lambda_0\lambda_i^{-1}$ it follows that $\kappa_{n-1}\le\kappa_i$ for all $i$. {\editC Now $t_m$ is determined by
$$0=\langle\kappa,v\rangle  =\kappa_{n-1}-t_m\sum_{i=1}^{n-2}\kappa_i\qquad \&\ \editD\kappa_i\ge 0$$
implies $0<t_m\le1/(n-2)$.} In what follows we omit the subscript $m$
from $\lambda_{m,i}$. {\editC Setting $v=v_m$ in (\ref{ccformula}) and recalling that $\xi_i=\lambda_ie_i^*$ gives}
\begin{equation}
\label{eq423}
3{\editC {\editC\varkappa}}c(\Phi_{\lambda,\kappa})(v_m)=\left(\sum_{i=1}^{n-2}\lambda_i^{-2}(-\lambda_it_m)^3\right)+\lambda_{n-1}^{-2}(\lambda_{n-1})^3=\lambda_{n-1}-t_m^3\sum_{i=1}^{n-2}\lambda_i
\end{equation}
Since $\lambda_{i}\le\lambda_{n-1}$, and $t_m\le 1/(n-2)$, 
\begin{equation}\label{wtscubic}3{\editC {\editC\varkappa}}c(\Phi_{\lambda,\kappa})(v_m)\ge \lambda_{n-1}(1-(n-2)/(n-2)^3)
\end{equation}
{\editC Since $0\le\kappa_i\le 1$ it follows that $\|\kappa\|\le n-1$ thus ${\editC\varkappa}=(1+\|\kappa\|^2)^{1/(n+1)}\le 1+n^2$.}
If $n>3$ then {\editC$(1-(n-2)/(n-2)^3)>0$} . Using $\lambda_{n-1}\to\infty$ as $m\to\infty$ {\editC and ${\editC\varkappa}$ is bounded}, it follows that $c(\Phi_{\lambda,\kappa})$ is unbounded, a contradiction.

{\editC This leaves the case that $n=3$} then (\ref{cformula1}) gives
$$3{\editC {\editC\varkappa}}c(\Phi_{\lambda,\kappa})(x,y)=
\lambda_1 x^3+\lambda_2 y^3 + \lambda_0^{-2}(-\lambda_0((\lambda_0/\lambda_1)x+(\lambda_0/\lambda_2)y))^3$$
The coefficients of this cubic are
$$(\lambda_2^4-\lambda_0^4)\lambda_2^{-3},\quad -3\lambda_0^4/(\lambda_1\lambda_2^2),\quad
-3\lambda_0^4/(\lambda_1^2\lambda_2),\quad (\lambda_1^4-\lambda_0^4)\lambda_1^{-3}$$
Assuming each coefficient has absolute value at most $b$, the second term gives $\lambda_0^4\le b\cdot \lambda_1\lambda_2^2$.
Then $\lambda_2^4-\lambda_0^4\ge \lambda_2^4-b\cdot \lambda_1\lambda_2^2$. But $\lambda_2\ge\lambda_1$
so $$\lambda_2^4-\lambda_0^4\ge\lambda_2^4-b\cdot \lambda_1\lambda_2^2\ge \lambda_2^4-b\cdot\lambda_2^3=\lambda_2^3(\lambda_2-b)$$ 
{\editD Hence $\lambda_2-b\le  (\lambda_2^4-\lambda_0^4)\lambda_2^{-3}\le b$,
so} $\lambda_2\le 2b$. Since $\lambda_i\le\lambda_2$ for $i=0,1$ it follows that all the $\lambda_i\le 2b$.
This is a contradiction. {\editE Hence $\Kappa$ is proper.}
\end{proof}

 \begin{proof}[Proof of (\ref{localmaxlemmadiag})]  
 {\editC Let $L=\{s\cdot e_i:1\le i\le n,\ \ s>0\}$ be the set of positive coordinate axes in $\RR^n$
 and $\pi:\RR^n\to\mathfrak g$ orthogonal projection with respect to $\langle\cdot,\cdot\rangle_{\psi}$.
 We will show that
 the local maxima of $(c|S)$ are  the points  $ {\editC (\pi L)}\cap S$ on 
 $S$ that meet the images under orthogonal projection $L$. }

 Write $J=J_{\ppsi}$. Since $q|S=1$ it follows that $J|S=1+c|S$ .  First we find the critical points of $J|S$. The derivative of $J$ at $v\in\RR^n$ is
\begin{equation}\label{dJ}dJ_v(w)=\langle v,w\rangle_{\psi}+(1/2)\langle v^2,w\rangle_{\psi}\end{equation}
If $v\in S$ then  $w\in T_vS$ if and only if $\langle v,w\rangle_{\psi}=0$ and $\langle p,w\rangle_{\psi}=0$. Thus $v$ is a critical point of $J|S$ if and only if 
$$\forall w\in\RR^n\qquad\left(\langle v,w\rangle_{\psi}=0\quad\text{and}\quad \langle p,w\rangle_{\psi}=0\right)\quad\Longrightarrow\quad \langle v^2,w\rangle_{\psi}=0$$
This is equivalent to $$\exists\ \alpha,\beta\in\RR\qquad v^2=\alpha v+\beta p$$
Writing  $v=(v_1,\cdots,v_n)$ then each  $v_i$ is a solution of $t^2=\alpha t+\beta$. Let $\mathfrak{s}_{_{\pm}}$ be the solutions and set
$$A_{\pm}=\{i:\ 1\le i\le n\quad v_i=\mathfrak{s}_{_{\pm}}\}$$
Thus $\{A_+,A_-\}$ is a partition of $\{1,\cdots,n\}$ and $i\in A_+$ if and only if $v_i=\mathfrak{s}_{_+}$. Let $e_1,\cdots,e_n$ be the
standard basis of $\RR^n$ and define
\begin{equation}\label{eequation}e_{_{\pm}}=\sum_{i\in A_{\pm}} e_i\qquad \text{so}\qquad  p=e_{_+}+e_{_-}\end{equation}
then 
\begin{equation}\label{vequation} v=v(A_+)=\mathfrak{s}_{_+} e_{_+}+\mathfrak{s}_{_-} e_{_-}\end{equation}
The standard basis is orthogonal so $\langle e_{_+},e_{_-}\rangle_{\psi}=0$. Now $v\in p^{\perp}$ implies $$0=\langle p,v\rangle_{\psi}=\langle e_{_+}+e_{_-},\mathfrak{s}_{_+}e_{_+}+\mathfrak{s}_{_-}e_{_-}\rangle_{\psi}=\mathfrak{s}_{_+}\langle e_{_+},e_{_+}\rangle_{\psi}+\mathfrak{s}_{_-}\langle e_{_-},e_{_-}\rangle_{\psi}$$
Since $\langle e_{_+},e_{_+}\rangle_{\psi},\langle e_{_-},e_{_-}\rangle_{\psi} >0$ it follows that $\mathfrak{s}_{_+}\mathfrak{s}_{_-}<0$. We choose the labelling so that 
\begin{equation}\label{sigmasign} \mathfrak{s}_{_+}>0\qquad\text{and}\qquad\mathfrak{s}_{_-}<0\end{equation}
Then there is $t>0$ so that 
$$
\mathfrak{s}_{_+}=t\cdot \langle e_{_-},e_{_-}\rangle_{\psi} \qquad \mathfrak{s}_{_-}=-t\cdot \langle e_{_+},e_{_+}\rangle_{\psi}
$$
Hence
$$t^{-1}\cdot v=\langle e_{_-},e_{_-}\rangle_{\psi} e_{_+} - \langle e_{_+},e_{_+}\rangle_{\psi} e_{_-}$$
We will ignore the $t$ factor in what follows. This is justified by observing that the critical points of $J$ restricted to $t\cdot S$ are the critical
points of $J|S$ multiplied by $t$. 
Then
\begin{equation}\label{sigmaequation}
\mathfrak{s}_{_+}= \langle e_{_-},e_{_-}\rangle_{\psi} \qquad \mathfrak{s}_{_-}=- \langle e_{_+},e_{_+}\rangle_{\psi}
\end{equation}

$$\mathfrak{s}_{_+}=\sum_{i\in A_-}\ppsi_i,\qquad \mathfrak{s}_{_-}=-\sum_{i\in A_+} \ppsi_i, \qquad \text{and} \qquad\mathfrak{s}=\mathfrak{s}_{_+}-\mathfrak{s}_{_-}=\sum_{i=1}^n\ppsi_i$$
We have show the critical points of $J|S$ are in one to one correspondence with the non-empty subsets
$A_+\subset\{1,\cdots,n\}$ with non-empty complement. Given a quadratic form $Q$ define $\mu(Q)$ to be the dimension of the positive eigenspace. 
This is the {\em Morse index} of $-Q$. Thus a non-degenerate critical point is a local maximum if and only if the Hessian has $\mu=0$.

{\bf Claim 1} The critical point of $f=J|S$ at $v=v(A_+)$ is non-degenerate, and $\mu(d^2f_v)=|A_+|-1$. 
 \halfgap
 
 Assuming this we prove the lemma. The claim implies the local maxima
 occur when $|A_+|=1$ so $A_+=\{i\}$ for some $1\le i\le n$. When $A_+=\{i\}$ by (\ref{eequation})  $$e_{_+}=e_i\qquad e_{_-}=p-e_i$$ By (\ref{sigmaequation})
 $$ \mathfrak{s}_{_+}=\mathfrak{s}-\langle e_i,e_i\rangle_{\psi}\qquad \mathfrak{s}_{_-}=-\langle e_i,e_i\rangle_{\psi}$$
Using $\langle e_i,e_i\rangle_{\ppsi}=\ppsi_i$    and (\ref{vequation}) 
$$v(A_+)=(\mathfrak{s}-\ppsi_i) e_i-\ppsi_i(p-e_i)\editE =\mathfrak{s} e_i-\ppsi_i p$$
Using $\langle p,e_i\rangle_{\ppsi}=\ppsi_i$ and $\langle p,p-e_i\rangle_{\ppsi}=\mathfrak{s}-\ppsi_i$ gives
\begin{align}\label{cvA+}
6c(v(A_+)) &= \langle p,v_i^3\rangle_{\ppsi}\nonumber\\
& =  (\mathfrak{s}-\ppsi_i)^3\langle p,e_i\rangle_{\ppsi} -\ppsi_i^3\langle p,p-e_i\rangle_{\ppsi}\nonumber\\
&=\left(\psi_i(\mathfrak{s}-\psi_i)^3-\psi_i^3(\mathfrak{s}-\ppsi_i)\right)\nonumber\\
& =  \ppsi_i(\mathfrak{s}-\ppsi_i) \left((\mathfrak{s}-\psi_i)^2-\psi_i^2\right)\nonumber\\
& =  \ppsi_i(\mathfrak{s}-\ppsi_i)\mathfrak{s}(\mathfrak{s}-2\ppsi_i)
\end{align}
Now \begin{equation}\label{eq888}\|v(A_+)\|_{\ppsi}^2=(\mathfrak{s}-\ppsi_i)^2\ppsi_i +\ppsi_i^2(\mathfrak{s}-\ppsi_i)=\mathfrak{s}\ppsi_i(\mathfrak{s}-\ppsi_i)\end{equation}
It follows that the critical point on $S$ is 
$$v_i=\frac{v(A_+)}{\|v(A_+)\|_{\ppsi}}= 
\frac{(\mathfrak{s}-\ppsi_i) e_i-\ppsi_i(p-e_i)}{\sqrt{\mathfrak{s}\ppsi_i(\mathfrak{s}-\ppsi_i)}}
\editC =\frac{\mathfrak{s} e_i-\ppsi_ip}{\|\mathfrak{s} e_i-\ppsi_ip\|_{\ppsi}}$$
{\editD Thus
\begin{align*}
6c(v_i) & = 6c(v(A_+))/\|v(A_+)\|_{\ppsi}^{3}\\
& =  \ppsi_i(\mathfrak{s}-\ppsi_i)\mathfrak{s}(\mathfrak{s}-2\ppsi_i)/(\mathfrak{s}\ppsi_i(\mathfrak{s}-\ppsi_i))^{3/2} & {\rm using\ \ } (\ref{cvA+}), (\ref{eq888})\\
& =  (\mathfrak{s}-2\ppsi_i)/\sqrt{\mathfrak{s}\ppsi_i(\mathfrak{s}-\ppsi_i)}\\
& =  \frac{1}{\sqrt{\ppsi_i}}(1-2\ppsi_i/\mathfrak{s})/\sqrt{1-\ppsi_i/\mathfrak{s}}
\end{align*}}
{\editE If $c(v_i)<0$ then} $\psi_i>\mathfrak{s}/2$. Since $\mathfrak{s}=\sum\psi_i$, and
 all $\psi_i>0$, it follows
that $c(v_i)<0$ for at most one value of $i$. Thus $|K^+|\ge n-1$. We compute
\begin{align*}\langle \mathfrak{s} e_i-\ppsi_ip,\mathfrak{s} e_j-\ppsi_jp\rangle_{\psi}
&=\mathfrak{s}^2\langle e_i,e_j\rangle_{\psi}+\psi_i\psi_j\langle p,p\rangle_{\psi}-
\mathfrak{s}\left(\psi_j\langle e_i,p\rangle_{\psi}+\psi_i\langle p,e_j\rangle_{\psi}\right)\\
&=\delta_{ij}\mathfrak{s}^2\psi_i+\psi_i\psi_j\mathfrak{s}-2\mathfrak{s}\psi_i\psi_j\\
&=\mathfrak{s}\psi_i(\delta_{ij}\mathfrak{s}-\psi_j)\end{align*}
Using this gives
\begin{align*}
\alpha_{ij} =\langle v_i,v_j\rangle_{\psi} 
&=\langle \mathfrak{s} e_i-\ppsi_ip,\mathfrak{s} e_j-\ppsi_jp\rangle_{\psi}/(\|\mathfrak{s} e_i-\ppsi_ip\|_{\ppsi}\|\mathfrak{s} e_j-\ppsi_jp\|_{\ppsi})\\
&=(-\mathfrak{s}\ppsi_i\ppsi_j)/\sqrt{\mathfrak{s}\psi_i(\mathfrak{s}-\psi_i)\mathfrak{s}\psi_j(\mathfrak{s}-\psi_j)}\\
&=-\sqrt{\ppsi_i\ppsi_j/(\mathfrak{s}-\psi_i)(\mathfrak{s}-\psi_j)}\\
&<0\end{align*}
Then
\begin{align*}1-\alpha_{jk}/\alpha_{ij}\alpha_{ik}&=1+\sqrt{\frac{\psi_j\psi_k(\mathfrak{s}-\psi_i)(\mathfrak{s}-\psi_j)(\mathfrak{s}-\psi_i)(\mathfrak{s}-\psi_k)}
{(\mathfrak{s}-\psi_j)(\mathfrak{s}-\psi_k)\psi_i\psi_j\psi_i\psi_k}}\\
&=1+(\mathfrak{s}-\psi_i)/\psi_i\\
&=\mathfrak{s}/\psi_i
\end{align*}
Hence $\alpha_{ij}\alpha_{ik}/(\alpha_{ij}\alpha_{ik}-\alpha_{jk})=\psi_i/\mathfrak{s}$.

Let $\pi:\RR^n\to\mathfrak g$ be orthogonal projection. Since $\mathfrak g=p^{\perp}$ it follows that
 $$\pi(x)=x-\frac{\langle p,x\rangle_{\psi}}{\langle p,p\rangle_{\psi}}p$$
Using that $p=e_1+\cdots e_n$, and that the standard basis  $\{e_1,\cdots,e_n\}$ is orthogonal, gives $$\langle p,e_i\rangle_{\psi} =\langle e_i,e_i\rangle_{\psi}$$ and $\langle p,p\rangle_{\psi}=\mathfrak{s}$ so
$$\pi(e_i)=e_i-(\langle e_i,e_i\rangle_{\psi}/\mathfrak{s})p$$
Thus {\editE the local maxima are on the projections of the coordinate axes:} $$v(A_+)=\mathfrak{s}\cdot \pi(e_i)$$
This proves the lemma, {\editE modulo claim 1}.

{\bf Claim 2} At $v=v(A_+)$ then
 $\label{secondderiv} d^2(J|S)_v(w,w)=(\mathfrak{s}/2)\langle e_{_+}-e_{_-},w^2\rangle_{\psi}$  for $w\in T_vS$
 \gap
 
\noindent  Assuming this, the quadratic form $$Q(w,w)=\langle e_{_+}-e_{_-},w^2\rangle_{\psi}$$ is defined and non-singular on $\RR^n$
 and  $\mu(Q|T_vS)=\mu( d^2(J|S)_v)$.
 Let $L:\RR^n\to\RR^n$ be the linear map defined by
 $$L|A_{\pm}=\pm \Id$$
 Then
 \begin{equation}\label{Qeqtn}Q(x,y)=\langle Lx,y\rangle_{\psi}\end{equation}
 Now $p=e_{_+}+e_{_-}$ so $Lp=e_{_+}-e_{_-}$ and 
 $$Lv=L(\mathfrak{s}_{_+}e_{_+}+\mathfrak{s}_{_-}e_{_-})=\mathfrak{s}_{_+}e_{_+}-\mathfrak{s}_{_-}e_{_-}$$
  Now $T_vS$ is the orthogonal complement with respect
 to the inner product {\editA $\langle\cdot,\cdot\rangle_{\psi}$ of
 the subspace spanned by $\{ p,v \}$, because $S$ is a sphere in the orthogonal complement of $p$.
 Using (\ref{Qeqtn}) $T_vS$ is also the orthogonal complement with respect to $Q$ of 
 the subspace $W$ spanned by $\{  Lp,Lv \}$.}

{\bf Claim 3}  $Q|W$ is non-singular and  $\mu(Q|W)=1$. 
\gap

\noindent {\editE Assuming this,} since $W$ and $T_pV$ are orthogonal with respect to $Q$, it follows that $$\mu(Q)=\mu(Q|W)+\mu(Q|T_sV)=1+\mu(Q|T_vS)$$
 But $\mu(Q)=|A_+|$ so $\mu(Q|T_vS)=|A_+|-1$ which proves claim 1.
 
 \gap
  To prove claim 3 we first evaluate $Q(Lp,Lp)$, and $Q(Lp,Lv)$, and $Q(Lv,Lv)$ to obtain the matrix of $Q$
 in the basis $\{Lp,Lv\}$.
\begin{align*} Q(Lp,Lp) & = \langle L^2p,Lp\rangle_{\psi}\\
 &=\langle p,Lp\rangle_{\psi}\\
 & =  \langle e_{_+}+e_{_-},e_{_+}-e_{_-}\rangle_{\psi}\\
 & = -\mathfrak{s}_{_-}-\mathfrak{s}_{_+}
\end{align*}
\begin{align*}
 Q(Lv,Lv) & =  \langle v,Lv\rangle_{\psi}\\
 & =  \langle\mathfrak{s}_{_+}e_{_+}+\mathfrak{s}_{_-}e_{_-},\mathfrak{s}_{_+}e_{_+}-\mathfrak{s}_{_-}e_{_-}\rangle_{\psi}\\
 & =  \mathfrak{s}_{_+}^2\langle e_{_+},e_{_+}\rangle_{\psi} -\mathfrak{s}_-^2\langle e_{_-},e_{_-}\rangle_{\psi}\\
 & =  \mathfrak{s}_{_+}^2(-\mathfrak{s}_{_-})-\mathfrak{s}_{_-}^2(\mathfrak{s}_{_+})\\
 & =  -\mathfrak{s}_{_-}\mathfrak{s}_{_+}(\mathfrak{s}_{_+}+\mathfrak{s}_{_-})
\end{align*}
\begin{align*}
Q(Lp,Lv)&= \langle p,Lv\rangle_{\psi}\\
& =  \langle e_{_+}+e_{_-},\mathfrak{s}_{_+}e_{_+}-\mathfrak{s}_{_-}e_{_-}\rangle_{\psi}\\
& =  \mathfrak{s}_{_+}\langle e_{_+},e_{_+}\rangle_{\psi} -\mathfrak{s}_{_-}\langle e_{_-},e_{_-}\rangle_{\psi}\\
& =  \mathfrak{s}_{_+}(-\mathfrak{s}_{_-})-\mathfrak{s}_{_-}\mathfrak{s}_{_+}\\
& =  -2\mathfrak{s}_{_-}\mathfrak{s}_{_+}
\end{align*}
\begin{align*}
\therefore\ \det (Q|W)& = \det\left[\begin{array}{cc}
Q(Lp,Lp) & Q(Lp,Lv)\\
Q(Lp,Lv) & {Q(Lv,Lv)}
\end{array}\right] \\
& = \det\left[
\begin{array}{cc}
-(\mathfrak{s}_{_+}+\mathfrak{s}_{_-}) & -2\mathfrak{s}_{_+}\mathfrak{s}_{_-} \\
{-2\mathfrak{s}_{_+}\mathfrak{s}_{_-}} & -\mathfrak{s}_{_-}\mathfrak{s}_{_+}(\mathfrak{s}_{_+}+\mathfrak{s}_{_-})
\end{array}\right]\\
& =  \mathfrak{s}_{_+}\mathfrak{s}_{_-}\left[(\mathfrak{s}_{_+}+\mathfrak{s}_{_-})^2-4\mathfrak{s}_{_+}\mathfrak{s}_{_-}\right]\\
& =  \mathfrak{s}_{_+}\mathfrak{s}_{_-}(\mathfrak{s}_{_+}-\mathfrak{s}_{_-})^2\\
& =  \mathfrak{s}_{_+}\mathfrak{s}_{_-}\mathfrak{s}^2\\
& <  0
\end{align*}
Thus $Q|W$ has one eigenvalue of each sign, which proves claim 3. 

\gap

It only remain to prove claim 2. 
Give a critical point $v=v(A_+)$  of $f:=J|S$
 we compute $d^2(J|S)_v$. Let $\gamma:(-\epsilon,\epsilon)\to S$ be a smooth curve with $\gamma(0)=v$ and $\gamma'(0)=w\in T_vS$.
Then 
\begin{align*}
(f\circ\gamma)'(t)
&= \sum_{i=1}^n\left.\frac{\partial f}{\partial x_i}\right|_{x=\gamma(t)}\gamma'_i(t)\\
\therefore\ (f\circ\gamma)''(0)&=\sum_{i,j=1}^n\left.\frac{\partial^2f}{\partial x_i^2}\right|_{x=v}\gamma_i'(0)\gamma_j'(0)+\sum_{i=1}^n\left.\frac{\partial f}{\partial x_i}\right|_{x=v}\gamma''(0)\\
&=d^2J_v(w,w)+d J_v(\gamma''(0))
\end{align*}
In the following everything is evaluated at $t=0$
\begin{align*}
&\langle\gamma,\gamma\rangle_{\psi}=1\\
&\Rightarrow\langle\gamma',\gamma\rangle_{\psi}=0\\
&\Rightarrow\langle\gamma'',\gamma\rangle_{\psi} +\langle\gamma',\gamma'\rangle_{\psi}=0\\
&\Rightarrow\gamma''(0)\in \left(-\langle\gamma'(0),\gamma'(0)\rangle_{\psi}/\langle\gamma(0),\gamma(0)\rangle_{\psi}\right)\gamma(0) + T_vS
\end{align*}
Using $\gamma(0)=v$ and $\gamma'(0)=w$
$$\gamma''(0)=-\left(\langle w,w\rangle_{\psi}/\langle v,v\rangle_{\psi}\right)v+t$$
for some $t\in T_vS$.
Since $dJ_v$ vanishes on $T_vS$ we get
\begin{equation}\label{eq1}d^2f_v(w,w)=(f\circ\gamma)''(0)=d^2J_v(w,w)-\left(\langle w,w\rangle_{\psi}/\langle v,v\rangle_{\psi}\right)dJ_v(v)\end{equation}
Now we compute these two terms
\begin{align}\label{eq2}
d^2J_v(w,w)
& = \left.\frac{d^2}{dt^2}\right|_{t=0}\left(\frac{1}{2}\langle(v+tw)^2,p\rangle_{\psi}+\frac{1}{6}\langle(v+tw)^3,p\rangle_{\psi}\right)\nonumber \\
& =\langle w^2,p\rangle_{\psi} + \langle vw^2,p\rangle_{\psi}\nonumber \\
& =\langle w,w\rangle_{\psi} + \langle v,w^2\rangle_{\psi}
\end{align}
By (\ref{dJ})
$dJ_v(v)=\langle v,v\rangle_{\psi}+(1/2)\langle v^2,v\rangle_{\psi}$,
so
\begin{align}\label{eq33}\left(\langle w,w\rangle_{\psi}/\langle v,v\rangle_{\psi}\right)dJ_v(v)=\langle w,w\rangle_{\psi} +(1/2)\left(\langle v,v^2\rangle_{\psi}/\langle v,v\rangle_{\psi}\right)\langle w,w\rangle_{\psi}\end{align}
At the critical point $v=v(A_+)=\mathfrak{s}_{_+}e_{_+}+\mathfrak{s}_{_-}e_{_-}$ so
\begin{align*}
\langle v,v\rangle_{\psi} 
& =\langle \mathfrak{s}_{_+}e_{_+}+\mathfrak{s}_{_-}e_{_-},\mathfrak{s}_{_+}e_{_+}+\mathfrak{s}_{_-}e_{_-}\rangle_{\psi}\\
& =\mathfrak{s}_{_+}^2\langle e_{_+},e_{_+}\rangle_{\psi} +\mathfrak{s}_{_-}^2\langle e_{_-},e_{_-}\rangle_{\psi}\\
& =\mathfrak{s}_{_+}^2(-\mathfrak{s}_{_-}) + \mathfrak{s}_{_-}^2\mathfrak{s}_{_+}\\
& =\mathfrak{s}_{_+}\mathfrak{s}_{_-}(\mathfrak{s}_{_-}-\mathfrak{s}_{_+})\\
& =-\mathfrak{s}_{_+}\mathfrak{s}_{_-}\mathfrak{s}\\
\langle v,v^2\rangle_{\psi}
& =\langle \mathfrak{s}_{_+}e_{_+}+\mathfrak{s}_{_-}e_{_-},\mathfrak{s}_{_+}^2e_{_+}+\mathfrak{s}_{_-}^2e_{_-}\rangle_{\psi}\\
& = \mathfrak{s}_{_+}^3\langle e_{_+},e_{_+}\rangle_{\psi} +\mathfrak{s}_{_-}^3\langle e_{_-},e_{_-}\rangle_{\psi}\\
& = \mathfrak{s}_{_+}^3\mathfrak{s}_{_-}+\mathfrak{s}_{_-}^3\mathfrak{s}_{_+}\\
& =\mathfrak{s}_{_+}\mathfrak{s}_{_-}(\mathfrak{s}_{_-}-\mathfrak{s}_{_+})(\mathfrak{s}_{_-}+\mathfrak{s}_{_+})\\
& =-\mathfrak{s}_{_+}\mathfrak{s}_{_-}\mathfrak{s}(\mathfrak{s}_{_+}+\mathfrak{s}_{_-})
\end{align*}
Thus $\langle v,v^2\rangle_{\psi}/\langle v,v\rangle_{\psi}  = \mathfrak{s}_{_+}+\mathfrak{s}_{_-}$.
Using this with  (\ref{eq33}) gives
\begin{align*}\left(\langle w,w\rangle_{\psi}/\langle v,v\rangle_{\psi}\right)dJ_v(v)
=\langle w,w\rangle_{\psi} +(1/2)\left( \mathfrak{s}_{_+}+\mathfrak{s}_{_-}\right)\langle w,w\rangle_{\psi}\end{align*}
Using this and (\ref{eq2}) to substitute in to (\ref{eq1}) gives
\begin{align*}
d^2f_v(w,w) &= \langle w,w\rangle_{\psi} + \langle v,w^2\rangle_{\psi} -
( 1 +(1/2)(\mathfrak{s}_{_+}+\mathfrak{s}_{_-}))\langle w,w\rangle_{\psi}\\
& =  \langle v,w^2\rangle_{\psi} -
(1/2)\langle (\mathfrak{s}_{_+}+\mathfrak{s}_{_-})w,w\rangle_{\psi}\\
& =  \langle v,w^2\rangle_{\psi} -
(1/2)\langle (\mathfrak{s}_{_+}+\mathfrak{s}_{_-})p,w^2\rangle_{\psi}
\end{align*}
Recall $p=e_++e_-$ and $v=\mathfrak{s}_{_+} e_{_+}+\mathfrak{s}_{_-} e_{_-}$ from (\ref{vequation}). Then
\begin{align*}
d^2f_v(w,w)& =  \langle \mathfrak{s}_{_+}e_{_+}+\mathfrak{s}_{_-}e_{_-},w^2\rangle_{\psi} -
(1/2)\langle (\mathfrak{s}_{_+}+\mathfrak{s}_{_-})(e_++e_-),w^2\rangle_{\psi}\\
& =  \langle \mathfrak{s}_{_+}e_{_+}+\mathfrak{s}_{_-}e_{_-}-(1/2)(\mathfrak{s}_{_+}+\mathfrak{s}_{_-})(e_{_+}+e_{_-}),w^2\rangle_{\psi}\\
& =  (1/2)\mathfrak{s}\langle e_{_+}-e_{_-},w^2\rangle_{\psi}
\end{align*}
where we used $\mathfrak{s}=\mathfrak{s}_+-\mathfrak{s}_-$. This proves claim 2.\end{proof}

\begin{proof}[Proof of (\ref{localmaxlemmanotdiag})] In this lemma summation is over the set of integers in $[1,n-1]$ unless otherwise indicated, and $\langle\cdot,\cdot\rangle$ is
the standard inner product on $V=\RR^{n-1}$. A point $v=\sum v_ie_i\in S$ is a critical point of $c|S$ if and only if there is some $\alpha\in V$ such that for all
$w\in V$ we have $dc_v(w)=\alpha\cdot\langle v,w\rangle$, so
$$dc_v(w)=\sum\lambda_i v_i^2 w_i=\alpha \sum v_iw_i$$
This equation is satisfied if and only if $\forall i\ \lambda_iv_i^2=\alpha v_i$.   Since $\lambda_i\ge0$ the requirement that $c(v)=(1/3)\psi v^3>0$ implies $\psi v^2>0$ thus $\alpha\ne 0$.
Thus the set of positive critical points of $c|S$ is
 $$W=\{v\in S:\  \exists\ \alpha\ne 0\ \ \psi v^2=\alpha v\}$$ Given $v\in W$, let $A=\{i:v_i\ne 0\}$, then $A$ is not empty and $i\in A\Rightarrow \lambda_i\ne 0$ 
 and 
$$v=v(A)=\alpha\sum_{i\in A}\lambda_i^{-1}e_i$$
{\bf Claim}  $d^2(c|S)$ at $v=v(A)$ is the restriction to $T_vS$ of the quadratic form on $V$
$$Q(w)=\alpha\left(\sum_A w_i^2-\sum_{A^c}w_i^2\right)$$
where $A^c=\{1,\cdots,n-1\}\setminus A$. 

Observe that $Q$ is non-degenerate and
 $v=v(A)\in\langle e_i:i\in A\rangle$ and $T_vS=v^{\perp}$ so $Q|T_vS$ is also non-degenerate. 
 If follows  that $v(A)$ is a local maximum if $\alpha>0$ and $|A|=1$ or $\alpha<0$ and $A^c=\emptyset$. In the first
 case $A=\{i\}$ and $v(A)=e_i$ and $c(e_i)=\lambda_i>0.$ In the second case $A=\{1,\cdots,n-1\}$ and $v(A)=-(n-1)^{-1/2}\sum e_i$ so
 $c(v(A))<0$. This proves the lemma modulo the claim.
 
 To prove the claim, adapting the derivation of  (\ref{eq1}) gives
 \begin{equation}\label{eq5} d^2(c|S)_v(w,w)=d^2c_v(w,w)-(\langle w,w\rangle_{\psi}/\langle v,v\rangle_{\psi})dc_v(v)
 \end{equation}
Using $\lambda_iv_i=\alpha$ for $i\in A$ and $\lambda_i v_i=0$ for $i\notin A$ gives
 \begin{equation} d^2c_v(w,w)=2\sum\lambda_iv_iw_i^2 =2\sum_{i\in A} \alpha w_i^2\end{equation}
Using $\psi v^2=\alpha v$ and $v=v(A)=\alpha\sum_{i\in A} e_i$ gives
\begin{equation}
d_c(v)=\sum\lambda_iv_i^2v_i=\sum\alpha v_iv_i=\alpha\langle v,v\rangle_{\psi}
\end{equation}
 Hence
 \begin{equation}
 (\langle w,w\rangle_{\psi}/\langle v,v\rangle_{\psi})dc_v(v)=\alpha\sum w_i^2
 \end{equation}
 Substituting into (\ref{eq5})
  \begin{equation} d^2(c|S)_v(w,w)=2\sum_{i\in A} \alpha w_i^2-\alpha\sum w_i^2=  \alpha\left(\sum_{i\in A} w_i^2-\sum_{i\in A^c} w_i^2\right)
 \end{equation}
Which proves the claim.\end{proof}

  \small 
  \bibliography{refs.bib} 

\begin{thebibliography}{10}

\bibitem{Asch}
M.~Aschbacher.
\newblock Chevalley groups of type {$G_2$} as the group of a trilinear form.
\newblock {\em J. Algebra}, 109(1):193--259, 1987.

\bibitem{BCL}
S.~Ballas, D.~Cooper, and A.~Leitner.
\newblock A classification of generalized cusps in projective manifolds.
\newblock {\em Jour Topol}, pages 1455--1496, Dec. 2020.

\bibitem{Bieb}
L.~Bieberbach.
\newblock \"{U}ber die {B}ewegungsgruppen der {E}uklidischen {R}\"{a}ume.
\newblock {\em Math. Ann.}, 70(3):297--336, 1911.

\bibitem{Blaschke}
W.~Blaschke.
\newblock {\em Vorlesungen \"{u}ber {D}ifferentialgeometrie und geometrische
  {G}rundlagen von {E}insteins {R}elativit\"{a}tstheorie. {B}and {I}.
  {E}lementare {D}ifferentialgeometrie}.
\newblock Dover Publications, New York, N. Y., 1945.
\newblock 3d ed.

\bibitem{calabi2}
E.~Calabi.
\newblock Improper affine hyperspheres of convex type and a generalization of a
  theorem by {K}. {J}\"{o}rgens.
\newblock {\em Michigan Math. J.}, 5:105--126, 1958.

\bibitem{Cha}
L.~S. Charlap.
\newblock {\em Bieberbach groups and flat manifolds}.
\newblock Universitext. Springer-Verlag, New York, 1986.

\bibitem{MR437805}
S.~Y. Cheng and S.~T. Yau.
\newblock On the regularity of the {M}onge-{A}mp\`ere equation {${\rm
  det}(\partial ^{2}u/\partial x_{i}\partial sx_{j})=F(x,u)$}.
\newblock {\em Comm. Pure Appl. Math.}, 30(1):41--68, 1977.

\bibitem{choi1}
S.~{Choi}.
\newblock {The convex real projective manifolds and orbifolds with radial ends:
  the openness of deformations}.
\newblock {\em ArXiv e-prints}, Nov. 2010.

\bibitem{CLT2}
D.~Cooper, D.~Long, and S.~Tillmann.
\newblock Deforming convex projective manifolds.
\newblock {\em Geom. Topol.}, 22(3):1349--1404, 2018.

\bibitem{CLT1}
D.~Cooper, D.~D. Long, and S.~Tillmann.
\newblock On convex projective manifolds and cusps.
\newblock {\em Adv. Math.}, 277:181--251, 2015.

\bibitem{Dold}
A.~Dold.
\newblock Homology of symmetric products and other functors of complexes.
\newblock {\em Ann. of Math. (2)}, 68:54--80, 1958.

\bibitem{Gur}
S.~Garibaldi and R.~M. Guralnick.
\newblock Simple groups stabilizing polynomials.
\newblock {\em Forum Math. Pi}, 3:e3, 41, 2015.

\bibitem{Geoghegan}
R.~Geoghegan.
\newblock {\em Topological methods in group theory}, volume 243 of {\em
  Graduate Texts in Mathematics}.
\newblock Springer, New York, 2008.

\bibitem{hatcher}
A.~E. Hatcher.
\newblock Concordance spaces, higher simple-homotopy theory, and applications.
\newblock In {\em Algebraic and geometric topology ({P}roc. {S}ympos. {P}ure
  {M}ath., {S}tanford {U}niv., {S}tanford, {C}alif., 1976), {P}art 1}, Proc.
  Sympos. Pure Math., XXXII, pages 3--21. Amer. Math. Soc., Providence, R.I.,
  1978.

\bibitem{Hit}
N.~J. Hitchin.
\newblock Lie groups and {T}eichm\"{u}ller space.
\newblock {\em Topology}, 31(3):449--473, 1992.

\bibitem{Klartag}
B.~Klartag.
\newblock Affine hyperspheres of elliptic type.
\newblock {\em https://arxiv.org/abs/1508.00474}.

\bibitem{Lab1}
F.~Labourie.
\newblock Flat projective structures on surfaces and cubic holomorphic
  differentials.
\newblock {\em Pure Appl. Math. Q.}, 3(4, Special Issue: In honor of Grigory
  Margulis. Part 1):1057--1099, 2007.

\bibitem{Lang}
S.~Lang.
\newblock {\em Algebra}, volume 211 of {\em Graduate Texts in Mathematics}.
\newblock Springer-Verlag, New York, third edition, 2002.

\bibitem{Lof1}
J.~C. Loftin.
\newblock Affine spheres and convex {$\Bbb{RP}^n$}-manifolds.
\newblock {\em Amer. J. Math.}, 123(2):255--274, 2001.

\bibitem{Nom}
K.~Nomizu and T.~Sasaki.
\newblock {\em Affine differential geometry}, volume 111 of {\em Cambridge
  Tracts in Mathematics}.
\newblock Cambridge University Press, Cambridge, 1994.
\newblock Geometry of affine immersions.

\bibitem{Pog}
A.~V. Pogorelov.
\newblock On the improper convex affine hyperspheres.
\newblock {\em Geometriae Dedicata}, 1(1):33--46, 1972.

\bibitem{MR870934}
B.~Reichstein.
\newblock On expressing a cubic form as a sum of cubes of linear forms.
\newblock {\em Linear Algebra Appl.}, 86:91--122, 1987.

\bibitem{MR3105781}
B.~Reznick.
\newblock Some new canonical forms for polynomials.
\newblock {\em Pacific J. Math.}, 266(1):185--220, 2013.

\bibitem{That}
T.~T. That.
\newblock Lie group representations and harmonic polynomials of a matrix
  variable.
\newblock {\em Trans. Amer. Math. Soc.}, 216:1--46, 1976.

\end{thebibliography}


\begin{thebibliography}{99} %%%%%%%%%%%%%%%%%%%%%%%%%%%%%%%%%%%

\small



\bibitem{Asch} Aschbacher, Michael \textit{Chevalley Groups of Type $G_2$ as the Group of a Trilinear Form} J. Alg 109 (1987) 193--259

\bibitem{BCL} Ballas, S., Cooper, D., and Leitner, A. \textit{Generalized Cusps in Convex Projective Manifolds: Classification} https://arxiv.org/pdf/1710.03132.pdf

\bibitem{Bieb} Bieberbach, Ludwig; \textit{Uber die Bewegungsgruppen der Euklidischen Raume.} (German) Math. Ann. 70 (1911), no. 3, 297--336.

\bibitem{Cha} Charlap, Leonard S.
\textit{Bieberbach groups and flat manifolds. }
Universitext. Springer-Verlag, New York, 1986.

\bibitem{CLT} Cooper, D., Long, D., and Tillmann, S. \textit{Deforming Convex Projective Manifolds} Geom. Topol. 22 (2018), no. 3, 1349--1404

\bibitem{Geoghegan} Geoghegan, R. \textit{Topological Methods in Group Theory} Graduate Texts in Mathematics Vol 243. Springer, 2007. 

\bibitem{Lang} Lang, S. \textit{Algebra} Graduate Texts in Mathematics Vol 211. Springer, 2002. 

\bibitem{Gur}   Garibaldi, Skip, Guralnick, Robert M.,
   \textit{Simple groups stabilizing polynomials}
  Forum Math. Pi,
   3,
   (2015),
  3--41
  
  \bibitem{hatcher}  Hatcher,  A.E.,
    \textit {Concordance spaces, higher simple-homotopy theory, and
              applications},
  Proc. Sympos. Pure Math., XXXII},
     ,3--21, 1978.
 	
  
  \bibitem{Hit} N. J. Hitchin. \textit{Lie groups and Teichm\"uller space. Topology,} 31(3):449–473, 1992.
  
  \bibitem{Lab1} F. Labourie. \textit{ Flat projective structures on surfaces and cubic holomorphic dfferentials.} Pure and Applied Mathematics Quarterly, 3(4):1057--1099, 2007. 
  
  \bibitem{Lof1} J. C. Loftin. \textit{ Affine spheres and convex $\RPn$  manifolds.} American Journal of
Mathematics, 123(2):255--274, 2001.

\bibitem{Nom} Nomizu, K.  and Sasaki, T. \textit{Affine Differential Geometry} CUP 1994

 
 \end{thebibliography}
\bibliographystyle{abbrv}

\if0

\fi

\end{document}